\documentclass[11pt]{amsart}

\usepackage[title,titletoc,toc]{appendix}
\usepackage[T1]{fontenc}
\usepackage[applemac]{inputenc}
\usepackage{amsmath}
\usepackage{amssymb}
\usepackage{amsfonts}
\usepackage[english]{babel}
\usepackage{amsthm}
\usepackage{enumerate}
\usepackage{graphicx}
\usepackage{caption}
\usepackage{subcaption}

\usepackage{pdfsync}

\usepackage{graphicx,color}

\usepackage[all]{xy}

\usepackage[top=3cm, bottom=2cm, left=3cm, right=3cm]{geometry}

\newtheorem{theorem}{Theorem}
\newtheorem{definition}[theorem]{Definition}
\newtheorem{proposition}[theorem]{Proposition}
\newtheorem{lemma}[theorem]{Lemma}
\newtheorem{corollary}[theorem]{Corollary}
\newtheorem{remark}[theorem]{Remark}

\numberwithin{equation}{section}
\numberwithin{theorem}{section}

\newcommand{\rd}{\mathrm{d}}

\newcommand{\RR}{\mathbb{R}}
\newcommand{\N}{\mathbb{N}}
\newcommand{\C}{\mathbb{C}}
\newcommand{\Z}{\mathbb{Z}}

\newcommand{\eps}{\varepsilon}

\newcommand{\supp}{{\rm supp\ }}

\newcommand{\resc}{{\bf resc}}

\newcommand{\bW}{{\bf W}}

\newcommand{\mU}{{\mathcal U}}
\newcommand{\mW}{{\mathcal W}}
\newcommand{\mV}{{\mathcal V}}
\newcommand{\mF}{{\mathcal F}}

\newcommand{\mFr}{{\mathcal F_{k,\resc}}}
\newcommand{\mG}{{\mathcal G}}

\newcommand{\mY}{{\mathcal Y}}
\newcommand{\mT}{{\mathcal T}}
\newcommand{\mTr}{{\mathcal{T}_{k, \resc}}}
\newcommand{\mX}{{\mathcal X}}

\title{Equilibria of homogeneous functionals in the fair-competition regime}

\author{V. Calvez, J. A. Carrillo, F. Hoffmann}

\begin{document}

\maketitle

\begin{abstract}
We consider macroscopic descriptions of particles where repulsion is modelled by non-linear power-law diffusion
and attraction by a homogeneous singular/smooth kernel leading to variants of the Keller-Segel model of chemotaxis. 
We analyse the regime in which both homogeneities scale the same with respect to dilations, that we coin as fair-competition. In the singular kernel case, we
show that existence of global equilibria can only happen at a certain critical value and they are characterised as optimisers of 
a variant of HLS inequalities. We also study the existence of self-similar solutions for the sub-critical case, or equivalently of optimisers
of rescaled free energies. These optimisers are shown to be compactly supported radially symmetric and non-increasing stationary solutions
of the non-linear Keller-Segel equation. On the other hand, we show that no radially symmetric non-increasing stationary solutions exist in the smooth kernel case, implying that there is no criticality. However, we show the existence of positive self-similar solutions for all values of the parameter under the condition that diffusion is not too fast.
We finally illustrate some of the open problems in the smooth kernel case by numerical experiments.
\end{abstract}

%\tableofcontents

%%%%%%%%%%%%%%%%%%%%%%%%%%%%
%%%%%%%%%%%%%%%%%%%%%%%%%%%%
%%%%%%%%%%%%%%%%%%%%%%%%%%%%
%%%%%%%%%%%%%%%%%%%%%%%%%%%%
\section{Introduction}
\label{sec:Introduction}
%%%%%%%%%%%%%%%%%%%%%%%%%%%%
%%%%%%%%%%%%%%%%%%%%%%%%%%%%
%%%%%%%%%%%%%%%%%%%%%%%%%%%%
%%%%%%%%%%%%%%%%%%%%%%%%%%%%

The goal of this work is to investigate properties of the following class of homogeneous
functionals, defined for centered probability densities $\rho(x)$,
belonging to suitable $L^p$-spaces, and some interaction
strength coefficient $\chi> 0$ and diffusion power $m>0$:
\begin{align}
\mathcal F_{m,k}[\rho] &=  \int_{\RR^N}   U_m\left(\rho(x)\right)
\, dx + \chi  \iint_{\RR^N\times\RR^N} \rho(x)  W_k(x-y) \rho(y)\,
dxdy \nonumber\\
&:= \mathcal U_m[\rho] + \chi \mathcal W_k[\rho] \, ,
\label{eq:functional}\\
&\rho(x)\geq 0\, , \quad \int_{\RR^N}\rho(x)\, dx = 1\, , \quad
\int_{\RR^N} x\rho(x)\, dx = 0\,  , \nonumber
\end{align}
with
\vspace{-0.2cm}
\begin{equation*}
U_m(\rho) = \left\{\begin{array}{ll} \dfrac1{N(m-1)}\rho^m\, ,
&\mbox{if}\quad m\neq 1 \medskip\\ \dfrac 1N\rho\log \rho \, ,
&\mbox{if}\quad  m=1  \end{array}\right.\,, %\label{eq:ass1}
\end{equation*}
and
\vspace{-0.2cm}
\begin{equation*}
W_k(x) = \left\{\begin{array}{ll} \dfrac{|x|^k}k\, , &
\mbox{if}\quad k\in(-N,N)\setminus\{0\}
\medskip\\ \log |x|\, ,  & \mbox{if}\quad  k = 0
\end{array}\right.\, . %\label{eq:ass2}
\end{equation*}
The conditions on $k$ imply that the kernel $W_k(x)$ is locally
integrable in $\RR^N$. The center of mass is assumed
to be zero since the free energy functional is invariant by
translation.

There exists a strong link between the
aforementioned functional \eqref{eq:functional} and the following family
of partial differential equations modelling self-attracting
diffusive particles at the macroscopic scale,
\begin{equation}
\left\{
\begin{array}{l}
\partial_t \rho  =  \dfrac 1N \Delta\rho ^m +
2\chi \nabla\cdot \left( \rho \nabla S_k\right)\,
, \quad t>0\, , \quad x\in \RR^N\, , \smallskip\\
\displaystyle \rho(t=0,x) = \rho_0(x)\geq 0\, ,
\quad \int_{\RR^N}\rho_0(x)\, dx = 1\, ,
\quad \int_{\RR^N}x\rho_0(x)\, dx = 0\, ,
\end{array}
\right. \label{eq:KS}
\end{equation}
where we define the mean-field potential $S_k(x) := W_k(x)*\rho(x)$. For $k>1-N$, the gradient $\nabla S_k:= \nabla \left(W_k \ast \rho\right)$ is well defined.
For $-N<k\leq 1-N$ however, it becomes a singular integral, and we thus define it via a Cauchy principal value. Hence, the mean-field potential gradient in equation \eqref{eq:KS} is given by
\begin{equation}\label{gradS}
 \nabla S_k(x) :=
 \begin{cases}
  \nabla W_k \ast \rho\, ,
  &\text{if} \, \, k>1-N\, , \\[2mm]
  \displaystyle\int_{\RR^N} \nabla W_k (x-y)\left(\rho(y)-\rho(x)\right)\, dy\, ,
  &\text{if} \, \, -N<k\leq 1-N\, . 
 \end{cases}
\end{equation}
The noticeable characteristic of the class of PDEs \eqref{eq:KS} and the functional $\mF_{m,k}$ consists in the competition between the diffusion (possibly non-linear), and the non-local, quadratic non-linearity which is due to the
self-attraction of the particles through the mean-field potential $S_k$. The parameter $\chi>0$ measures the strength of the interaction and scales with the mass of solution densities.

The strong connection between the functional $\mathcal F_{m,k}$ and the
PDE \eqref{eq:KS} is due to the fact that the functional $\mathcal F_{m,k}$ is non-increasing along the
trajectories of the system. Namely $\mathcal F_{m,k}$ is the free
energy of the system and it satisfies at least formally
\begin{equation*}\label{eq:entropydissipation}
\dfrac{d}{dt}\mathcal{F}_{m,k}[\rho(t)] = -\int_{\RR^N}
\rho(t,x)\left| \nabla \left( \dfrac m{N(m-1)}\rho(t,x)^{m-1} +
2\chi W_k(x)*\rho(t,x) \right) \right|^2\,dx\,  .
\end{equation*}
Furthermore, the system \eqref{eq:KS} is the formal gradient flow
of the free energy functional \eqref{eq:functional} when the
space of probability measures is endowed with the Euclidean
Wasserstein metric $\bW$. This means that the family of PDEs \eqref{eq:KS}
can be written as
\begin{equation*}
\partial_t\rho(t)
=\nabla \cdot \left( \rho(t)\, \nabla \mathcal T_{m,k} [\rho(t)]\right)
=- \nabla_{\bf W} \mathcal F_{m,k}[\rho(t)]\, ,
%\label{eq:gradient flow}
\end{equation*}
where $\mathcal T_{m,k} [\rho]$ denotes the first variation of the energy functional in the set of probability densities:
\begin{equation}\label{eq:1stvar}
 \mathcal T_{m,k} [\rho](x) :=\frac{\delta \mF_{m,k}}{\delta \rho}[\rho](x) = \frac{m}{N(m-1)}\rho^{m-1}(x)+ 2 \chi W_k(x) \ast \rho(x)\, .
\end{equation}
This illuminating statement has been clarified in the seminal
paper by Otto \cite{Otto} for the porous medium equation, and generalized to a large family of equations
subsequently in \cite{CaMcCVi03,AmGiSa05,CaMcCVi06}, we refer to \cite{Villani03,AmGiSa05}
for a comprehensive presentation of this theory of gradient flows
in Wasserstein metric spaces, particularly in the convex case. Let us mention that such a gradient flow can be constructed as the
limit of discrete in time steepest descent schemes. Performing gradient flows of a convex functional is a
natural task, and suitable estimates from below on the right notion of Hessian of
$\mF_{m,k}$ translate into a rate of convergence
towards equilibrium for the PDE \cite{Villani03,CaMcCVi03,AmGiSa05}.
However, performing gradient flows of non-convex functionals
is  much more delicate, and one has to seek compensations. Such
compensations do exist in our case, and we will observe them at the level
of existence of minimisers for the free energy functional $\mF_{m,k}$ and stationary states
of the family of PDEs \eqref{eq:KS} in particular regimes.

The family of non-local problems \eqref{eq:KS} has been
intensively studied in various contexts arising in physics and
biology. The two-dimensional logarithmic case $(m = 1,k=0)$ is the
so-called Keller-Segel system in its simplest formulation
\cite{KeSe70,KeSe71a,Nanjundiah73,JaLu92,DoPe04,BlaDoPe06,Perthame06}. It
has been proposed as a model for chemotaxis in cell populations.
Cells may interact with each other by secreting a chemical substance to attract cells around them. This occurs for instance during the
starvation stage of the slime mold {\em Dyctiostelium discoideum}.
More generally, chemotaxis is widely observed in various biological
fields (morphogenesis, bacterial self-organization, inflammatory
processes among others). The two- and three-dimensional configurations with 
Newtonian interaction $(m = 1,k = 2-N)$ are the so-called Smoluchowski-Poisson system arising in
gravitational physics. It describes macroscopically a
density of particles subject to a self-sustained gravitational
field \cite{CLM,CM}.

Let us describe in more detail the two-dimensional Keller-Segel
system as the analysis of its peculiar structure will serve as a guideline to understand the other cases. 
In fact, the functional \eqref{eq:functional} $(m = 1,k = 0)$ is bounded from below if and only if $\chi = 1$. The gradient flow is also subject to a
remarkable dichotomy, well described mathematically. The density
exists globally in time if $\chi<1$ (diffusion overcomes self-attraction), whereas blow-up occurs in finite time when $\chi>1$ (self-attraction overwhelms diffusion). This transition has been first formulated in \cite{ChiPe81}. Mathematical contributions are \cite{JaLu92} for the existence part,
\cite{Nagai95} for the radial case, and \cite{DoPe04,BlaDoPe06} in the full space. The critical case $\chi=1$ was analyzed further in \cite{BlaCaMa07,BCC12,CF} in terms of stability of stationary states.

The effect of substituting linear diffusion by non-linear diffusion with $m>1$ in 
two dimensions and higher was described in \cite{CaCa06,Sugi1} where it is shown
that solutions exists globally in time for all values of the parameter $\chi>0$. The role of both non-linear diffusion and non-local agreggration terms was clarified in \cite{BCL}, see also \cite{Sugi2}, where the authors find that there is a similar dichotomy to the two-dimensional classical Keller-Segel case $(N=2,m = 1,k = 0)$, for a whole range of parameters, choosing the non-local term as the Newtonian potential, $(N\geq 3,m = 2- 2/N,k = 2-N)$. The main difference is that the stationary states found for the critical case are compactly supported. Choosing the non-local term as the Newtonian potential, this range of parameters can be understood as fixing the non-linear diffusion such that both terms in the functional $\mF_{m,k}$ scale equally for mass-preserving dilations. This mass-preserving dilation homogeneity of the functional $\mF_{m,k}$ is shared by the range of parameters $(m,k)$ with $N(m-1)+k=0$ for all  dimensions, $m>0$ and $k\in (-N,N)$. We call this range of parameters the fair-competition regime, since both terms are competing each other at equal foot. 

We will analyse the properties of the functional $\mF_{m.k}$ in relation to global minimisers and its relation to stationary
states of \eqref{eq:KS} in this work. We will first define properly the notion of stationary states to \eqref{eq:KS} and analyse
their basic properties in Section 2. We will also state and explain the main results of this work once the different regimes have
been introduced. We postpone further discussion of the related literature to Section 2. Section 3 is devoted to the fair-competition
regime with $k<0$ for which we show a similar dichotomy to \cite{BCL} in the whole range $k\in (-N,0)$ including the most singular
cases. We show that stationary states exist only for a critical value of $\chi$ and that they are compactly supported, bounded, radially symmetric decreasing and 
continuous functions. Moreover, we show that they are global minimisers of $\mF_{m.k}$. The sub-critical case is also analysed in
scaled variables and we show the existence of global minimisers with the properties above leading to the existence of self-similar
solutions in original variables. The critical parameter is characterised by a variant of HLS inequalities as in \cite{BCL}. Let us mention
that the regularity results need a careful treatment of the problem in radial coordinates involving non-trivial properties of hypergeometric
functions. The properties of the kernel in radial coordinates are postponed to the Appendix A.

In Section 4, we analyse the case $k>0$. Let us mention that there are no results in the literature to our knowledge concerning 
the case $k\in(0,N)$ in which $0<m=1-k/N<1$. There is one related result in \cite{CL} for the limiting case in one dimension taking $m=0$, corresponding to logarithmic diffusion, and $k=1$.
They showed that no criticality is present in that case as solutions to \eqref{eq:KS} with $(m = 0,k = 1)$ are globally defined in time for all values of the parameter $\chi>0$. 
We show that no radially symmetric non-increasing stationary states and no radially symmetric non-increasing global minimisers exist in original variables for all values of the critical parameter $\chi$ and for $k\in(0,N)$ while we show the existence of stationary states for all values of the 
critical parameter $\chi$ in scaled variables for $k\in (0,1]$. In this sense, we show that there is no criticality for $k>0$. 
A full proof of non-criticality involves the analysis of the minimisation problem in scaled variables as for $k<0$ showing that global minimisers exist 
in the right functional spaces for all values of the critical parameter and that they are indeed stationary states.  This is proven in one dimension in \cite{CCHCetraro} 
by optimal transport techniques and postponed for further future investigations in general dimension. We finally illustrate these results by numerical experiments 
in one dimension corroborating the absence of critical behavior for $k>0$.

%%%%%%%%%%%%%%%%%%%%%%%%%%%%
%%%%%%%%%%%%%%%%%%%%%%%%%%%%
%%%%%%%%%%%%%%%%%%%%%%%%%%%%
%%%%%%%%%%%%%%%%%%%%%%%%%%%%
\section{Stationary States \& Main results}
\label{sec:Preliminaries}
%%%%%%%%%%%%%%%%%%%%%%%%%%%%
%%%%%%%%%%%%%%%%%%%%%%%%%%%%
%%%%%%%%%%%%%%%%%%%%%%%%%%%%
%%%%%%%%%%%%%%%%%%%%%%%%%%%%

%%%%%%%%%%%%%%%%%%%%%%%%
%%%%%%%%%%%%%%%%%%%%%%%%
%%%%%%%%%%%%%%%%%%%%%%%%
%%%%%%%%%%%%%%%%%%%%%%%%
\subsection{Stationary States: Definition \& Basic Properties}
\label{sec:sstates}
%%%%%%%%%%%%%%%%%%%%%%%
%%%%%%%%%%%%%%%%%%%%%%%
%%%%%%%%%%%%%%%%%%%%%%%
%%%%%%%%%%%%%%%%%%%%%%%

Let us define precisely the notion of stationary states to the aggregation equation \eqref{eq:KS}.

\begin{definition}\label{def:sstates}
 Given $\bar \rho \in L_+^1\left(\RR^N\right) \cap L^\infty\left(\RR^N\right)$ with $||\bar \rho||_1=1$, it is a \textbf{stationary state} for the evolution equation \eqref{eq:KS} if $\bar \rho^{m} \in \mathcal{W}_{loc}^{1,2}\left(\RR^N\right)$, $\nabla \bar S_k\in L^1_{loc}\left(\RR^N\right)$, and it satisfies
 \begin{equation}\label{eq:steady}
  \frac{1}{N}\nabla \bar \rho^m=-2 \chi \, \bar \rho \nabla \bar S_k
 \end{equation}
in the sense of distributions in $\RR^N$. 
%If $k>0$, we further require $|x|^k\bar \rho \in L^1\left(\RR^N\right)$. 
If $-N<k\leq 1-N$, we further require $\bar \rho \in C^{0,\alpha}\left(\RR^N\right)$ with $\alpha \in (1-k-N,1)$.
\end{definition}

We start by showing that the function $ S_k$ and its gradient defined in \eqref{gradS} satisfy even more than the regularity $\nabla  S_k\in L^1_{loc}\left(\RR^N\right)$ required in Definition \ref{def:sstates}.

\begin{lemma}\label{lem:regS} Let $\rho \in  L_+^1\left(\RR^N\right) \cap L^\infty\left(\RR^N\right)$ with $||\rho||_1=1$. If $0<k<N$, we additionally assume $|x|^k \rho \in L^1\left(\RR^N\right)$. Then the following regularity properties hold:
\begin{itemize}
\item[i)] $S_k \in L^\infty_{loc}\left(\RR^N\right)$ for $0<k<N$ and $ S_k \in L^{\infty}\left(\RR^N\right)$ for $-N<k<0$.
\item[ii)] $\nabla  S_k \in L^\infty_{loc}\left(\RR^N\right)$ for $1<k<N$ and $\nabla  S_k \in L^{\infty}\left(\RR^N\right)$ for $-N<k<0$ and $0<k\leq 1$, assuming additionally $\rho \in C^{0,\alpha}\left(\RR^N\right)$ with $\alpha \in (1-k-N,1)$ in the range $-N<k\leq 1-N$. 
\end{itemize}
\end{lemma}

\begin{proof}
A direct decomposition in near- and far-field sets $\mathcal{A}:=\left\{y:|x-y|\leq 1\right\}$ and $\mathcal{B}:=\RR^N-\mathcal{A}$ yields for $-N<k<0$ and $x \in \RR^N$,
\begin{align*}
 | S_k(x)| 
 &\leq \int_{\RR^N} |W_k(x-y)|  \rho(y)\, dy
 \leq \frac{1}{|k|} \int_{\mathcal{A}} |x-y|^{k}  \rho(y) \, dy + \frac{1}{|k|}\int_{\mathcal{B}}  \rho(y)\, dy\notag\\
 &\leq \frac{1}{|k|}\left(\frac{\sigma_N}{k+N} || \rho||_\infty + 1\right)<\infty\, ,%\label{eq:Sbound}
\end{align*}
where $\sigma_N=2 \pi^{(N/2)}/\Gamma(N/2)$ denotes the surface area of the $N$-dimensional unit ball. 
Locally, boundedness extends to the fast diffusion regime $0<k<N$ by using the inequality
\begin{equation}\label{eq:kconvex}
 |x-y|^k\leq \eta \left(|x|^k+|y|^k\right)\, ,
 \qquad \eta=\max\{1,2^{k-1}\}\, .
\end{equation}
This inequality follows directly from splitting into cases $k<1$ and $k>1$. The inequality 
$|x-y|^k \leq |x|^k+|y|^k$
is true for any $k \in (0,1]$ with $x,y \in \RR^N$ by direct inspection. For $N>1$ and $k \in (1,N)$, we have by convexity
$
|x-y|^k \leq 2^{k-1} \left(|x|^k+|y|^k\right),
$
for any $x,y \in \RR^N$, and so \eqref{eq:kconvex} holds true.

Similarly, in order to prove ii) for $1-N<k\leq 1$ and $x \in \RR^N$, we estimate $\nabla  S_k$ as
\begin{align*}
 |\nabla  S_k(x)| 
 &\leq \int_{\RR^N} |\nabla W_k(x-y)|  \rho(y)\, dy
 \leq \int_{\mathcal{A}} |x-y|^{k-1}  \rho(y) \, dy + \int_{\mathcal{B}}  \rho(y)\, dy \notag\\
 &\leq \left(\frac{\sigma_N}{k+N-1} || \rho||_\infty + 1\right)<\infty\, .
\end{align*}
In the Cauchy integral range $-N<k\leq 1-N$, we additionally require a certain H\"{o}lder regularity, yielding 
\begin{align*}
 |\nabla S_k(x)| 
 &= \left|\int_{\mathcal{A}} \nabla W_k(x-y) \left(\rho(y)-\rho(x)\right)\, dy + \int_{\mathcal{B}} \nabla W_k(x-y) \left(\rho(y)-\rho(x)\right)\, dy\right|\\
 &\leq \int_{\mathcal{A}} |\nabla W_k(x-y)| \left| \rho(y)-\rho(x)\right|\, dy
 + \int_{\mathcal{B}} |\nabla W_k(x-y)| \rho(y)\, dy\\
 &\leq \int_{\mathcal{A}} |x-y|^{k-1}|x-y|^{\alpha} \, dy +\int_{\mathcal{B}}  \rho(y)\, dy <\infty\, , \notag
\end{align*}
where the term $\int_{\mathcal{B}}\nabla W_k(x-y) \, dy$ vanishes by anti-symmetry. For $1<k<N$ and $x$ in some compact set, we have
 \begin{align}\label{eq:gradSboundloc}
 |\nabla  S_k (x)| 
 &\leq \int_{\mathcal{A}} |x-y|^{k-1} \rho(y) \, dy + \int_{\mathcal{B}} |x-y|^{k-1}  \rho(y)\, dy \notag\\
 &\leq \frac{\sigma_N}{k+N-1} || \rho||_\infty +\int_{\mathcal{B}} |x-y|^{k}  \rho(y)\, dy \notag
\end{align}
which concludes $\nabla S_k \in L^\infty_{loc}\left(\RR^N\right)$ using \eqref{eq:kconvex} and the fact that the $k$th moment of $ \rho$ is bounded.
\end{proof}

We will prove that for certain cases there are no stationary states to \eqref{eq:KS} in the sense of Definition \ref{def:sstates}, for instance for the sub-critical classical Keller-Segel model in two dimensions \cite{BlaDoPe06}. However, the scale invariance of \eqref{eq:KS} motivates us to look for self-similar solutions instead. To this end, we rescale equation \eqref{eq:KS} to a non-linear Fokker-Planck type equation as in \cite{CaTo00}. Let us define 
$$
u(t,x):= \alpha^N(t)\rho\left( \beta(t), \alpha(t) x\right),
$$
where $\rho(t,x)$ solves \eqref{eq:KS} and the functions $\alpha(t)$, $\beta(t)$ are to be determined. If we assume $u(0,x)=\rho(0,x)$, then $u(t,x)$ satisfies the rescaled drift-diffusion equation
\begin{equation}
\left\{
\begin{array}{l}
\partial_t u =  \frac{1}{N} \Delta u ^m +
2\chi \nabla\cdot \left( u \nabla S_k \right)+ \nabla\cdot \left(x u\right)\,
, \quad t>0\, , \quad x\in \RR^N\, , \smallskip\\
\displaystyle u(t=0,x) = \rho_0(x)\geq 0\, ,
\quad \int_{-\infty}^\infty \rho_0(x)\, dx = 1\, ,
\quad \int_{-\infty}^\infty x\rho_0(x)\, dx = 0\, ,
\end{array}
\right. \label{eq:KSresc}
\end{equation}
for the choices
\begin{equation*}%\label{scaling}
\alpha(t)= e^t, \quad \beta(t)=
\begin{cases}
 \frac{1}{2-k} \left(e^{(2-k)t} -1\right), &\text{if} \, k\neq2,\\
 t, &\text{if} \, k=2,
\end{cases}
\end{equation*}
and with $\nabla S_k$ given by \eqref{gradS} with $u$ instead of $\rho$. By differentiating the centre of mass of $u$, we see easily that 
\begin{equation*}\label{COM}
 \int_{\RR^N} x u(t,x)\, dx = e^{-t} \int_{\RR^N} x \rho_0(x)\, dx=0\, , \qquad \forall t>0\, ,
\end{equation*}
and so the initial zero centre of mass is preserved for all times. Self-similar solutions to \eqref{eq:KS} now correspond to stationary solutions of \eqref{eq:KSresc}. Similar to Definition \ref{def:sstates}, we state what we exactly mean by stationary states to the aggregation equation \eqref{eq:KSresc}.

%%%%%%%%%%%%%%%%%%%%%%%%%%%%%%%%%%%%%%%%%%%%%%%%%%%%%%%
%%%%%%%%%%%%%%%%%%%%%%%%%%%%%%%%%%%%%%%%%%%%%%%%%%%%%%%

\begin{definition}\label{def:sstates resc}
Given $\bar u \in L_+^1\left(\RR^N\right) \cap L^\infty\left(\RR^N\right)$ with $||\bar u||_1=1$, it is a \textbf{stationary state} for the evolution equation \eqref{eq:KSresc} if $\bar u^{m} \in \mathcal{W}_{loc}^{1,2}\left(\RR^N\right)$, $\nabla \bar S_k\in L^1_{loc}\left(\RR^N\right)$, and it satisfies
 \begin{equation}\label{eq:steady resc}
  \frac{1}{N}\nabla \bar u^m=-2 \chi \, \bar u \nabla \bar S_k - x\, \bar u
 \end{equation}
in the sense of distributions in $\RR^N$. 
% If $k>0$, we further require $|x|^k\bar u \in L^1\left(\RR^N\right)$.
If $-N<k\leq 1-N$, we further require $\bar u \in C^{0,\alpha}\left(\RR^N\right)$ with $\alpha \in (1-k-N,1)$.
\end{definition}
From now on, we switch notation from $u$ to $\rho$ for simplicity, it should be clear from the context if we are in original or rescaled variables. In fact, stationary states as defined above have even more regularity:
\begin{lemma}\label{lem:sstatesreg}
 Let $k \in (-N,N)\backslash\{0\}$ and $\chi>0$. 
 \begin{enumerate}[(i)]
  \item If $\bar \rho$ is a stationary state of equation \eqref{eq:KS} with $|x|^k\bar\rho\in L^1\left(\RR^N\right)$ in the case $0<k<N$, then $\bar \rho$ is continuous on $\RR^N$.
  \item If $\bar \rho_\resc$ is a stationary state of equation \eqref{eq:KSresc} with $|x|^k\bar\rho_\resc\in L^1\left(\RR^N\right)$ in the case $0<k<N$, then $\bar \rho_\resc$ is continuous on $\RR^N$. 
 \end{enumerate}
\end{lemma}
\begin{proof}
\begin{enumerate}[(i)]
 \item First, note that  $\nabla \bar S_k \in L_{loc}^\infty\left(\RR^N\right)$ by Lemma \ref{lem:regS}, and therefore, $\bar \rho \nabla \bar S_k \in L^1_{loc}\left(\RR^N\right) \cap L^\infty_{loc}\left(\RR^N\right)$. Hence, we get by interpolation that $\nabla \cdot \left(\bar \rho \nabla \bar S_k\right) \in \mathcal{W}_{loc}^{-1,p}\left(\RR^N\right)$ for all $1<p < \infty$. Recall from Definition \ref{def:sstates} that 
 $\bar \rho^m$ is a weak $\mathcal{W}_{loc}^{1,2}\left(\RR^N\right)$ solution of
 \begin{equation*}\label{eq:steady2}
  \frac{1}{N}\Delta \bar \rho^m = -2 \chi \, \nabla \cdot \left(\bar \rho \nabla \bar S_k\right)
 \end{equation*}
in $\RR^N$, and so $\bar \rho^m$ is in fact a weak solution in $\mathcal{W}_{loc}^{1,p}\left(\RR^N\right)$ for all $1< p< \infty$ by classic elliptic regularity. Using Morrey's inequality, we deduce that $\bar \rho^m$ belongs to the H\"{o}lder space $C_{loc}^{0,\alpha}\left(\RR^N\right)$ with $\alpha=(p-N)/N$ for any $N< p< \infty$, and thus $\bar \rho^m \in C\left(\RR^N\right)$. Hence, $\bar \rho$ itself is continuous as claimed.
\item Since $x\bar \rho_\resc \in L^1_{loc}\left(\RR^N\right) \cap L^\infty_{loc}\left(\RR^N\right)$, we obtain again by interpolation $\nabla \cdot \left(x\bar \rho_\resc\right) \in \mathcal{W}_{loc}^{-1,p}\left(\RR^N\right)$ for all $1< p< \infty$. By Definition \ref{def:sstates resc}, $\bar \rho_\resc^m$ is a weak $\mathcal{W}_{loc}^{1,2}\left(\RR^N\right)$ solution of
 \begin{equation*}
  \frac{1}{N}\Delta \bar \rho_\resc^m = -2 \chi \, \nabla \cdot \left(\bar \rho_\resc \nabla \bar S_k\right)-\nabla \cdot \left(x\bar \rho_\resc\right) 
 \end{equation*}
 in $\RR^N$, and so $\bar \rho_\resc^m$ is again a weak solution in $\mathcal{W}_{loc}^{1,p}\left(\RR^N\right)$ for all $1< p< \infty$ by classic elliptic regularity. We conclude as in original variables.
\end{enumerate}
\end{proof}

In the case $k<0$, we furthermore have a non-linear algebraic equation for stationary states:

\begin{corollary}[Necessary Condition for Stationary States] \label{cor:EL}
Let $k \in (-N,0)$ and $\chi>0$.
\begin{enumerate}[(i)]
%%%%%%%%%%%%%%%%%%%%%%%%%%%%%%%%%%%%%%%%%%
 \item If $\bar \rho$ is a stationary state of equation \eqref{eq:KS}, then $\bar \rho \in \mathcal{W}^{1,\infty}\left(\RR^N\right)$ and it satisfies 
  \begin{equation}\label{eq:EL}
 \bar \rho(x)^{m-1} = \frac{N(m-1)}{m} \left( C_k[\bar\rho](x)-2 \chi \, \bar S_k(x)\right)_+\, , \qquad
 \forall \, x \in \RR^N\, ,
\end{equation}
 where $C_k[\bar\rho](x)$ is constant on each connected component of $\supp(\bar \rho)$. 
 %%%%%%%%%%%%%%%%%%%%%%%%%%%%%%%%%%%%%%%%%%%%
\item If $\bar \rho_{\resc}$ is a stationary state of equation \eqref{eq:KSresc}, then $\bar \rho_\resc \in \mathcal{W}^{1,\infty}_{loc}\left(\RR^N\right)$ and it satisfies 
 \begin{equation}\label{eq:ELresc}
 \bar \rho_\resc(x)^{m-1} = \frac{N(m-1)}{m} \left( C_{k,\resc}[\bar\rho](x)-2 \chi \, \bar S_k(x)-\frac{|x|^2}{2}\right)_+\, , \qquad
 \forall \, x \in \RR^N\, ,
\end{equation}
where $C_{k,\resc}[\bar\rho](x)$ is constant on each connected component of $\supp(\bar \rho_\resc)$. 
%%%%%%%%%%%%%%%%%%%%%%%%%%%%%%%%%%%%%%%%%%%%%%
\end{enumerate}
\end{corollary}

\begin{proof}
\begin{enumerate}[(i)]
%%%%%%%%%%%%%%%%%%%%%%%%%%%%%%%%%%%%%%%%%%%%%%%%
 \item For a stationary state $\bar \rho$ of equation \eqref{eq:KS}, let us us define the set 
 $$\Omega = \left\{x \in \RR^N : \bar \rho(x)>0\right\}\, .$$ Since $\bar \rho$ is continuous by Lemma \ref{lem:sstatesreg}, $\Omega$ is an open set with countably many open, possibly unbounded connected components. Let us take any bounded smooth connected open subset $\mathcal{U}$ such that $\overline{\mathcal{U}} \subset \Omega$. By continuity, $\bar \rho$ is bounded away from zero in $\mathcal{U}$, and thus $\nabla \bar \rho^{m-1}= \frac{m-1}{m\bar \rho} \nabla \bar \rho^m$ holds in the distributional sense in $\mathcal{U}$. From \eqref{eq:steady} in the definition of stationary states, we conclude that 
\begin{equation}\label{eq:steady3}
\nabla \left( \frac{m}{N(m-1)} \bar \rho^{m-1} + 2 \chi \, \bar S_k\right) = 0\, ,
\end{equation}
in the sense of distributions in $\Omega$. Hence, the function $C_k[\bar\rho](x): = \frac{m}{N(m-1)} \bar \rho^{m-1}(x) + 2 \chi \, \bar S_k(x)$ is constant in each connected component of $\Omega$, and so \eqref{eq:EL} follows. Additionally, it follows from \eqref{eq:steady3} that for any $x \in \RR^N$\
\begin{align*}
 \left|\nabla \bar \rho^{m-1}(x)\right| 
 =  \frac{2\chi N(m-1)}{m} \left|\nabla \bar S_k (x)\right| \leq c
\end{align*}
for some constant $c>0$ since $\bar S_k \in \mathcal{W}^{1,\infty}\left(\RR^N\right)$ by Lemma \ref{lem:regS}. Since $m\in (1,2)$, we conclude $\bar \rho \in \mathcal{W}^{1,\infty}\left(\RR^N\right)$.
%%%%%%%%%%%%%%%%%%%%%%%%%%%%%%%%%%%%%%%%%%%%%%%%%%%%%%%%%%%%%
\item We follow the same argument for a stationary state $\bar \rho_\resc$ of the rescaled equation \eqref{eq:KSresc} and using \eqref{eq:steady resc} in Definition \ref{def:sstates resc}, we obtain 
\begin{equation*}%\label{eq:steady3resc}
\nabla \left( \frac{m}{N(m-1)} \bar \rho_\resc^{m-1} + 2 \chi \, \bar S_k + \frac{|x|^2}{2}\right) = 0\, ,
\end{equation*}
in the sense of distributions in $\Omega$. Here, the function
$C_{k,\resc}[\bar\rho_\resc](x): = \frac{m}{N(m-1)} \bar \rho_\resc^{m-1}(x) + 2 \chi \, \bar S_k(x)+ \frac{|x|^2}{2}$ is again constant in each connected component of $\supp(\bar \rho_\resc)$. Similarly, it follows from Lemma \ref{lem:regS} that for any $\omega>0$ and $x \in B(0,\omega)$,
\begin{align*}
 \left|\nabla \bar \rho_\resc^{m-1}(x)\right| 
 =  \frac{N(m-1)}{m} \left(2\chi\left|\nabla \bar S_k (x)\right| + |x|\right) \leq c
\end{align*}
for some constant $c>0$, and so $\bar \rho_\resc \in \mathcal{W}^{1,\infty}_{loc}\left(\RR^N\right)$.
\end{enumerate}
\end{proof}

%%%%%%%%%%%%%%%%%%%%%%%%%%%%
%%%%%%%%%%%%%%%%%%%%%%%%%%%%
%%%%%%%%%%%%%%%%%%%%%%%%%%%%
%%%%%%%%%%%%%%%%%%%%%%%%%%%%
\subsection{Fair-competition: main results}
\label{sec:definitions}
%%%%%%%%%%%%%%%%%%%%%%%%%%%%
%%%%%%%%%%%%%%%%%%%%%%%%%%%%
%%%%%%%%%%%%%%%%%%%%%%%%%%%%
%%%%%%%%%%%%%%%%%%%%%%%%%%%%

It is worth noting that the functional $\mathcal F_{m,k}[\rho]$
possesses remarkable homogeneity properties. Indeed, the
mass-preserving dilation $\rho_\lambda(x) =
\lambda^{N}\rho(\lambda x)$ transforms the functionals as
follows:
\begin{align*}
\mU_m\left[\rho_\lambda\right] = 
 \begin{cases}
  \lambda^{N(m-1)}\mU_m[\rho]\, ,  &\text{if} \quad m\neq 1\, ,\\
  \mU_m[\rho] +\log\lambda\, , &\text{if}\quad  m=1\, ,
 \end{cases}
\end{align*}
% 
% \[
% \mU_m\left[\rho_\lambda\right] = \left\{\begin{array}{ll}
% \lambda^{N(m-1)}\mU_m[\rho]\, , \quad &m\neq 1 \medskip\\ 
% \mU_m[\rho] +
% \log\lambda\, , \quad \mbox{if}\quad  &m=1  \end{array}\right\}\,
% ,  \] 
and,
\begin{align*}
\mW_k \left[\rho_\lambda\right] =
 \begin{cases}
  \lambda^{-k} \mW_k[\rho] \, ,  &\text{if} \quad k\neq 0\, ,\\
   \mW_k[\rho] - \log \lambda\, , &\text{if}\quad  k=0\, .
 \end{cases}
\end{align*}
% \[\quad \mW_k \left[\rho_\lambda\right]  = \left\{\begin{array}{ll} \lambda^{-k} \mW_k[\rho] , \quad \, &k\neq 0 \medskip\\ \mW_k[\rho] - \log \lambda\, , \quad \mbox{if}\quad  &k = 0  \end{array}\right\}\, .
% \]
This motivates the following classification:
\begin{definition}[Three different regimes]
\
\begin{description}
\item[$\mathbf{N(m-1)+k = 0}$] This is the {\bf fair-competition}
regime, where homogeneities of the two competing contributions
exactly balance. If $k<0$, or equivalently $m>1$, then we will have a dichotomy according to $\chi>0$ (see Definition \ref{def:k} below). Some variants of the HLS inequalities are very related to this dichotomy. This was already proven in \cite{DoPe04,BlaDoPe06} for the Keller-Segel case in $N=2$, and in \cite{BCL} for the Keller-Segel case in $N\geq 3$. If $k>0$, that is $m<1$, no critical $\chi$ exists as we will prove in Section \ref{sec:Fast Diffusion Case k pos}. 

\item[$\mathbf{N(m-1)+k > 0}$] This is the {\bf
diffusion-dominating} regime. Diffusion is strong, and is expected to overcome aggregation, whatever $\chi>0$ is. This domination effect means that solutions exist globally in time and are bounded uniformly in time \cite{CaCa06,Sugi1,Sugi2}. Stationary states were found by minimisation of the free energy functional in two and three dimensions \cite{Strohmer2008,CCV} in the case of attractive Newtonian potentials. Stationary states are radially symmetric if $2-N\leq k<0$ as proven in \cite{CHVY}. Moreover, in the particular case of $N=2$, $k=0$, and $m>1$ it has been proved in \cite{CHVY} that the asymptotic behavior is given by compactly supported stationary solutions independently of $\chi$.

\item[$\mathbf{N(m-1)+k < 0}$] This is the {\bf
attraction-dominating} regime. This regime is less understood.
Self-attraction is strong, and can overcome the regularizing
effect of diffusion whatever $\chi>0$ is, but there also exist
global in time regular solutions under some smallness assumptions, see \cite{CLW,BL,CW,LW}.
\end{description}
\label{def:classification regimes}
\end{definition}
%%%%%%%%%%%%%%%%%%%%%%%%%%
We will here only concentrate on the fair-competition regime, and denote the corresponding energy functional by $\mathcal F_k[\rho] = \mathcal F_{1-k/N,k}[\rho]$. From now on, assume $N(m-1)+k = 0$. Notice that the functional $\mF_k$ is homogeneous in this regime, i.e.,
\begin{equation}\label{dil}
\mF_k[\rho_\lambda] = \lambda^{-k} \mF_k[\rho]\, .
\end{equation}
The analysis in the fair competition regime depends on the sign of $k$:

\begin{definition}[Three different cases in the fair-competition regime]
\
\begin{description}
\item[$\mathbf{k<0}$] This is the {\bf porous medium case} with $m\in (1,2)$, where diffusion is small in regions of small densities. The classical porous medium equation, i.e.    $\chi=0$, is very well studied, see \cite{VazquezPME} and the references therein. We have a dichotomy for existence of stationary states and global minimisers of the energy functional $\mF_k$ depending on a critical parameter $\chi_c$ which will be defined in \eqref{defchic}, and hence separate the sub-critical, the critical and the super-critical case, according to $\chi \lessgtr \chi_c$. The case $k<0$ is discussed in Section \ref{sec: Porous Medium Case k neg}.

\item[$\mathbf{k=0}$] This is the {\bf
logarithmic case}. There exists an explicit extremal density $\bar \rho_0 $
which realises the minimum of the functional $\mathcal F_0$ when
$\chi = 1$. Moreover, the functional $\mathcal F_0$ is bounded
below but does not achieve its infimum for $0<\chi<1$ while it is
not bounded below for $\chi > 1$. 
% We even have an explicit formula for $\bar \rho_0$ in \eqref{eq:StSt log}. 
Hence, $\chi_c = 1$ is the critical parameter in the
logarithmic case that was analysed in \cite{CaCa11} in one dimension and radial initial data in two dimensions.

\item[$\mathbf{k>0}$] This is the {\bf fast diffusion case} with $m\in (0,1)$, where diffusion is strong in regions of small densities. For any $\chi>0$, no radially symmetric non-increasing stationary states with bounded $k$th moment exist, and $\mF_k$ has no radially symmetric non-increasing minimisers. However, we have existence of self-similar profiles independently of $\chi>0$ as long as diffusion is not too fast, i.e. $k \leq 1$. Self-similar profiles can only exist if diffusion is not too strong with restriction $0<k<2$, that is $(N-2)/N<m<1$. The fast diffusion case is discussed in Section \ref{sec:Fast Diffusion Case k pos}.
\end{description}
\label{def:k}
\end{definition}

When dealing with the energy functional $\mF_k$, we work in the set of non-negative normalised densities,
$$
 \mY := \left\{ \rho \in L_+^1\left(\RR^N\right) \cap L^m \left(\RR^N\right) : ||\rho||_1=1\, , \, \int x\rho(x)\, dx=0\right\}.
$$
In rescaled variables, equation \eqref{eq:KSresc} is the formal gradient flow of the rescaled free energy functional $\mFr$, which is complemented with an additional quadratic confinement potential,
\begin{equation*}
\mFr[\rho]= \mathcal F_k [\rho] + \frac{1}{2}\mV[\rho]\, ,
\qquad
\mV[\rho] = \int_{\RR^N} |x|^2 \rho(x)\, dx \, .
\end{equation*}
Defining the sets
\begin{align*}
 \mY_2:=\left\{ \rho \in \mY:\mV[\rho]<\infty \right\}\, ,
 \qquad
 \mY_k:=\left\{ \rho \in \mY:\int_{\RR^N} |x|^k \rho(x)\, dx<\infty \right\}\, ,
 \end{align*}
 we see that $\mFr$ is well-defined and finite on $\mY_2$ for $k<0$ and on $ \mY_{2,k}:=\mY_2 \cap \mY_k$ for $k>0$. Thanks to the formal gradient flow structure in the Euclidean Wasserstein metric {\bf{W}}, we can write the rescaled equation \eqref{eq:KSresc} as
\begin{equation*}
\partial_t\rho
=\nabla \cdot \left( \rho\, \nabla \mTr [\rho]\right)
=- \nabla_{\bf W} \mFr[\rho]\, ,
%\label{eq:gradient flow resc}
\end{equation*}
where $\mTr$ denotes the first variation of the rescaled energy functional,
\begin{equation}\label{eq:1stvar resc}
 \mTr[\rho](x) :=\mT_k[\rho](x) +\frac{|x|^2}{2}
\end{equation}
with $\mT_k$ as defined in \eqref{eq:1stvar}.
%%%%%%%%%%%%%%%%%%%%%%%%%%%%%%%%%%%%%%%%%%%%%%%%%%%%%%%%%
%%%%%%%%%%%%%%%%%%%%%%%%%%%%%%%%%%%%%%%%%%%%%%%%%%%%%%%%%
%%%%%%%%%%%%%%%%%%%%%%%%%%%%%%%%%%%%%%%%%%%%%%%%%%%%%%%%%
In this paper, we prove the following results:
\begin{theorem}[The Critical Porous Medium Regime]\label{mainthm1}
In the porous medium regime $k\in(-N,0)$ and for critical interaction strengths $\chi=\chi_c$, there exist global minimisers of $\mF_k$ and they are radially symmetric non-increasing, compactly supported and uniformly bounded. Furthermore, all stationary states with bounded second moment are global minimisers of the energy functional $\mF_k$, and conversely, all global minimisers of $\mF_k$ are stationary states of \eqref{eq:KS}.
\end{theorem}
\begin{theorem}[The Sub-Critical Porous Medium Regime]\label{mainthm2}
In the porous medium regime $k\in(-N,0)$ and for sub-critical interaction strengths $0<\chi<\chi_c$, no stationary states exist for equation \eqref{eq:KS} and no minimisers exist for $\mF_k$. In rescaled variables, all stationary states are continuous and compactly supported. There exist global minimisers of $\mFr$ and they are radially symmetric non-increasing and uniformly bounded stationary states of equation \eqref{eq:KSresc}.
\end{theorem}

\begin{remark}\label{rmk:!}
Due to the homogeneity  \eqref{dil} of the functional $\mF_k$, each global minimiser gives rise to a family of global minimisers for $\chi=\chi_c$ by dilation since they have zero energy, see \eqref{eq:inf} below. It is an open problem to show that there is a unique global minimiser for $\chi=\chi_c$ modulo dilations. This uniqueness was proven in the Newtonian case in \cite{Yao}. We will explore the uniqueness modulo dilations of global minimisers in radial variables in a forthcoming paper. Notice that one obtains the full set of stationary states with bounded second moment for \eqref{eq:KS} as a by-product.

In contrast, in rescaled variables, we do not know if stationary states with second moment bounded are among global minimisers of $\mFr$ for the sub-critical case $0<\chi<\chi_c$ except in one dimension, see \cite{CCHCetraro}. It is also an open problem to show the uniqueness of radially symmetric stationary states of the rescaled equation \eqref{eq:KSresc} for $N\geq2$. 
\end{remark}

\begin{theorem}[The Fast Diffusion Regime]\label{mainthm3}
 In the fast diffusion regime $k \in (0,N)$ equation \eqref{eq:KS} has no radially symmetric non-increasing stationary states with $k$th moment bounded, and there are no radially symmetric non-increasing global minimisers for the energy functional $\mF_k$ for any $\chi>0$. In rescaled variables, radially symmetric non-increasing stationary states can only exist if $0<k< 2$, that is $(N-2)/N<m<1$. Similarly, global minimisers with finite energy $\mFr$ can only exist in the range $0<k<2N/(2+N)$, that is $N/(2+N)<m<1$. For $k \in (0,1]$, there exists a continuous radially symmetric non-increasing stationary state of the rescaled equation \eqref{eq:KSresc}.
\end{theorem}

%%%%%%%%%%%%%%%%%%%%%%%%%%%%
%%%%%%%%%%%%%%%%%%%%%%%%%%%%
%%%%%%%%%%%%%%%%%%%%%%%%%%%%
%%%%%%%%%%%%%%%%%%%%%%%%%%%%
\section{Porous Medium Case $k<0$}
\label{sec: Porous Medium Case k neg}
%%%%%%%%%%%%%%%%%%%%%%%%%%%%
%%%%%%%%%%%%%%%%%%%%%%%%%%%%
%%%%%%%%%%%%%%%%%%%%%%%%%%%%
%%%%%%%%%%%%%%%%%%%%%%%%%%%%
%%%%%%%%%%%%%%%%%%%%%%%%%%%%
In the porous medium case, we have $-N<k<0$ and hence $1<m<2$. Our aim in this section is to make a connection between global minimisers of the functionals $\mF_k$ and $\mFr$ and stationary states of the equations \eqref{eq:KS} and \eqref{eq:KSresc} respectively. We will show that in the critical case $\chi=\chi_c$, global minimisers and stationary states are equivalent for original variables. In the sub-critical case $0<\chi<\chi_c$, all minimisers of $\mFr$ will turn out to be stationary states of the rescaled equation \eqref{eq:KSresc}. 

%%%%%%%%%%%%%%%%%%%%%%%%%%%%
%%%%%%%%%%%%%%%%%%%%%%%%%%%%
%%%%%%%%%%%%%%%%%%%%%%%%%%%%
%%%%%%%%%%%%%%%%%%%%%%%%%%%%
\subsection{Global Minimisers}
\label{sec:Existence of stationary states k neg}
%%%%%%%%%%%%%%%%%%%%%%%%%%%%
%%%%%%%%%%%%%%%%%%%%%%%%%%%%
%%%%%%%%%%%%%%%%%%%%%%%%%%%%
%%%%%%%%%%%%%%%%%%%%%%%%%%%%

A key ingredient for the analysis in the porous medium case are certain functional inequalities which are variants of the Hardy-Littlewood-Sobolev (HLS) inequality, also known as the weak Young's inequality (Theorem 4.3, \cite{LieLo01}):
\begin{align}
&\iint_{\RR^N\times \RR^N} f(x) {|x-y|^{k}} f(y)\, dxdy \leq C_{HLS}(p,q,\lambda) \|f\|_{L^p} \|f\|_{L^q}\, ,  \label{eq:HLS}\\ & \dfrac1p + \dfrac1q  = 2 + \dfrac k N \,  , \quad p,q>1\, ,\quad   k\in(-N,0) \, . \nonumber
\end{align}
%%%%%%%%%%%%%%%%%%%%%%%%%%%%%%%%%%%%%%%%%%%%%%%%%%HLSm
\begin{theorem}[Variation of HLS]\label{thm:HLSm}
Let $k \in (-N,0)$. For $f \in L^1\left(\RR^N\right) \cap L^m\left(\RR^N\right)$, we have
\begin{equation}\label{eq:HLSm}
\left|\iint_{\RR^N\times \RR^N} f(x){|x-y|^{k}}f(y) dxdy\right| \leq C_* ||f||^{\frac{N+k}{N}}_1 ||f||^m_m,
\end{equation}
where $C_*(k,N)$ is defined as the best constant.
\end{theorem}
%%%%%%%%%%%%%%%%%%%%%%%%%%%%%%%%%%
\begin{proof} The inequality is a direct consequence of the standard HLS inequality \eqref{eq:HLS} by choosing $p = q = \tfrac {2N}{2N+k}$, and of H\"{o}lder's inequality. For $k \in (-N,0)$ and for any $f \in L^1\left(\RR^N\right) \cap L^m\left(\RR^N\right)$, we have
$$
\left| \iint_{\RR \times \RR} f(x) |x-y|^k f(y) dx dy \right|
\leq C_{HLS} ||f||_p^2
\leq C_{HLS} ||f||_1^{\frac{N+k}{N}} ||f||_m^m.
$$
Consequently, $C_*$ is finite and bounded from above by $C_{HLS}$.
\end{proof}

Now, let us compute explicitly the energy of stationary states:

\begin{lemma}\label{lem:Fzero}
 For any $-N<k<0$ and $\chi>0$, all stationary states $\bar \rho$ of \eqref{eq:KS} with $|x|^2\bar \rho \in L^1\left(\RR^N\right)$ satisfy $\mF_k\left[\bar \rho\right]=0$.
\end{lemma}
\begin{proof}
 Integrating \eqref{eq:steady} against $x$, we obtain for $1-N<k<0$:
 \begin{align*}
  &\frac{1}{N}\int_{\RR^N}x\cdot \nabla \bar \rho^m
  = -2\chi \iint_{\RR^N\times\RR^N} x\cdot (x-y)|x-y|^{k-2} \bar \rho(x)\bar\rho(y)\, dxdy\\
  &-\int_{\RR^N}\bar \rho^m
  = -\chi \iint_{\RR^N\times\RR^N} (x-y)\cdot (x-y)|x-y|^{k-2} \bar \rho(x)\bar\rho(y)\, dxdy\\
   &\frac{1}{N(m-1)}\int_{\RR^N}\bar \rho^m
  = -\chi \iint_{\RR^N\times\RR^N} \frac{|x-y|^{k}}{k} \bar \rho(x)\bar\rho(y)\, dxdy\, ,
 \end{align*}
 and the result immediately follows. If $-N<k\leq1-N$, we obtain
 \begin{align*}
  \frac{1}{N}\int_{\RR^N}x\cdot \nabla \bar \rho^m
  &= -2\chi \iint_{\RR^N\times\RR^N} x\cdot (x-y)|x-y|^{k-2} \bar \rho(x)\left[\bar\rho(y)-\bar \rho(x)\right]\, dxdy\\
  &= -\chi \iint_{\RR^N\times\RR^N} (x-y)\cdot (x-y)|x-y|^{k-2} \bar \rho(x)\bar\rho(y)\, dxdy
\end{align*}
since the singularity disappears when splitting the interaction term in half and exchanging $x$ and $y$ in the second part. Hence, we conclude as above. In both cases, a bounded second moment is necessary to allow for the use of $|x|^2/2$ as a test function by a standard approximation argument using suitable truncations.
\end{proof}

Let us point out that the the previous computation is possible due to the homogeneity of the functional $\mF_k$. In fact, a formal use of the Euler theorem for homogeneous functions leads  to this statement. This argument does not apply in the logarithmic case $k=0$. Here, it allows to connect stationary states and minimisers of $\mF_k$.\\

It follows directly from Theorem \ref{thm:HLSm}, that for all $\rho \in \mY$ and for any $\chi>0$,
\begin{equation*}%\label{HLSmcons1}
    \mF_k[\rho]\geq \frac{1-\chi C_*}{N(m-1)} ||\rho||_m^m\, ,
\end{equation*}
where $C_*=C_*(k,N)$ is the optimal constant defined in \eqref{eq:HLSm}. Since global minimisers have always smaller or equal energy than stationary states, and stationary states have zero energy by Lemma \ref{lem:Fzero}, it follows that $\chi \geq 1/C_*$. We define the \emph{critical interaction strength} by 
\begin{equation}\label{defchic}
 \chi_c(k,N):=\frac{1}{C_*(k,N)}\, ,
\end{equation}
and so for $\chi=\chi_c$, all stationary states of equation \eqref{eq:KS} are global minimisers of $\mF_k$.
We can also directly see that for $0<\chi<\chi_c$, no stationary states exist. 
These observations can be summarised in the following theorem:

\begin{theorem}[Stationary States in Original Variables] \label{thm:sstatesmin}
Let $-N<k<0$. For critical interaction strength $\chi=\chi_c$, all stationary states $\bar \rho$ of equation \eqref{eq:KS} with $|x|^2\bar \rho \in L^1\left(\RR^N\right)$ are global minimisers of $\mF_k$. For sub-critical interaction strengths $0<\chi<\chi_c$, no stationary states with $|x|^2\bar \rho \in L^1\left(\RR^N\right)$ exist for equation \eqref{eq:KS}. 
\end{theorem}

We now turn to the study of global minimisers of $\mF_k$ and $\mFr$ with the aim of proving the converse implication to Theorem \ref{thm:sstatesmin}.
Firstly, we have the following existence result: 

\begin{proposition}[Existence of Global Minimisers]\label{prop:existence min}
 Let $k \in (-N,0)$. 
 \begin{enumerate}[(i)]
  \item If $\chi=\chi_c$, then there exists a radially symmetric and non-increasing function $\tilde \rho \in \mY$ satisfying $\mF_k[\tilde \rho]=0$.
  \item \label{prop:existenceminsubcrit}
 If $\chi < \chi_c$, then $\mF_k$ does not admit global minimisers, but there exists a global minimiser $\tilde \rho$ of $\mFr$ in $\mY_2$.
  \item If $\chi>\chi_c$, then both $\mF_k$ and $\mFr$ are not bounded below.
 \end{enumerate}
\end{proposition}

\begin{proof} 
\begin{enumerate}[(i)]
Generalising the argument in \cite[Proposition 3.4]{BCL}, we obtain the following result for the behaviour of the free energy functional $\mF_k$: Let $\chi>0$. For $k \in (-N,0)$, we have
 \begin{equation}\label{eq:inf}
  I_k(\chi): = \inf_{\rho \in \mathcal Y} \mF_k[\rho] =
  \begin{cases}
   0 &\text{if} \quad \chi \in (0,\chi_c],\\
   -\infty &\text{if} \quad \chi > \chi_c \, ,
  \end{cases}
 \end{equation}
and the infimum $I_k(\chi)$ is only achieved if $\chi=\chi_c$. This implies statements (ii) and (iii) for $\mF_k$. Case (iii) directly follows also in rescaled variables as in \cite[Proposition 5.1]{BCL}. The argument in the sub-critical case (ii) for $\mFr$ is a bit more subtle than in the critical case (i) since we need to make sure that the second moment of our global minimiser is bounded. We will here only prove (ii) for rescaled variables, as (i) and (ii) in original variables are straightforward generalisations from \cite[Lemma 3.3]{BCL} and \cite[Proposition 3.4]{BCL} respectively. \\

Inequality \eqref{eq:HLSm} implies that the rescaled free energy is bounded on $\mY_2$ by
\begin{equation}\label{eq:Fbounds resc}
-\frac{C_*}{k} \left(\chi_c + \chi \right) ||\rho||_m^m + \frac{1}{2}\mV[\rho]
\geq\mFr[\rho] \geq -\frac{C_*}{k} \left(\chi_c - \chi \right) ||\rho||_m^m + \frac{1}{2}\mV[\rho],
\end{equation}
and it follows that the infimum of $\mFr$ over $\mY_2$ in the sub-critical case is non negative. Hence, there exists a minimising sequence $(p_j) \in \mY_2$, 
$$
\mFr[p_j] \to \mu:= \inf_{\rho \in \mY_2} \mFr[\rho].
$$
Note that $||p_j||_m$ and $\mV[p_j]$ are uniformly bounded, $||p_j||_m + \mV[p_j] \leq C_0$ say,  since from \eqref{eq:Fbounds resc}
\begin{equation*}
0<
-\frac{C_*}{k} \left(\chi_c - \chi \right) ||p_j||_m^m + \frac{1}{2}\mV[p_j]
\leq
\mFr[p_j]
\leq
\mFr[p_0].
\end{equation*}
Further, the radially symmetric decreasing rearrangement $(p_j^*)$ of $(p_j)$ satisfies 
$$
||p_j^*||_m=||p_j||_m, \quad
\mV[p_j^*] \leq \mV[p_j], \quad 
\mW_k[p_j^*] \leq \mW_k[p_j]
$$ by the reversed Hardy-Littlewood-Sobolev inequality \cite{kesavan} and Riesz rearrangement inequality \cite{LieLo01}. In other words, $\mFr[p_j^*] \leq \mFr[p_j]$ and so $(p_j^*)$ is also a minimising sequence.

To show that the infimum is achieved, we start by showing that $(p_j^*)$ is uniformly bounded at a point. For any choice of $R>0$, we have
\begin{align*}
 1 = ||p^*_j||_1 
 &= \sigma_N \int_0^\infty p_j^*(r) r^{N-1} \, dr \\
&\geq \sigma_N\int_0^R  p_j^*(r)r^{N-1} \, dr 
\geq \sigma_N \frac{R^N}{N} p_j^*(R)\,.
\end{align*}
Similarly, since $||p_j^*||_m$ is uniformly bounded,
\begin{align*}
C_0 &\geq ||p_j^*||_m^m = \sigma_N \int_0^\infty r^{N-1} p_j^*(r)^m \, dr \\
&\geq \sigma_N\int_0^R r^{N-1} p_j^*(r)^m \, dr 
\geq \sigma_N \frac{R^N}{N} p_j^*(R)^m\,.
\end{align*}
We conclude that
\begin{equation}\label{pbound2}
0 \leq p_j(R) \leq b(R) := C_1 \inf\left\{R^{-N}, R^{- \frac{N}{m}}\right\}, \qquad \forall R>0
\end{equation}
for a positive constant $C_1$ only depending on $N$, $m$ and $C_0$. Then by Helly's Selection Theorem there exists a subsequence $(p^*_{j_n})$ and a non-negative function $\tilde \rho:\RR^N \to \RR$ such that $p^*_{j_n} \to \tilde \rho$ pointwise almost everywhere. In addition, a direct calculation shows that $ x \mapsto b(|x|) \in L^{\frac{2N}{2N+k}}\left(\RR^N\right)$, and hence, using \eqref{eq:HLS} for $p=q=2N/(2N+k)$, we obtain
$$
(x,y) \mapsto |x-y|^k b(|x|)b(|y|)\, \in L^1(\RR^N \times \RR^N).
$$
Together with \eqref{pbound2} and the pointwise convergence of $(p^*_{j_n})$, we conclude
$$
\mathcal W_k (p^*_{j_n}) \to \mathcal W_k (\tilde \rho) < \infty
$$
by Lebesgue's dominated convergence theorem. 
In fact, since $||p^*_{j_n}||_m$ and $\mV[p^*_{j_n}]$ are uniformly bounded and $||p^*_{j_n}||_1=1$, we have the existence of a subsequence $(p^*_{j_l})$ and a limit $P \in L^1\left(\RR^N\right)$ such that $p^*_{j_l} \to P$ weakly in $L^1\left(\RR^N\right)$ by the Dunford-Pettis Theorem. 
Using a variant of Vitali's Lemma \cite{Rudin87}, we see that the sequence $(p^*_{j_l})$ actually converges strongly to $\tilde \rho$ in $L^1\left(\RR^N\right)$ on all finite balls in $\RR^N$. In other words, $P=\tilde \rho$ almost everywhere. Furthermore, $\tilde \rho$ has finite second moment by Fatou's Lemma,
$$
\mV[\tilde \rho] \leq \liminf_{l \to \infty} \mV[p^*_{j_l}] \leq C_0,
$$
and by convexity of $|.|^m$ for $m\in(1,2)$, we have lower semi-continuity,
$$
\int \tilde \rho^m \leq \liminf_{l \to \infty} \int \left(p^*_{j_l}\right)^m \leq C_0.
$$
We conclude that $\tilde \rho \in \mY_2$ and
\begin{align*}
 \mFr[\tilde \rho] 
 \leq \lim_{l\to \infty} \mFr[p^*_{j_l}] = \mu.
\end{align*}
Hence, $\tilde \rho$ is a global minimiser of $\mFr$.
\end{enumerate}
\end{proof}
%%%%%%%%%%%%%%%%%%%%%%%%%%%%
\begin{remark}
 The existence result in original variables also provides optimisers for the variation of the HLS inequality \eqref{eq:HLSm}, and so the supremum in the definition of $C_*(N,k)$ is in fact attained.
\end{remark}

The following necessary condition is a generalisation of results in \cite{BCL}, but using a different argument inspired by \cite{CCV}.

%%%%%%%%%%%%%%%%%%%%%%%%%%%%%%%%%%%%%%%%%%%
%%%%%%%%%%%%%%%%%%%%%%%%%%%%%%%%%%%%%%%%%%%
\begin{proposition}[Necessary Condition for Global Minimisers] \label{prop:charac min}
Let $k\in(-N,0)$.
\begin{enumerate}[(i)]
 \item If $\chi=\chi_c$ and $\rho \in \mY$ is a global minimiser of $\mF_k$, then $\rho$ is radially symmetric non-increasing, satisfying 
 \begin{equation}\label{eq:min}
   \rho^{m-1}(x) = 
	    \frac{N(m-1)}{m}\left( -2\chi \frac{|x|^k}{k} \ast \rho(x) + D_k[\rho]\right)_+ \quad \text{a.e. in}\, \,  \RR^N.
 \end{equation}
Here, we denote
$$
D_k[\rho] := 2 \mF_k[\rho] + \frac{m-2}{N(m-1)} ||\rho||_m^m.
$$
\item If $0<\chi<\chi_c$ and $\rho \in \mY_2$ is a global minimiser of $\mFr$, then $\rho$ is radially symmetric non-increasing, satisfying 
	  \begin{equation}\label{eq:min resc}
	    \rho^{m-1}(x) = 
	    \frac{N(m-1)}{m}\left( -2\chi \frac{|x|^k}{k} \ast \rho(x)-\frac{|x|^2}{2} + D_{k,\resc} [\rho]\right)_+ \quad \text{a.e. in}\, \,  \RR^N.
	  \end{equation}
Here, we denote 
\begin{align*}
D_{k, \resc}[\rho] :=\, &2 \mFr[\rho] - \frac{1}{2} \mV[\rho] + \frac{m-2}{N(m-1)} ||\rho||_m^m\, .
% = &\frac{m}{N(m-1)}||\rho||_m^m +2\chi \int\int_{\RR^N \times \RR^N} \frac{|x-y|^k}{k}\rho(x)\rho(y), dxdy + \frac{1}{2}\int_{\RR^N} |x|^2 \rho(x)\, dx.
\end{align*}
\end{enumerate}
\end{proposition}
%%%%%%%%%%%%%%%%%%%%%%%%%%%%%%%%%%%%%%%%%%%%%
%%%%%%%%%%%%%%%%%%%%%%%%%%%%%%%%%%%%%%%%%%%%%
\begin{proof}
\begin{enumerate}[(i)]
 \item 
Let us write as in \eqref{eq:functional}
$$
\mF_k[\rho]=\mU_{1-k/N}[\rho] + \chi \mW_k[\rho], \qquad
\mU_{m}[\rho]= \frac{1}{N(m-1)} ||\rho||_m^m, \quad \mbox{and}
$$
$$
\mW_k [\rho] = \iint_{\RR^N \times \RR^N} \frac{|x-y|^k}{k} \rho(x) \rho(y)\, dx dy.
$$
We will first show that all global minimisers of $\mF_k$ are radially symmetric non-increasing. Indeed, let $\rho$ be a global minimiser of $\mF_k$ in $\mY$, then for the symmetric decreasing rearrangement $\rho^*$ of $\rho$, we have $\mU_m[\rho^*]=\mU_m[ \rho]$ and by the Riesz rearrangement inequality \cite[Lemma 2]{CaLo92}, $\mW[\rho^*] \leq \mW[\rho]$. So $\mF_k[\rho^*] \leq \mF_k[\rho]$ and since $ \rho$ is a global minimiser this implies $\mW_k[\rho^*]=\mW_k[\rho]$. By Riesz rearrangement properties \cite[Lemma 2]{CaLo92}, there exists $x_0 \in \RR^N$ such that $\rho(x) = \rho^*(x-x_0)$ for all $x \in \RR^N$. Moreover, we have 
$$
\int_{\RR^N} x \rho(x) \, dx = x_0 + \int_{\RR^N} x \rho^*(x) \, dx = x_0,
$$
and thus the zero centre-of-mass condition holds if and only if $x_0=0$, giving $ \rho = \rho^*$.
For any test function $\psi \in C_c^\infty\left(\RR^N\right)$ such that $\psi(-x)=\psi(x)$, we define
$$
\varphi(x)=\rho(x) \left( \psi(x) - \int_{\RR^N} \psi(x) \rho(x) \, dx \right).
$$
We fix $0 < \eps < \eps_0:=(2||\psi||_\infty)^{-1}$. Then
$$
\rho+\eps \varphi = \rho \left( 1+ \eps\left( \psi - \int_{\RR^N} \psi \rho \right) \right) \geq \rho \left(1-2||\psi||_\infty \eps\right) \geq0,
$$
and so $\rho + \eps \varphi \in L_+^1\left(\RR^N\right) \cap L^m\left(\RR^N\right)$. Further, $\int \varphi(x) \, dx = \int x \varphi(x) \, dx = 0$, and hence $\rho + \eps \varphi \in \mY$. Note also that $\supp (\varphi) \subseteq \bar\Omega:=\supp(\rho)$. To calculate the first variation $\mT_k$ of the functional $\mF_k$, we need to be careful about regularity issues. Denoting by $\Omega$ the interior of $\bar \Omega$, we write
\begin{align*}
 \frac{\mF_k[\rho+\eps \varphi] - \mF_k[\rho]}{\eps}
 &= \frac{1}{N(m-1)} \int_{\Omega} \frac{(\rho+\eps\varphi)^m - \rho^m}{\eps} \, dx \\
 &\qquad+ 2 \chi \int_{\RR^N} \left( \frac{|x|^k}{k} \ast \rho(x)\right) \varphi(x) \, dx + \eps \mW_k[\varphi]\\
 &= \frac{m}{N(m-1)} \int_0^1 \mG_\eps(t) \, dt \\
 &\qquad+ 2 \chi \int_{\RR^N} \left( \frac{|x|^k}{k} \ast \rho(x)\right) \varphi(x) \, dx + \eps \mW_k[\varphi],
\end{align*}
where $\mG_\eps(t) := \int_{\Omega} \left| \rho+t\eps\varphi\right|^{m-2}(\rho+t\eps\varphi) \varphi\, dx$. Then by H\"{o}lder's inequality,
$$
\left| \mG_\eps(t)\right| \leq \left(||\rho||_m+\eps_0||\varphi||_m\right)^{m-1}||\varphi||_m
$$
for all $t \in [0,1]$ and $\eps \in (0,\eps_0)$. Lebesgue's dominated convergence theorem yields
$$
\int_0^1 \mG_\eps(t) \, dt \to \int_{\Omega} \rho^{m-1}(x) \varphi(x)\, dx
$$
as $\eps \to 0$. In addition, one can verify that $\mW_k[\varphi]\leq 4 ||\psi||_\infty^2 \mW_k[\rho]<\infty$. Hence,
\begin{align*}
 \lim_{\eps \to 0} \left(\frac{\mF_k[\rho+\eps \varphi] - \mF_k[\rho]}{\eps}\right)
 &= \frac{m}{N(m-1)} \int_{\Omega}\rho^{m-1}(x) \varphi(x) \, dx \\
 &\qquad+ 2 \chi \int_{\RR^N} \left( \frac{|x|^k}{k} \ast \rho(x)\right) \varphi(x) \, dx\\
 &= \int_{\RR^N} \mT_k[\rho](x) \varphi(x)\, dx\, ,
\end{align*}
proving \eqref{eq:1stvar}.
Since $\rho$ is a global minimiser, $\mF_k[\rho+\eps \varphi] \geq \mF_k[\rho]$ and hence $\int \mathcal T_k[\rho](x) \varphi(x)\, dx \geq 0$.
Taking $-\psi$ instead of $\psi$, we obtain by the same argument $\int \mathcal T_k[\rho](x) \varphi(x)\, dx \leq 0$, and so
$$
\int_{\RR^N} \mathcal T_k[\rho](x) \varphi(x)\, dx =0.
$$
Owing to the choice of $\varphi$,
\begin{align*}
 0=\int_{\RR^N} \mT_k[\rho](x) \varphi(x)\, dx 
 &= \int_{\RR^N} \mT_k[\rho](x) \rho(x)\psi(x)\, dx \\
 &\qquad- \left(\int_{\RR^N} \psi \rho\right)\left(\int_{\RR^N} \mT_k[\rho](x) \rho(x)\, dx\right)\\
 &= \int_{\RR^N} \left(  \frac{m}{N(m-1)}\rho^m(x) +2\chi\rho(x) \left(\frac{|x|^k}{k} \ast \rho(x)\right) \right)\psi(x)\, dx \\
 &\qquad- \left(\int_{\RR^N} \psi \rho\right)\left(2 \mF_k[\rho] + \frac{m-2}{N(m-1)} ||\rho||_m^m\right)\\
 &= \int_{\RR^N} \rho(x) \psi(x)\left(\mT_k[\rho](x) - D_k[\rho]\right)\, dx 
\end{align*}
for any symmetric testfunction $\psi \in C_c^\infty\left(\RR^N\right)$. Hence $\mT_k[\rho](x) = D_k[\rho]$ a.e. in $\bar\Omega$, i.e.
\begin{equation}\label{eq:indomain}
 \rho^{m-1}(x) = \frac{N(m-1)}{m} \left( -2\chi \, \frac{|x|^k}{k} \ast \rho(x) + D_k[\rho]\right) \quad \text{a.e. in}\, \bar\Omega.
\end{equation}
%%%%%%%%%%%%%%%%%%%%%%%%%%
Now, we turn to conditions over $\rho$ on the whole space. Let $\psi\in C_c^\infty\left(\RR^N\right)$, $\psi(-x)=\psi(x)$, $\psi \geq 0$, and define
$$
\varphi(x):=\psi(x)-\rho(x) \int_{\RR^N}\psi(x)\,dx \quad \in L^1\left(\RR^N\right) \cap L^m\left(\RR^N\right).
$$
Then for $0<\eps<\eps_0:=(||\psi||_\infty |\supp(\psi)|)^{-1}$, we have
$$
\rho+\eps\varphi\geq\rho\left(1-\eps\int_{\RR^N}\psi\right)\geq\rho\left(1-\eps||\psi||_\infty |\supp(\psi)|\right).
$$
So $\rho+\eps\varphi \geq0$ in $\bar\Omega$, and also outside $\bar\Omega$ since $\psi\geq0$, hence $\rho+\eps\varphi \in \mY$. Repeating the previous argument, we obtain 
$$
\int_{\RR^N}\mT_k[\rho](x)\varphi(x)\,dx \geq0.
$$
Using the expression of $\varphi$, we have
\begin{align*}
0\leq
 \int_{\RR^N} \mT_k[\rho](x) \varphi(x)\, dx 
 &= \int_{\RR^N} \mT_k[\rho](x) \psi(x)\, dx \\
 &\qquad- \left(\int_{\RR^N} \psi\right)\left(\int_{\RR^N} \mT_k[\rho](x) \rho(x)\, dx\right)\\
 &= \int_{\RR^N} \left(  \frac{m}{N(m-1)}\rho^{m-1}(x) +2\chi\frac{|x|^k}{k} \ast \rho(x) \right)\psi(x)\, dx \\
&\qquad - \left(\int_{\RR^N} \psi\right)\left(2 \mF_k[\rho] + \frac{m-2}{N(m-1)} ||\rho||_m^m\right)\\
 &= \int_{\RR^N} \psi(x)\left(\mT_k[\rho](x) - D_k[\rho]\right)\, dx \, .
\end{align*}
Hence $\mT_k[\rho](x) \geq D_k[\rho]$ a.e. in $\RR^N$, and so
\begin{equation}\label{eq:outsidedomain}
 \rho^{m-1}(x) \geq \frac{N(m-1)}{m} \left( -2\chi\, \frac{|x|^k}{k} \ast \rho(x) + D_k[\rho]\right) \quad \text{a.e. in }\, \RR^N.
\end{equation}
Combining \eqref{eq:indomain} and \eqref{eq:outsidedomain} completes the proof of \eqref{eq:min}.

\item First, note that if $\rho \in \mY_2$ and $\rho^*$ denotes the symmetric decreasing rearrangement of $\rho$, then it follows from the reversed Hardy-Littlewood-Sobolev inequality \cite{kesavan} that $\mV[\rho^*] \leq \mV[\rho]$.
 Since $\mU_m[\rho^*]=\mU_m[ \rho]$ and $\mW[\rho^*] \leq \mW[\rho]$, we conclude $\mFr[\rho^*] \leq \mFr[\rho]$. For a global minimiser $\rho \in \mY_2$, we have $\mFr[\rho^*] = \mFr[\rho]$ and hence $\mW[\rho^*] = \mW[\rho]$ and $\mV[\rho^*] = \mV[\rho]$. The former implies that there exists $x_0 \in \RR^N$ such that $\rho(x) = \rho^*(x-x_0)$ for all $x \in \RR^N$ by Riesz rearrangement properties \cite[Lemma 2]{CaLo92}, and so the equality in second moments gives $\rho=\rho^*$.\\
 Next, we will derive equation \eqref{eq:min resc}. We define for any testfunction $\psi \in C_c^\infty\left(\RR^N\right)$ the function $\varphi(x)=\rho(x) \left( \psi(x) - \int_{\RR^N} \psi(x) \rho(x) \, dx \right)$,
and by the same argument as in (i), we obtain
\begin{align*}
 0=\int_{\RR^N} \mTr[\rho](x) \varphi(x)\, dx 
 = \int_{\RR^N} \rho(x) \psi(x)\left(\mTr [\rho](x) - D_{k, \resc}[\rho]\right)\, dx\, ,
\end{align*}
with $\mTr$ as given in \eqref{eq:1stvar resc}.
Hence $\mTr[\rho](x) = D_{k, \resc}[\rho]$ a.e. in $\bar\Omega:=\supp(\rho)$.
Following the same argument as in (i), we further conclude
$\mTr[\rho](x) \geq D_{k,\resc}[\rho]$ a.e. in $\RR^N$. Together with the equality on $\bar\Omega$, this completes the proof of \eqref{eq:min resc}.
\end{enumerate}
\end{proof}

\begin{remark} For critical interaction strength $\chi=\chi_c$, 
if $\bar \rho$ is a stationary state of equation \eqref{eq:KS} with bounded second moment, then it is a global minimiser of $\mF_k$ by Theorem \ref{thm:sstatesmin}. In that case, we can identify the constant $C_k[\bar \rho]$ in \eqref{eq:EL} with $D_k[\bar \rho]$ in \eqref{eq:min}, which is the same on all connected components of $\supp(\bar \rho)$.
\end{remark}

%%%%%%%%%%%%%%%%%%%%%%%%%%%
%%%%%%%%%%%%%%%%%%%%%%%%%%%
%%%%%%%%%%%%%%%%%%%%%%%%%%%
\subsection{Regularity Properties of Global Minimisers}
%%%%%%%%%%%%%%%%%%%%%%%%%%%¡¡
%%%%%%%%%%%%%%%%%%%%%%%%%%%
%%%%%%%%%%%%%%%%%%%%%%%%%%%
%%%%%%%%%%%%%%%%%%%%%%%%%%%
Proposition \ref{prop:charac min} allows us to conclude the following useful Corollary, adapting some arguments developed in \cite{BCL}. 
%%%%%%%%%%%%%%%%%%%%
%%%%%%%%%%%%%%%%%%%%
\begin{corollary}[Compactly Supported Global Minimisers]\label{cor:compactsupp_min}
 If $\chi=\chi_c$, then all global minimisers in $\mY$ are compactly supported. If $0<\chi<\chi_c$, then global minimisers of $\mFr$ are compactly supported.
\end{corollary}
%%%%%%%%%%%%%%%%%%%
%%%%%%%%%%%%%%%%%%%
\begin{proof}
Let $\rho\in \mY$ be a global minimiser of $\mF_k$. Then $\rho$ is radially symmetric and non-increasing by Proposition \ref{prop:charac min} (i) and has zero energy by \eqref{eq:inf} above. Using the expression of the constant $D_k[\rho]$ given by Proposition \ref{prop:charac min} (i), we obtain
 $$
 D_k[\rho]= \frac{m-2}{N(m-1)}||\rho||_m^m <0
 $$
 Let us assume that $\rho$ is supported on $\RR^N$. We will arrive at a contradiction by showing that $\rho^{m-1}$ and $W_k \ast \rho$ are in $L^{m/(m-1)}\left(\RR^N\right)$. Since 
 $$
 D_k[\rho]= \frac{m}{N(m-1)} \rho(x)^{m-1} + 2\chi \, \frac{|x|^k}{k} \ast \rho(x)
 $$
 a.e. in $\RR^N$ by \eqref{eq:min}, this would mean that the constant $D_k[\rho]<0$ is in $L^{m/(m-1)}$ and decays at infinity, which is obviously false.\\
It remains to show that $W_k \ast \rho$ is in $L^{m/(m-1)}\left(\RR^N\right)$ since $\rho \in L^m\left(\RR^N\right)$ by assumption. From  $\rho \in L^1\left(\RR^N\right)\cap L^m\left(\RR^N\right)$ we have $\rho \in L^r\left(\RR^N\right)$ for all $r \in (1,m]$ by interpolation, and hence $W_k \ast \rho \in L^s\left(\RR^N\right)$ for all $s \in (-N/k, Nm/(k(1-m))]$ by \cite[Theorem 4.2]{LieLo01}. Finally, we conclude that $W_k \ast \rho$ is in $L^{m/(m-1)}\left(\RR^N\right)$ since $-N/k < m/(m-1) < Nm/(k(1-m))$.

In the sub-critical case for the rescaled functional $\mFr$, we argue as above to conclude that for any global minimiser $\rho$ in $\mY_2$ we have $\rho^{m-1}$ and $W_k \ast \rho$ in $L^{m/(m-1)}\left(\RR^N\right)$. If $\rho$ were supported on the whole space, it followed from the Euler-Lagrange condition for the rescaled equation \eqref{eq:min resc} that $|x|^2 +C \in L^{m/(m-1)}\left(\RR^N\right)$ for some constant $C$. This is obviously false. 
\end{proof}
%%%%%%%%%%%%%%%%%%%%%%%%%%%%%
%%%%%%%%%%%%%%%%%%%%%%%%%%%%%
%%%%%%%%%%%%%%%%%%%%%%%%%%%%%

The same argument works for stationary states by using the necessary conditions \eqref{eq:EL} and \eqref{eq:ELresc}.

\begin{corollary}[Compactly Supported Stationary States]\label{compactsupp_sstates}
 If $\chi=\chi_c$, then all stationary states of equation \eqref{eq:KS} are compactly supported. If $0<\chi<\chi_c$, then all stationary states of the rescaled equation \eqref{eq:KSresc} are compactly supported.
\end{corollary}

We are now ready to show that all global minimisers are uniformly bounded, both in original and rescaled variables. In what follows, we use the notation $f(r) \lesssim g(r)$ to denote $f(r) \leq Cg(r)$ for some constant $C>0$ in order to avoid having to track constants that are not important for the argument.

%%%%%%%%%%%%%%%%%%%%%%%%%%%%%
%%%%%%%%%%%%%%%%%%%%%%%%%%%%%
%%%%%%%%%%%%%%%%%
\begin{lemma}\label{lem:induction}
 Let $\rho$ be either a global minimiser of $\mF_k$ over $\mY$ or a global minimiser of $\mFr$ over $\mY_2$. If there exists $p \in (-N,0]$ such that
 \begin{equation}\label{ind1}
   \rho(r) \lesssim 1+r^p \, \qquad \text{for all} \, \, r \in (0,1)\, ,
 \end{equation}
then for $r \in (0,1)$,
 \begin{equation}\label{ind2}
  \rho(r) \lesssim 
  \begin{cases}
   1+r^{g(p)} &\text{if} \, \, p \neq -N-k\, ,\\
   1+\left|\log(r)\right|^{\frac{1}{m-1}} &\text{if} \, \, p = -N-k\, ,\\
  \end{cases}
 \end{equation}
 where
 \begin{equation}\label{defgp}
  g(p)=\frac{p+N+k}{m-1}\, .
 \end{equation}
\end{lemma}

\begin{proof} Since global minimisers are radially symmetric non-increasing, we can bound $\rho(r)$ by $\rho(1)$ for all $r \geq 1$, and hence the bound \eqref{ind1} holds true for all $r>0$. Further, we know from Corollary \ref{cor:compactsupp_min} that all global minimisers are compactly supported. Let us denote $\supp(\rho)=B(0,R)$, $0<R<\infty$.
We split our analysis in four cases: (1) the regime $1-N<k<0$ with $k\neq 2-N$  and $N\geq 2$, where we can use hypergeometric functions in our estimates, (2) the Newtonian case $k=2-N$, $N\geq 3$, (3) the one dimensional regime $-1<k<0$ where we need a Cauchy principle value to deal with the singular integral in the mean-field potential gradient, but everyting can be computed explicitly, and finally (4) the regime $-N<k\leq1-N$ and $N\geq2$, where again singular integrals are needed to deal with the singularities of the hypergeometric functions.\\

\fbox{Case 1: $1-N<k<0$, $k\neq 2-N$, $N\geq 2$}\\

We would like to make use of the Euler-Lagrange condition \eqref{eq:min}, and hence we need to understand the behaviour of $W_k\ast \rho$. It turns out that it is advantegous to estimate the derivative instead, writing
\begin{equation}\label{eq:bd2}
 - \left(W_k \ast \rho\right)(r) =  - \left( W_k \ast \rho\right)(1)+\int_r^1\partial_r \left(W_k \ast \rho\right)(s) \, ds\, .
\end{equation}
The first term on the right-hand side can be estimated explicitly since for any $x \in \RR^N$ with $0<|x|=\gamma<\bar R$ for any $\bar R \geq R$:
\begin{align}\label{constest}
 - \left( W_k \ast \rho\right)(\gamma)
 %= &\left(-\frac{1}{k}\right)\int_{B(0,R)}|x-y|^k\rho(y)\, dy\notag\\
 = &\left(-\frac{1}{k}\right)\int_{B(0,R)\backslash B(0,\gamma/2)}|x-y|^k\rho(y)\, dy+\left(-\frac{1}{k}\right)\int_{B(0,\gamma/2)}|x-y|^k\rho(y)\, dy\notag\\
\lesssim &\left(1+\left(\frac{\gamma}{2}\right)^p\right)\int_{B(0,\bar R)\backslash B(0,\gamma/2)}|x-y|^k\, dy+\int_{0}^{\gamma/2}|\gamma-r|^k\rho(r)r^{N-1}\, dr\notag\\
\lesssim &\left(1+\left(\frac{\gamma}{2}\right)^p\right)\int_{B(x,\bar R+\gamma)}|x-y|^k\, dy+\left(\frac{\gamma}{2}\right)^k||\rho||_1\notag\\
= &\left(1+\left(\frac{\gamma}{2}\right)^p\right)\sigma_N \int_0^{\bar R+\gamma}r^{k+N-1}\, dr+\left(\frac{\gamma}{2}\right)^k<\infty\, .
\end{align}
For the second term, we use the formulation \eqref{derivpsi0} from the Appendix,
\begin{equation} \label{derivpsi}
 \partial_r \left(W_k \ast \rho\right)(r) = r^{k-1} \int_0^\infty \psi_k\left(\frac{\eta}{r}\right)\rho(\eta)\eta^{N-1}\, d \eta\, \,,
\end{equation}
where $\psi_k$ is given by \eqref{psi1} and can be written in terms of Gauss hypergeometric functions, see \eqref{psi2}. \\
%%%%%%%%%%%%%%%%%%%%%%%%%%%%%%%%%%%%%%%%%%%%%%%%%%%%
%%%%%%%%%%%%%%%%%%%%%%%%%%%%%%%%%%%%%%%%%%%%%%%%%%%%
%%%%%%%%%%%%%%%%%%%%%%%%%%%%%%%%%%%%%%%%%%%%%%%%%%%%

\textbf{Sub-Newtonian Regime} $\mathbf{1-N<k<2-N}$\\
Note that $\psi_k(s)<0$ for $s>1$ in the sub-Newtonian regime $1-N<k<2-N$ (see Appendix Lemma \ref{lem:psibehaviour} and Figure \ref{fig:psik_fig1}). Together with the induction assumption \eqref{ind1} and using the fact that $\rho$ is compactly supported, we have for any $r \in (0,R)$
\begin{align}
\partial_r \left(W_k \ast \rho\right)(r) 
% =\, &r^{k-1} \int_0^\infty \psi_k\left(\frac{\eta}{r}\right)\rho(\eta)\eta^{N-1}\, d \eta\notag\\
=\, &r^{k-1} \int_0^r \psi_k\left(\frac{\eta}{r}\right)\rho(\eta)\eta^{N-1}\, d \eta
+ r^{k-1} \int_r^R \psi_k\left(\frac{\eta}{r}\right)\rho(\eta)\eta^{N-1}\, d \eta\notag\\
=\, &r^{k+N-1} \int_0^1 \psi_k\left(s\right)\rho(rs)s^{N-1}\, ds
+ r^{k+N-1} \int_1^{R/r} \psi_k\left(s\right)\rho(rs)s^{N-1}\, ds\notag\\
\leq\, &r^{k+N-1} \int_0^1 \psi_k\left(s\right)\rho(rs)s^{N-1}\, ds\notag\\
\lesssim\,
&r^{k+N-1} \left(\int_0^1 \psi_k\left(s\right)s^{N-1}\, ds \right)
+r^{p+k+N-1} \left(\int_0^1 \psi_k\left(s\right)s^{p+N-1}\, ds\right)\notag\\
=\,
& C_1 r^{k+N-1}+C_2r^{p+k+N-1} \, ,\label{dWrho1sub}
\end{align}
where we defined
\begin{align*}
 C_1:=\int_0^1 \psi_k\left(s\right)s^{N-1}\, ds\, ,
 \qquad
 C_2:=\int_0^1 \psi_k\left(s\right)s^{p+N-1}\, ds\, .
\end{align*}
In the case when $r \in [R,\infty)$, we use the fact that $\psi_k(s)>0$ for $s\in(0,1)$ by Lemma \ref{lem:psibehaviour} and so we obtain by the same argument
\begin{align}
\partial_r \left(W_k \ast \rho\right)(r) 
=\, &r^{k-1} \int_0^R \psi_k\left(\frac{\eta}{r}\right)\rho(\eta)\eta^{N-1}\, d \eta
\lesssim\,
r^{k-1} \int_0^R \psi_k\left(\frac{\eta}{r}\right)\left(1+\eta^p\right)\eta^{N-1}\, d \eta\notag\\
=\,
&r^{k+N-1} \left(\int_0^{R/r} \psi_k\left(s\right)s^{N-1}\, ds \right)
+r^{p+k+N-1} \left(\int_0^{R/r} \psi_k\left(s\right)s^{p+N-1}\, ds\right)\notag\\
\leq\,
& C_1 r^{k+N-1}+C_2 r^{p+k+N-1} \, ,\label{dWrho2sub}
\end{align}
with constants $C_1$, $C_2$ as given above. It is easy to see that $C_1$ and $C_2$ are indeed finite. From \eqref{lowerlim1} it follows that $\psi_k\left(s\right)s^{N-1}$ and $\psi_k\left(s\right)s^{p+N-1}$ are integrable at zero since $-N<p$. Similarly, both expressions are integrable at one using \eqref{upperlim1-} in Lemma \ref{Beh1}. Hence, we conclude from \eqref{dWrho1sub} and \eqref{dWrho2sub} that for any $r \in (0,1)$,
\begin{equation*}\label{ind3}
 \partial_r \left(W_k \ast \rho\right)(r) 
\lesssim r^{k+N-1} +r^{p+k+N-1} \, .
\end{equation*}
%%%%%%%%%%%%%%
Substituting into the right-hand side of \eqref{eq:bd2} and using \eqref{constest} yields
\begin{align*}%\label{Wrho0sub}
- \left(W_k \ast \rho\right)(r)
% =&-\left(W_k \ast \rho\right)(1) + \int_r^1 \partial_r\left(W_k\ast \rho\right)(s)\, ds
\lesssim \,
1+ \int_r^1 \left(s^{k+N-1}+s^{p+k+N-1}\right)\, ds
\end{align*}
for any $r \in (0,1)$.
%%%%%%%%%%%%%
It follows that for $p\neq -k-N$,
\begin{align*}
- \left(W_k \ast \rho\right)(r)\lesssim \,
&1+ \frac{1-r^{k+N}}{k+N} + \frac{1-r^{p+k+N}}{p+k+N}
\lesssim\,
1+ r^{p+k+N}\, .
\end{align*}
If $p=-k-N$, we have instead
\begin{align*}
- \left(W_k \ast \rho\right)(r)\lesssim \,
&1+ \frac{1-r^{k+N}}{k+N} -\log(r)
\lesssim\,
1+ |\log(r)|\, .
\end{align*}
If $\rho$ is a global minimiser of $\mF_k$, then it satisfies the Euler-Lagrange condition \eqref{eq:min}. Hence, we obtain \eqref{ind2} with the function $g(p)$ as defined in \eqref{defgp}. If $\rho$ is a global minimiser of the rescaled functional $\mFr$, then it satisfied condition \eqref{eq:min resc} instead, and we arrive at the same result.\\

%%%%%%%%%%%%%%%%%%%%%%%%%%%%%%%%%%%
%%%%%%%%%%%%%%%%%%%%%%%%%%%%%%%%%%%
%%%%%%%%%%%%%%%%%%%%%%%%%%%%%%%%%%%
\textbf{Super-Newtonian Regime} $\mathbf{k>2-N}$\\
In this regime, $\psi_k(s)$ is continuous, positive and strictly decreasing for $s>0$ (see Appendix Lemma \ref{lem:psibehaviour} and Figure \ref{fig:psik_fig4}) and hence integrable at any positive constant. 
Under the induction assumption \eqref{ind1} and using the fact that $\rho$ is compactly supported and radially symmetric non-increasing, we have for any $r \in (0,R)$
\begin{align*}
\partial_r \left(W_k \ast \rho\right)(r) 
=\, &r^{k-1} \int_0^R \psi_k\left(\frac{\eta}{r}\right)\rho(\eta)\eta^{N-1}\, d \eta
=\, r^{N+k-1} \int_0^{R/r} \psi_k\left(s\right)\rho(rs)s^{N-1}\, d \eta\notag\\
\lesssim\,
&r^{k+N-1} \left(\int_0^{R/r} \psi_k\left(s\right)s^{N-1}\, ds \right)
+r^{p+k+N-1} \left(\int_0^{R/r} \psi_k\left(s\right)s^{p+N-1}\, ds\right)\notag\\
=\,
& C_1(r)r^{k+N-1}+C_2(r)r^{p+k+N-1} \, ,%\label{dWrho1super}
\end{align*}
where we defined
\begin{align*}
 C_1(r):=\int_0^{R/r} \psi_k\left(s\right)s^{N-1}\, ds\, ,
 \qquad
 C_2(r):=\int_0^{R/r} \psi_k\left(s\right)s^{p+N-1}\, ds\, .
\end{align*}
Next, let us verify that $C_1(\cdot)$ and $C_2(\cdot)$ are indeed bounded above. From \eqref{lowerlim1} it follows again that $\psi_k\left(s\right)s^{N-1}$ and $\psi_k\left(s\right)s^{p+N-1}$ are integrable at zero since $-N<p$. In order to deal with the upper limit, we make use of property \eqref{lim0}, which implies that there exist constants $L>1$ and $C_L>0$ such that for all $s\geq L$, we have
$$
\psi_k(s) \leq C_L s^{k-2}\, .
$$
It then follows that for $r<R/L$, 
$$
\int_L^{R/r} \psi_k\left(s\right)s^{N-1}\, d s
\leq \frac{C_L}{N+k-2}\left(\left(\frac{R}{r}\right)^{N+k-2}-L^{N+k-2}\right)\, ,
$$
and hence we obtain
$$
C_1(r)
=\int_0^{L} \psi_k\left(s\right)s^{N-1}\, d s
+\int_L^{R/r} \psi_k\left(s\right)s^{N-1}\, d s
\lesssim 1+r^{-N-k+2}\, .
$$
Similarly,
$$
C_2(r)
=\int_0^{L} \psi_k\left(s\right)s^{p+N-1}\, d s
+\int_L^{R/r} \psi_k\left(s\right)s^{p+N-1}\, d s
\lesssim 1+r^{-p-N-k+2}\, .
$$
We conclude
\begin{align}
\partial_r \left(W_k \ast \rho\right)(r) 
\lesssim
& \left(1+r^{-N-k+2}\right)r^{k+N-1}+\left(1+r^{-p-N-k+2}\right)r^{p+k+N-1}\notag\\
\lesssim
& 1+r^{k+N-1}+r^{p+k+N-1}
\, .\label{dWrho3super}
\end{align}
For $R/L \leq r < R$ on the other hand we can do an even simpler bound:
$$
C_1(r)+C_2(r)
\leq \int_0^{L} \psi_k\left(s\right)s^{N-1}\, ds+\int_0^{L} \psi_k\left(s\right)s^{p+N-1}\, d s\lesssim 1\, ,
$$
and so we can conclude for \eqref{dWrho3super} directly.
%%%%%%%%%%%%%%%%%%%%%%%%%%%%%%%%%%%%%%%%%
%%%%%%%%%%%%%%%%%%%%%%%%%%%%%%%%%%%%%%%%%
In the case when $r \in [R,\infty)$, we obtain by the same argument
\begin{align*}
\partial_r \left(W_k \ast \rho\right)(r) 
\lesssim\,
&r^{k+N-1} \left(\int_0^{R/r} \psi_k\left(s\right)s^{N-1}\, ds \right)
+r^{p+k+N-1} \left(\int_0^{R/r} \psi_k\left(s\right)s^{p+N-1}\, ds\right)\notag\\
\leq\,
& C_1(R) r^{k+N-1}+C_2(R) r^{p+k+N-1} \, ,
\end{align*}
with constants $C_1(\cdot)$, $C_2(\cdot)$ as given above, and so we conclude that the estimate \eqref{dWrho3super} holds true for any $r >0$.
Substituting \eqref{dWrho3super} into \eqref{eq:bd2}, we obtain for $r \in (0,1)$
\begin{align*}%\label{Wrho0super}
- \left(W_k \ast \rho\right)(r)
\lesssim \,
&1+ \int_r^1 \left(s^{k+N-1}+s^{p+k+N-1}\right)\, ds
\end{align*}
and so we conclude as in the sub-Newtonian regime.\\
\bigskip
%%%%%%%%%%%%%%%%%%%%%%%%%%%%%%%%%%%%%%%%%%%%%%
%%%%%%%%%%%%%%%%%%%%%%%%%%%%%%%%%%%%%%%%%%%%%%
%%%%%%%%%%%%%%%%%%%%%%%%%%%%%%%%%%%%%%%%%%%%%%

 \fbox{Case 2: $k=2-N$, $N\geq 3$}
 \textbf{Newtonian Regime}\\ 
 
 In the Newtonian case, we can make use of known explicit expressions. We write as above
\begin{equation}\label{eq:bd1}
 - \left(W_k \ast \rho\right)(r)= -\left( \frac{r^{2-N}}{(2-N)} \ast \rho\right)(1)+\int_r^1\partial_r \left(\frac{r^{2-N}}{(2-N)} \ast \rho\right)(s) \, ds\, ,
\end{equation}
where $-\left( \tfrac{r^{2-N}}{(2-N)} \ast \rho\right)(1)$ is bounded using \eqref{constest}. To controle $\int_r^1\partial_r \left(\tfrac{r^{2-N}}{(2-N)} \ast \rho\right)(s) \, ds$, we use Newton's Shell Theorem implying
\begin{equation*}
 \partial_r \left(\frac{r^{2-N}}{(2-N)} \ast \rho\right)(s) = \frac{\sigma_N M(s)}{|\partial B(0,s)|}=M(s)s^{1-N}\, ,
\end{equation*}
where we denote by $M(s)=\sigma_N \int_0^s \rho(t)t^{N-1}\, dt$ the mass of $\rho$ in $B(0,s)$. Note that this is precisely expression \eqref{derivpsi} we obtained in the previous case, choosing $\psi_k(s)=1$ for $s<1$ and $\psi_k=0$ for $s>1$ with a jump singularity at $s=1$ (see also \eqref{Newtonianpsi} in the Appendix). By our induction assumption \eqref{ind1}, we have
$$
M(s)\lesssim \sigma_N \int_0^s (1+t^p) t^{N-1}\, dt
= \sigma_N \left(\frac{s^N}{N}+\frac{s^{N+p}}{N+p}\right)\, , \qquad s \in (0,1)\, ,
$$
and hence if $p\neq -2$, then
$$
\int_r^1\partial_r \left(\frac{r^{2-N}}{(2-N)} \ast \rho\right)(s) \, ds
\lesssim  \frac{1}{2N}\left(1-r^2\right) + \frac{1}{(N+p)(p+2)}\left(1-r^{p+2}\right)\, .
$$
If $p=-2$, we obtain instead
$$
\int_r^1\partial_r \left(\frac{r^{2-N}}{(2-N)} \ast \rho\right)(s) \, ds
\lesssim \frac{1}{2N}\left(1-r^2\right) - \frac{1}{(N+p)}\log(r)\, .
$$
Substituting into the right-hand side of \eqref{eq:bd1} yields for all $ r \in (0,1)$
$$
-\left(W_k \ast \rho\right)(r) \lesssim 
\begin{cases}
 1+r^{p+2} &\text{if} \, \, p \neq -2\, ,\\
 1+\left|\log(r)\right| &\text{if} \, \, p = -2\, .\\
\end{cases}
$$
Thanks again to the Euler-Lagrange condition \eqref{eq:min} if $\rho$ is a global minimiser of $\mF_k$, or thanks to condition \eqref{eq:min resc} if $\rho$ is a global minimiser of $\mFr$ instead, we arrive in both cases at \eqref{ind2}.\\

%%%%%%%%%%%%%%%%%%%%%%%%%%%%%%%%%%%%%%%%%%%%
%%%%%%%%%%%%%%%%%%%%%%%%%%%%%%%%%%%%%%%%%%%%
%%%%%%%%%%%%%%%%%%%%%%%%%%%%%%%%%%%%%%%%%%%%
\fbox{Case 3: $-1<k<0$, $N=1$}\\

In one dimension, we can calculate everything explicitly. Since the mean-field potential gradient is a singular integral, we have
\begin{align*}
 \partial_x S_k(x) 
= &\int_\RR \frac{x-y}{|x-y|^{2-k}}\left(\rho(y)-\rho(x)\right)\, dy\\
= &\lim_{\delta \to 0} \int_{|x-y|>\delta} \frac{x-y}{|x-y|^{2-k}}\rho(y)\, dy
=  \frac{x}{r}  \partial_r S_k(r) 
\end{align*}
with the radial component for $r \in (0,R)$ given by 
\begin{align*}
 \partial_r S_k(r) 
 = &\int_0^\infty \left(\frac{r-\eta}{|r-\eta|^{2-k}} + \frac{r+\eta}{|r+\eta|^{2-k}}\right) \left(\rho(\eta)-\rho(r)\right)\,d\eta\\
 = &\int_0^\infty  \frac{r+\eta}{|r+\eta|^{2-k}}\rho(\eta)\,d\eta
 + \lim_{\delta \to 0}\int_{|r-\eta|>\delta}\frac{r-\eta}{|r-\eta|^{2-k}}\rho(\eta)\,d\eta\\
 = &r^{k-1}\int_0^\infty \psi_1\left(\frac{\eta}{r}\right)\rho(\eta)\,d\eta
 + r^{k-1}\lim_{\delta \to 0}\int_{|r-\eta|>\delta}\psi_2\left(\frac{\eta}{r}\right)\rho(\eta)\,d\eta\\
 =&r^{k-1}\int_0^R \psi_1\left(\frac{\eta}{r}\right)\rho(\eta)\,d\eta
 + r^{k-1}\lim_{\delta \to 0}\left(\int_0^{r-\delta}+\int_{r+\delta}^R \right)\psi_2\left(\frac{\eta}{r}\right)\rho(\eta)\,d\eta
\end{align*}
where
$$
\psi_1(s):=\frac{1+s}{|1+s|^{2-k}}=(1+s)^{k-1}, \qquad
\psi_2(s):=\frac{1-s}{|1-s|^{2-k}}=
\begin{cases}
 (1-s)^{k-1} &\text{if}\, \, 0\leq s<1\, ,\\
 -(s-1)^{k-1} &\text{if}\, \, s>1
\end{cases}
% \psi_k(t)=\frac{1-t}{|1-t|^{2-k}} + \frac{1+t}{|1+t|^{2-k}}
$$
are well defined on $[0,1)\cap (1,\infty)$. Define $\gamma:= \min\{1,R/2\}$. Since $\psi_2(s)<0$ for $s>1$ and $\rho$ radially symmetric decreasing, we can estimate the last term for any $r \in (0,\gamma)$ and small $\delta>0$ by
\begin{align*}
  &r^{k-1}\left(\int_0^{r-\delta}+\int_{r+\delta}^R \right)\psi_2\left(\frac{\eta}{r}\right)\rho(\eta)\,d\eta\\
 &\qquad\leq
  r^{k-1}\int_0^{r-\delta}\psi_2\left(\frac{\eta}{r}\right)\rho(\eta)\,d\eta
  +r^{k-1}\rho(2r)\int_{r+\delta}^{2r}\psi_2\left(\frac{\eta}{r}\right)\,d\eta\\
  &\qquad=r^{k}\int_0^{1-\delta/r}\psi_2\left(s\right)\rho(rs)\,ds
  +r^{k}\rho(2r)\int_{1+\delta/r}^{2}\psi_2\left(s\right)\,ds\, .
\end{align*}
Under assumption \eqref{ind1}, we can bound the above expression by
\begin{align*}
 \partial_r S_k(r) 
 &\lesssim\,
 r^{k} \int_0^{R/r}\psi_1(s)\, ds + r^{k+p} \int_0^{R/r}\psi_1(s)s^p\, ds\\
 &+ r^{k}\lim_{\delta \to 0}\int_0^{1-\delta/r} \psi_2(s)\, ds+ r^{k+p}\lim_{\delta \to 0}\int_0^{1-\delta/r} \psi_2(s)s^p\, ds\\
 &+r^{k}\left(1+r^p\right)\lim_{\delta \to 0}\int_{1+\delta/r}^{2}\psi_2\left(s\right)\,ds\\
 &:= r^k \lim_{\delta \to 0} C_1(r,\delta) + r^{k+p}\lim_{\delta \to 0}C_2(r,\delta)\, ,
\end{align*}
where we defined
\begin{align*}
 &C_1(r,\delta):=\int_0^{R/r}\psi_1(s)\, ds +\int_0^{1-\delta/r} \psi_2(s)\, ds+\int_{1+\delta/r}^{2}\psi_2\left(s\right)\,ds\, ,\\
 &C_2(r,\delta):=\int_0^{R/r}\psi_1(s)s^p\, ds +\int_0^{1-\delta/r} \psi_2(s)s^p\, ds+\int_{1+\delta/r}^{2}\psi_2\left(s\right)\,ds\, .
\end{align*}
Next, let us show that the functions $\lim_{\delta \to 0} C_1(r,\delta)$ and $\lim_{\delta \to 0}C_2(r,\delta)$ can be controlled in terms of $r$. The function $\psi_2$ has a non-integrable singularity at $s=1$, however, we can seek compensations from below and above the singularity. One can compute directly that 
\begin{align*}
 C_1(r,\delta)
 :=&\frac{1}{k}\left[\left(\frac{R}{r}+1\right)^k-1\right]
 +\frac{1}{k}\left[1-\left(\frac{\delta}{r}\right)^k\right]
 +\frac{1}{k}\left[\left(\frac{\delta}{r}\right)^k-1\right]\\
 = &\frac{1}{k}\left(\left(\frac{R}{r}+1\right)^k-1\right)
 \leq -\frac{1}{k}\, ,\\
  C_2(r,\delta)
 :=&\left[\left(\frac{R}{r}+1\right)^{k-1}\left(\frac{R}{r}\right)^{p+1}\right]
 -\frac{1}{k}\left[\left(\frac{\delta}{r}\right)^k\right]
 +\frac{1}{k}\left[\left(\frac{\delta}{r}\right)^k-1\right]\\
 = &\left(\frac{R}{r}+1\right)^{k-1}\left(\frac{R}{r}\right)^{p+1} -\frac{1}{k}
 \leq \left(\frac{R}{r}\right)^{k-1}\left(\frac{R}{r}\right)^{p+1} -\frac{1}{k}
 =R^{k+p}r^{-k-p} -\frac{1}{k}\, ,
\end{align*}
so that we obtain the estimate
\begin{align*}
 \partial_r S_k(r) 
 &\lesssim\,1+r^k+r^{k+p}\, .
\end{align*}
Finally, we have for all $r \in (0,\gamma)$:
\begin{align*}
- \left(W_k \ast \rho\right)(r)=&-\left(W_k \ast \rho\right)(\gamma) + \int_r^\gamma \partial_r S_k(s)\, ds
\lesssim \,
1+ \int_r^\gamma \left(s^{k}+s^{p+k}\right)\, ds\, ,
\end{align*}
where we made again use of estimate \eqref{constest}. If $p\neq -k-1$, we have
\begin{align*}
- \left(W_k \ast \rho\right)(r)\lesssim \,
&1+ \frac{\gamma^{k+1}-r^{k+1}}{k+1} + \frac{\gamma^{p+k+1}-r^{p+k+1}}{p+k+1}
\lesssim\,
1+ r^{p+k+1}\, .
\end{align*}
If $p=-k-1$ however, we obtain
\begin{align*}
- \left(W_k \ast \rho\right)(r)\lesssim \,
&1+ \frac{\gamma^{k+1}-r^{k+1}}{k+1} +\log(\gamma)-\log(r)
\lesssim\,
1+|\log(r)|\, .
\end{align*}
Using again the Euler-Lagrange condition \eqref{eq:min} for a global minimiser of $\mF_k$ or \eqref{eq:min resc} for a global minimiser of $\mFr$ respectively, we obtain \eqref{ind2} in the one dimensional case.\\
%%%%%%%%%%%%%%%%%%%%%%%%%%%%%%%%%%%%%%%%%%%%
%%%%%%%%%%%%%%%%%%%%%%%%%%%%%%%%%%%%%%%%%%%%
%%%%%%%%%%%%%%%%%%%%%%%%%%%%%%%%%%%%%%%%%%%

\fbox{Case 4: $-N<k\leq1-N$, $N\geq 2$}\\

In this case, we can again use hypogeometric functions, but here the mean-field potential gradient is a singular integral due to the singularity properties of hypogeometric functions. It writes as
\begin{align*}%\label{Case3derivS1}
 \nabla S_k (x) 
% = &\int_{\RR^N} \frac{x-y}{|x-y|^{2-k}}\left(\rho(y)-\rho(x)\right)\, dy\notag
= \lim_{\delta \to 0} \int_{|x-y|>\delta} \frac{x-y}{|x-y|^{2-k}}\rho(y)\, dy
=  \frac{x}{r}  \partial_r S_k(r) 
\end{align*}
with the radial component given by
\begin{align*}%\label{Case3derivS2}
 \partial_r S_k(r) 
%  =&r^{k-1}\int_0^\infty \psi_k\left(\frac{\eta}{r}\right)\left(\rho(\eta)-\rho(r)\right)\eta^{N-1}\, d\eta\notag\\
 =&r^{k-1} \lim_{\delta \to 0}\int_{|r-\eta|>\delta}\psi_k\left(\frac{\eta}{r}\right)\rho(\eta)\eta^{N-1}\,d\eta\notag\\
 = &r^{k-1}\lim_{\delta \to 0}\left(\int_0^{r-\delta}+\int_{r+\delta}^R \right)\psi_k\left(\frac{\eta}{r}\right)\rho(\eta)\eta^{N-1}\,ds\, ,
\end{align*}
where $\psi_k$ is given by \eqref{psi2} on $[0,1)\cap (1,\infty)$, and we used the fact that $\rho$ is compactly supported. In this regime, the singularity at $s=1$ is non-integrable and has to be handled with care. Define $\gamma:= \min\{1,R/2\}$. Since $\psi_2(s)<0$ for $s>1$ (see Appendix Lemma \ref{lem:psibehaviour}) and since $\rho$ is radially symmetric non-increasing, we can estimate the second integral above for any $r \in (0,\gamma)$ and small $\delta>0$ by
\begin{align*}
  r^{k-1}\int_{r+\delta}^R \psi_k\left(\frac{\eta}{r}\right)\rho(\eta)\eta^{N-1}\,d\eta
 \leq
  &r^{k-1}\rho(2r)\int_{r+\delta}^{2r}\psi_k\left(\frac{\eta}{r}\right)\eta^{N-1}\,d\eta\\
 =&r^{N+k-1}\rho(2r)\int_{1+\delta/r}^{2}\psi_k\left(s\right)s^{N-1}\,ds\, .
\end{align*}
%%%%%%%%%%%
Under assumption \eqref{ind1}, we can then bound the above expression by
\begin{align*}
 \partial_r \left(W_k\ast \rho\right)(r) 
 \lesssim\, 
&r^{N+k-1}\left(\lim_{\delta \to 0}\int_0^{1-\delta/r} \psi_k(s)s^{N-1}\, ds \right)\\
&+ r^{p+N+k-1}\left(\lim_{\delta \to 0}\int_0^{1-\delta/r} \psi_k(s)s^{p+N-1}\, ds \right)\\
 &+ r^{N+k-1}\left(1+r^p\right) \lim_{\delta \to 0}\int_{1+\delta/r}^{2} \psi_k\left(s\right)s^{N-1}\,ds\\
= &r^{N+k-1}\lim_{\delta \to 0}C_1(r,\delta) +  r^{p+N+k-1}\lim_{\delta \to 0}C_2(r,\delta)\, ,
\end{align*}
where
\begin{align*}
 &C_1(r,\delta)=\int_0^{1-\delta/r} \psi_k(s)s^{N-1}\, ds
 + \int_{1+\delta/r}^{2} \psi_k\left(s\right)s^{N-1}\,ds\, ,\\
 &C_2(r,\delta)=\int_0^{1-\delta/r} \psi_k(s)s^{p+N-1}\, ds
 +\int_{1+\delta/r}^{2} \psi_k\left(s\right)s^{N-1}\,ds\, .
\end{align*}
The crucial step is again to show that $\lim_{\delta \to 0}C_1(r,\delta)$ and $\lim_{\delta \to 0}C_2(r,\delta)$ are well-defined and can be controlled in terms of $r$, seeking compensations from above and below the singularity at $s=1$. 
Recalling that $\psi_k(s)s^{N-1}$ and $\psi_k(s)s^{p+N-1}$ are integrable at zero by Lemma \ref{Beh0} and at any finite value above $s=1$ by continuity, we see that the lower bound $0$ and upper bound $2$ in the integrals only contribute constants, independent of $r$ and $\delta$. The essential step is therefore to check integrability close to the singularity $s=1$. From \eqref{upperlim1-} and \eqref{upperlim1+} in Lemma \ref{Beh1}(2) in the Appendix, we have for any $\alpha \in \RR$ and $s$ close to 1:
 \begin{align*}
 &s<1:
 \qquad 
 \psi_k\left(s\right)s^\alpha
 =K_1 \left(1-s\right)^{N+k-2} +O\left( \left(1-s\right)^{N+k-1}\right)\, ,\\
 &s>1:
 \qquad 
 \psi_k\left(s\right)s^\alpha
 = -K_1 \left(s-1\right)^{N+k-2} +O\left( \left(s-1\right)^{N+k-1}\right)\, ,
\end{align*}
where the constant $K_1$ is given by \eqref{K1K2}--\eqref{gamma}. Hence, for $-N<k<1-N$ we obtain
\begin{align*}
 C_1(r,\delta)
 &\lesssim 1 -\frac{K_1}{N+k-1} \left(\frac{\delta}{r}\right)^{N+k-1} 
 +\frac{K_1}{N+k-1} \left(\frac{\delta}{r}\right)^{N+k-1} 
 +O\left(\left(\frac{\delta}{r}\right)^{N+k}\right)\\
 &=1  +O\left(\left(\frac{\delta}{r}\right)^{N+k}\right)
\end{align*}
with exactly the same estimate for $C_2(r,\delta)$. Taking the limit $\delta \to 0$, we see that both terms are bounded by a constant. For $k=1-N$, we obtain similarly that both $ C_1(r,\delta)$ and $ C_2(r,\delta)$ are bounded by 
\begin{align*}
1 - K_1 \log\left(\frac{\delta}{r}\right)+K_1\log\left(\frac{\delta}{r}\right)
 +O\left(\left(\frac{\delta}{r}\right)\right)
=1  +O\left(\left(\frac{\delta}{r}\right)\right)
\end{align*}
multiplied by some constant.
In other words, for any $r \in (0,\gamma)$ and $-N<k\leq 1-N$ we have
\begin{align*}
 \partial_r \left(W_k\ast \rho\right)(r) 
 \lesssim\, r^{N+k-1} + r^{p+N+k-1}\, .
\end{align*}
Now, we are ready to estimate the behaviour of $\rho$ around the origin using again the Euler-Lagrange condition. To estimate the mean-field potential, we use again \eqref{constest} and write
\begin{align*}
- \left(W_k \ast \rho\right)(r)=&-\left(W_k \ast \rho\right)(\gamma) + \int_r^\gamma \partial_r\left(W_k\ast \rho\right)(s)\, ds
\lesssim \,
1+ \int_r^\gamma \left(s^{k+N-1}+s^{p+k+N-1}\right)\, ds
\end{align*}
for any $r \in (0,\gamma)$.
%%%%%%%%%%%%%
It follows that for $p\neq -k-N$,
\begin{align*}
- \left(W_k \ast \rho\right)(r)\lesssim \,
&1+ \frac{\gamma^{k+N}-r^{k+N}}{k+N} + \frac{\gamma^{p+k+N}-r^{p+k+N}}{p+k+N}
\lesssim\,
1+ r^{p+k+N}\, .
\end{align*}
If $p=-k-N$, we have instead
\begin{align*}
- \left(W_k \ast \rho\right)(r)\lesssim \,
&1+ \frac{\gamma^{k+N}-r^{k+N}}{k+N} +\log(\gamma)-\log(r)
\lesssim\,
1+ |\log(r)|\, .
\end{align*}
This concludes the proof of Lemma \ref{lem:induction} using again Euler-Lagrange condition \eqref{eq:min} if $\rho$ is a minimiser of $\mF_k$, or condition \eqref{eq:min resc} if $\rho$ is a minimiser of $\mFr$, to obtain \eqref{ind2}.
\end{proof}

%%%%%%%%%%%%%%%%%%%%%%%%%%%%%
%%%%%%%%%%%%%%%%%%%%%%%%%%%%%
%%%%%%%%%%%%%%%%%%%%%%%%%%%%%
%%%%%%%%%%%%%%%%%%%%%%%%%%%%%BOUNDEDNESS
\begin{corollary}[Boundedness]\label{cor:boundedness_min}
If $\chi=\chi_c$ and $\rho $ is a global minimiser of $\mF_k$ over $\mY$, then $\rho \in L^\infty\left(\RR^N\right)$. If $0<\chi<\chi_c$ and $\rho $ is a global minimiser of $\mFr$ over $\mY_2$, then $\rho \in L^\infty\left(\RR^N\right)$.
\end{corollary}

\begin{proof}
Let $\rho$ be a global minimiser of either $\mF_k$ over $\mY$, or of $\mFr$ over $\mY_2$.
Since $\rho$ is radially symmetric non-increasing by Proposition \ref{prop:charac min}, it is enough to show that $\rho(0)<\infty$. Following the argument in \cite{CHVY}, we use induction to show that there exists some $\alpha>0$ such that for all $r \in (0,1)$ we have
\begin{equation}\label{ind4}
\rho(r) \lesssim 1+r^\alpha\, .
\end{equation}
Note that $g(p)$ as defined in \eqref{defgp} is a linear function of $p$ with positive slope, and let us denote $g^{(n)}(p)=\left(g \circ g \dotsm \circ g\right)(p)$. Computing explicitly, we have for all $n \in \N$
$$
g^{(n)}(p)= \frac{N+k}{m-2} + \frac{p(m-2)-N-k}{(m-2)(m-1)^n}
=-N+\frac{p+N}{(m-1)^n}\, ,
$$
so that 
$$
\lim_{n \to \infty} g^{(n)}(p) = + \infty \quad \text{for any} \, \, p>-N\, .
$$
Since $\rho(r)^m |B(0,r)| \leq ||\rho||_m^m<\infty$ we obtain the estimate
$$
\rho(r)\leq C(N,m,||\rho||_m)r^{-N/m} \quad \text{for all}\, \, r>0\,.
$$
It follows that $\rho$ satisfies the induction requirement \eqref{ind1} with choice $p_0:=-N/m$. Since $p_0>-N$ there exists $n_0 \in \N$ such that $g^{(n_0)}(p_0)>0$ and so we can apply Lemma \ref{lem:induction} $n_0$ times.
This concludes the proof with $\alpha=g^{(n_0)}(p_0)$. We point out that $p_0<-N-k$ and so there is a possibility that $g^{(n)}(p_0)=-N-k$ might occur for some $0<n\leq n_0$: if this happens, the logarithmic case occurs and by the second bound in \eqref{ind2}, we obtain
$$
\rho(r) \lesssim 1+|\log(r)|^{\frac{1}{m-1}} \leq 1+ r^{-1}\, ,
$$
hence applying the first bound in \eqref{ind2} for $p=-1$ yields \eqref{ind4} with $\alpha=1/(m-1)$.
\end{proof}
%%%%%%%%%%%%%%%%%%%%%%%%%%%
%%%%%%%%%%%%%%%%%%%%%%%%%%%
%%%%%%%%%%%%%%%%%%%%%%%%%%%

\begin{corollary}[Regularity]\label{cor:regmin}
 If $\chi=\chi_c$, then all global minimisers $\rho\in\mY$ of $\mF_k$ satisfy $S_k \in \mathcal{W}^{1,\infty}\left(\RR^N\right)$ and $\rho^{m-1} \in \mathcal{W}^{1,\infty}\left(\RR^N\right)$. If $0<\chi<\chi_c$, then all minimisers  $\rho\in\mY_2$ of $\mFr$ satisfy $S_k \in \mathcal{W}^{1,\infty}\left(\RR^N\right)$ and $\rho^{m-1} \in \mathcal{W}^{1,\infty}\left(\RR^N\right)$. In the singular range $-N<k\leq 1-N$, we further obtain $\rho \in C^{0,\alpha}\left(\RR^N\right)$ with $\alpha \in (1-k-N,1)$ in both original and rescaled variables.
 \end{corollary}
\begin{proof}
Let $\rho$ be a global minimiser either of $\mF_k$ over $\mY$, or of $\mFr$ over $\mY_2$. Then $\rho \in L^\infty\left(\RR^N\right)$ by Corollary \ref{cor:boundedness_min}. 
Let us start by considering the singular regime $-N<k\leq 1-N$, $N\geq 2$ or $-1<k<0$ for $N=1$. Since $\rho \in L^1\left(\RR^N\right)\cap L^\infty\left(\RR^N\right)$, we have $\rho \in L^p\left(\RR^N\right)$ for any $1< p<\infty$.

Using the fact that $\rho=\left(-\Delta\right)^s S_k$ with fractional exponent $s=(N+k)/2 \in (0,1/2)$, we gain $2s$ derivatives implying $S_k \in \mathcal{W}^{2s,p}\left(\RR^N\right)$ for $p\ge 2$ if $N\geq 2$ or for $p>-1/k$ if $N=1$ by using the HLS inequality for Riesz kernels, see \cite[Chapter V]{Stein}. By classical Sobolev embedding, $\mathcal{W}^{2s,p}\left(\RR^N\right) \subset C^{0,\beta}\left(\RR^N\right)$ with $\beta = 2s-N/p$ for $p>\tfrac{N}{2s}\geq 2$ if $N\geq2$ or for $p>\max\{\tfrac{1}{k+1},-\tfrac{1}{k}\}$ if $N=1$, which yields $\rho \in C^{0,\beta}\left(\RR^N\right)$. In the case $1/2-N<k\leq 1-N$, we can ensure $\beta>1-k-N$ for large enough $p$, obtaining the required H\"{o}lder regularity. For $k\leq 1/2-N$ on the other hand, we need to bootstrap a bit further. Let us fix $n \in \N$, $n\geq 2$ such that 
 $$
 \frac{1}{n+1}-N<k\leq \frac{1}{n}-N
 $$
 and let us define $\beta_n:=\beta+(n-1)2s=n2s-N/p$. Note that $S_k \in L^\infty\left(\RR^N\right)$ by Lemma \ref{lem:regS}, and $\beta_{n-1}+2s<1$. This allows us to repeatedly apply  \cite[Proposition 2.8]{Silvestre} stating that $\rho \in C^{0,\gamma}\left(\RR^N\right)$ implies $S_k \in C^{0,\gamma+2s}\left(\RR^N\right)$ for any $\gamma \in (0,1]$ such that $\gamma+2s<1$. It then follows that $\rho^{m-1} \in C^{0,\gamma+2s}\left(\RR^N\right)$ using the Euler-Lagrange conditions \eqref{eq:min} and \eqref{eq:min resc} respectively and Corollary \ref{cor:compactsupp_min}. Since $m \in (1,2)$, we conclude $\rho \in C^{0,\gamma+2s}\left(\RR^N\right)$. Iterating this argument $(n-1)$ times starting with $\gamma=\beta$, we obtain $\rho \in C^{0,\beta_n}\left(\RR^N\right)$ and choosing $p$ large enough, we have indeed
 $$
 \beta_n>1-k-N\, .
 $$
For any $-N<k<0$, we then have $S_k \in \mathcal{W}^{1,\infty}\left(\RR^N\right)$ by Lemma \ref{lem:regS}. It also immediately follows that $\rho^{m-1} \in \mathcal{W}^{1,\infty}\left(\RR^N\right)$ using the Euler-Lagrange conditions \eqref{eq:min} and \eqref{eq:min resc} respectively, Corollary \ref{cor:compactsupp_min} and Lemma \ref{lem:regS}. Since $m\in (1,2)$, we also conclude $\rho \in \mathcal{W}^{1,\infty}\left(\RR^N\right)$. 
\end{proof}

\begin{remark}
For proving sufficient H\"{o}lder regularity in the singular regime $-N<k\leq 1-N$, one may choose to bootstrap on the fractional Sobolev space $\mathcal{W}^{2s,p}\left(\RR^N\right)$ directly, making use of the Euler-Lagrange conditions \eqref{eq:min} and \eqref{eq:min resc} respectively to show that $\rho \in \mathcal{W}^{r,p}\left(\RR^N\right) \Rightarrow S_k \in  \mathcal{W}^{r+2s,p}\left(\RR^N\right)$ with $r>0$ for $p$ large enough depending only on $N$, see \cite[Chapter V]{Stein}. Here, we need that $ \mathcal{W}^{r,p}\left(\RR^N\right)$ is preserved under taking positive parts of a function for $0<r\leq 1$ and compositions with Lipschitz functions since we take the $1/(m-1)$ power of $\rho$, see \cite[Section 3.1]{Sickel}.
\end{remark}

%%%%%%%%%%%%%%%%%%%%%%%%%%%%%%%%%%%%%%%%%%%%%%%%%%%%%%%%%%%%
%%%%%%%%%%%%%%%%%%%%%%%%%%%%%%%%%%%%%%%%%%%%%%%%%%%%%%%%%%%%
\begin{theorem}[Global Minimisers as Stationary States]\label{thm:minsstates}
 If $\chi=\chi_c$, then all global minimisers of $\mF_k$ are stationary states of equation \eqref{eq:KS}. If $0<\chi<\chi_c$, then all global minimisers of $\mFr$ are stationary states of the rescaled equation \eqref{eq:KSresc}.
\end{theorem}
\begin{proof}
 For $\chi=\chi_c$, let $\rho \in \mY$ be a global minimiser of $\mF_k$. The regularity properties provided by Corollary \ref{cor:regmin} imply that $\nabla \rho^m =\frac{m}{m-1}\rho \nabla \rho^{m-1}$ and that $\rho$ is indeed a distributional solution of \eqref{eq:steady} using \eqref{eq:min}. As a consequence, $\rho$ is a stationary state of equation \eqref{eq:KS} according to Definition \ref{def:sstates}. A similar argument holds true in the rescaled case for sub-critical $\chi$.
\end{proof}

\begin{remark}\label{rmk:radiality}
As a matter of fact, the recent result of radial symmetry of stationary states \cite{CHVY} applies to the critical case $\chi=\chi_c$ in the range $k\in [2-N,0)$. Together, Theorem \ref{thm:sstatesmin} and Proposition \ref{prop:charac min} show that all stationary states are radially symmetric for the full range $k\in (-N,0)$. In other words, the homogeneity of the energy functional $\mF_k$ allows us to extend the result in \cite{CHVY} to $k\in (-N,2-N)$ and to find a simple  alternative proof in the less singular than Newtonian range.
\end{remark}

%%%%%%%%%%%%%%%%%%%%%%%%%%%%%%%
%%%%%%%%%%%%%%%%%%%%%%%%%%%%%%%
%%%%%%%%%%%%%%%%%%%%%%%%%%%%%%%
%%%%%%%%%%%%%%%%%%%%%%%%%%%%%%%
\section{Fast Diffusion Case $k>0$}
\label{sec:Fast Diffusion Case k pos}
%%%%%%%%%%%%%%%%%%%%%%%%%%%%%%%
%%%%%%%%%%%%%%%%%%%%%%%%%%%%%%%
%%%%%%%%%%%%%%%%%%%%%%%%%%%%%%%
%%%%%%%%%%%%%%%%%%%%%%%%%%%%%%%

We investigate in this section the case $k \in (0,N)$ and hence $m\in (0,1)$ where the diffusion is fast in regions where the density of particles is low. The main difficulty is that it seems there is no HLS-type inequality in this range which would provide a lower bound on the free energy, and so a different approach is needed than in the porous medium regime. We concentrate here on the radial setting.
Let us define $\mX$ to be the set 
 \begin{equation*}
 \mX := \left\{ \rho \in L_+^1\left(\RR^N\right): ||\rho||_1=1\, , \, \int x\rho(x)\, dx=0 \right\}\, .
 \end{equation*}
 The following Lemma will be a key ingredient for studying the behaviour in the fast diffusion case.
 \begin{lemma}\label{lem:kthmomentbound}
  For $k \in (0,N)$, any radially symmetric non-increasing $\rho \in \mX$ with $|x|^k\rho \in L^1\left(\RR^N\right)$ satisfies
  \begin{equation}\label{eq:kthmomentbound}
   I_k[\rho] \leq \frac{|x|^k}{k} \ast \rho \leq \eta \left(\frac{|x|^k}{k} + I_k[\rho]\right)\,
  \end{equation}
with  
  $$
  I_k[\rho]:=\int_{\RR^N} \frac{|x|^k}{k}\rho(x) \, dx, \qquad
  \eta = \max\{1,2^{k-1}\}.
  $$
 \end{lemma}
\begin{proof}The bound from above was proven in \eqref{eq:kconvex}. To prove the lower bound in one dimension, we use the symmetry and monotonicity assumption to obtain
\begin{equation*}
\partial_x \left(\frac{|x|^k}{k} \ast \rho  \right) = \frac{1}{k}\int_{y>0} \left(|x-y|^k-|x+y|^k\right)\partial_y \rho\, dy \quad \geq 0
\end{equation*}
since $|x-y|^k-|x+y|^k \leq 0$ for $x, y \geq 0$. Hence the above integral is non-negative and \eqref{eq:kthmomentbound} holds true in one dimension for the bound from below.\\
For $N\ge 2$, note that
since both $W_k(x)=|x|^k/k$ and $\rho$ are radial functions, so is the convolution $W_k \ast \rho$. By slight abuse of notation, we write $(W_k \ast \rho) (r)$. 
For $r>0$, we have
\begin{align*}
 \int_{B(0,r)} \Delta_x\left(W_k \ast \rho\right)\, dx
& = \int_{\partial B(0,r)} \nabla_x\left(W_k \ast \rho\right) \cdot \overline{n}\, dS \\
 &= |\partial B(0,r)|\, \partial_r\left(W_k \ast \rho\right) 
 = r^{N-1} \sigma_N \, \partial_r\left(W_k \ast \rho\right) .
\end{align*}
From 
$
\Delta_x W_k (x)=\left(N+k-2\right) |x|^{k-2} >0
$,
it then follows that $\partial_r\left(W_k \ast \rho\right)(r) >0$ for all $r>0$. This implies the lower bound in higher dimensions.
\end{proof}

%%%%%%%%%%%%%%%%%%%%%%%%%%%%%%%
%%%%%%%%%%%%%%%%%%%%%%%%%%%%%%%
\subsection{Results in Original Variables}
\label{sec:Results in Original Variables}
%%%%%%%%%%%%%%%%%%%%%%%%%%%%%%%
%%%%%%%%%%%%%%%%%%%%%%%%%%%%%%%

\begin{theorem}[Non-Existence of Stationary States]\label{thm:noEsstates}
 Let $k \in (0,N)$. For any $\chi>0$, there are no radially symmetric non-increasing stationary states in $\mX$ for equation \eqref{eq:KS} with $k$th moment bounded.
\end{theorem}

 \begin{proof} 
 Assume $\bar\rho \in \mX$ is a radially symmetric non-increasing stationary state for equation \eqref{eq:KS} such that $|x|^k\bar\rho\in L^1\left(\RR^N\right)$. Then $\bar \rho$ is continuous by Lemma \ref{lem:sstatesreg}. We claim that $\bar \rho$ is supported on $\RR^N$ and satisfies 
\begin{equation}\label{eq:sstates kpos}
  \bar\rho(x) = \left(A W_k\ast \bar\rho(x) + C[\bar \rho]\right)^{-N/k}, \quad a.e.\, \, x \in \RR^N\, ,
 \end{equation}
 with $A:=2 \chi Nk/(N-k) >0$ and some suitably chosen constant $C[\bar \rho]$. Indeed, by radiality and monotonicity, $\supp(\bar \rho)=B(0,R)$ for some $R \in (0,\infty]$ and by the same arguments as in Corollary \ref{cor:EL} leading to \eqref{eq:steady3}, we obtain 
 \begin{equation*}%\label{eq:sstates kpos}
  \bar\rho(x)^{-k/N} = A W_k\ast \bar\rho(x) + C[\bar \rho], \quad a.e.\, \, x \in B(0,R)\, .
 \end{equation*}
 Assume $\bar\rho$ has compact support, $R<\infty$. It then follows from Lemma \ref{lem:kthmomentbound} that the left-hand side is bounded above,
 \begin{equation*}%\label{eq:sstates kpos}
  \bar\rho(x)^{-k/N} \leq \eta A I_k[\bar \rho] + \frac{\eta A R^k}{k} + C[\bar \rho], \quad a.e.\, \, x \in B(0,R)\, .
 \end{equation*}
 By continuity, $\rho(x)\to 0$ as $|x|\to R$, but then $\rho(x)^{-k/N}$ diverges, contradicting the bound from above. We must therefore have $R=\infty$, which concludes the proof of \eqref{eq:sstates kpos}.\\
 Next, taking the limit $x \to 0$ in \eqref{eq:sstates kpos} yields
 \begin{equation*}
  AI_k[\bar\rho] + C[\bar\rho] >0.
 \end{equation*}
We then have from Lemma \ref{lem:kthmomentbound} for a.e. $x \in \RR^N$,
$$
0
\leq
\left(A\,\eta \, \left(\frac{|x|^k}{k}+ I_k[\bar\rho]\right) + C[\bar\rho]\right)^{-N/k}
\leq
\bar\rho(x) 
\, .
$$
However, the lower bound in the estimate above is not integrable on $\RR^N$, and hence $\bar \rho \notin L^1\left(\RR^N\right)$. This contradics $\bar \rho \in \mX$.
\end{proof}

In the fast diffusion regime, we do not have a suitable HLS-type inequalilty to show boundedness of the energy functional $\mF_k$. Although we do not know whether $\mF_k$ is bounded below or not, we can show that the infimum is not achieved in the radial setting.

\begin{theorem}[Non-Existence of Global Minimisers]\label{thm:noEmins}
 Let $k \in (0,N)$. For any $\chi>0$, there are no radially symmetric non-increasing global minimisers of $\mF_k$ over $\mY_k$.
\end{theorem}

 \begin{proof} 
 Let $\rho$ be a global minimiser of $\mF_k$ over $\mY_k$. Following the same argument as in Proposition \ref{prop:charac min}, we obtain
 \begin{align}
   \rho(x)^{-k/N} &= A W_k\ast \rho(x) + D_k[\rho] \quad \text{a.e. in}\, \supp(\rho)\,,\label{eq:min kpos 1}\\
   \rho(x)^{-k/N} &\geq A W_k\ast \rho(x) + D_k[\rho] \quad \text{a.e. in}\, \RR^N\, .\label{eq:min kpos 2}
 \end{align}
where
$$
D_k[\rho] := -\frac{2Nk}{(N-k)} \mF_k[\rho] - \left(\frac{N+k}{N-k}\right) \int_{\RR^N}\bar\rho^m(x)\, dx\,.
$$
Since $W_k$ is continuous and $\rho \in L^1\left(\RR^N\right)$, it follows from \eqref{eq:min kpos 1} that $\rho$ is continuous inside its support, being a continuous function of $W_k$ convolved with $\rho$. If $\rho$ is radially symmetric non-increasing, then $\supp(\rho)=B(0,R)$ for some $R \in (0,\infty]$. By continuity of $\rho$ at the origin, we can take the limit $|x|\to0$ in \eqref{eq:min kpos 1} to obtain $A I_k[\rho] + D_k[\rho]>0$. It then follows from \eqref{eq:min kpos 2} and \eqref{eq:kthmomentbound} that in fact $\rho(x)^{-k/N} >0$ for a.e. $x \in \RR^N$. Hence, we conclude that $\supp(\rho)=\RR^N$. The Euler-Lagrange condition \eqref{eq:min kpos 1} and estimate \eqref{eq:kthmomentbound} yield
\begin{align*}
\rho(x) 
 =  \left(A \, W_k \ast \rho +D_k[\rho]\right)^{-N/k}
 \geq   \left(A \, \eta\,\left(\frac{|x|^k}{k} + I_k[\rho]\right) + D_k[\rho]\right)^{-N/k}
\end{align*}
a.e. on $\RR^N$. Again, the right-hand side is not integrable for any $k \in (0,N)$ and hence $\rho \notin \mY_k$.
\end{proof}

%%%%%%%%%%%%%%%%%%%%%%%%%%%%%%%%%%
%%%%%%%%%%%%%%%%%%%%%%%%%%%%%%%%%%
%%%%%%%%%%%%%%%%%%%%%%%%%%%%%%%%%%
\subsection{Results in Rescaled Variables}
\label{sec:Results in Rescaled Variables}
%%%%%%%%%%%%%%%%%%%%%%%%%%%%%%%%%%
%%%%%%%%%%%%%%%%%%%%%%%%%%%%%%%%%%
%%%%%%%%%%%%%%%%%%%%%%%%%%%%%%%%%%

\begin{corollary}[Necessary Condition for Stationary States] \label{cor:sstates kpos}
Let $k \in (0,N)$, $\chi>0$ and $\bar \rho \in \mX$.
If $\bar \rho$ is a radially symmetric non-increasing stationary state of the rescaled equation \eqref{eq:KSresc} with bounded $k$th moment, then $\bar \rho$ is continuous, supported on $\RR^N$ and satisfies  
\begin{equation}\label{eq:sstates resc kpos}
  \bar\rho(x) = \left(A W_k\ast \bar\rho(x) + B|x|^2 + C[\bar \rho]\right)^{-N/k}, \quad a.e.\, \, x \in \RR^N\, .
 \end{equation}
 Here, the constant $C[\bar \rho]$ is chosen such that $\bar \rho$ integrates to one and
 \begin{equation}\label{ab}
 A:=2 \chi \frac{Nk}{(N-k)} >0, \quad
 B:=\frac{Nk}{2(N-k)} >0.
 \end{equation}
\end{corollary}

\begin{proof} Continuity follows from Lemma \ref{lem:sstatesreg}, and we can show $\supp(\bar\rho)=\RR^N$ and \eqref{eq:sstates resc kpos} by a similar argument as for \eqref{eq:sstates kpos}.
\end{proof}
From the above analysis, if diffusion is too fast, then there are no stationary states to the rescaled equation \eqref{eq:KSresc}:

\begin{theorem}[Non-Existence of Stationary States]\label{thm:noEsstates resc}
 Let $\chi>0$, $N\geq 3$ and $k\in[2,N)$, then there are no radially symmetric non-increasing stationary states in $\mX$ with $k$th moment bounded to the rescaled equation \eqref{eq:KSresc}.
\end{theorem}
\begin{proof}
 Assume $\bar \rho \in \mX$ is a radially symmetric non-increasing stationary state such that $|x|^k\bar\rho\in L^1\left(\RR^N\right)$. It follows from \eqref{eq:sstates resc kpos} and \eqref{eq:kthmomentbound} that
 \begin{align*}
 \bar\rho(x) \geq
 \left(A \, \eta\,\left(\frac{|x|^k}{k} + I_k[\bar\rho]\right) + B|x|^2+ C[\bar\rho]\right)^{-N/k}\, .
\end{align*}
However, the lowed bound is not integrable on $\RR^N$ for $k\geq 2$, contradicting $\bar \rho \in L^1\left(\RR^N\right)$.
\end{proof}

\begin{remark}
 Condition \eqref{eq:sstates resc kpos} tells us that radially symmetric non-increasing stationary states have so-called fat tails for large $r=|x|$. More precisely, Lemma \ref{lem:kthmomentbound} shows they behave at least like $r^{-N}$ for large $r$ if $k\geq 2$, whereas $\bar\rho(r)\sim r^{-2N/k}$ for large $r>0$ and for $k<2$. This means there is a critical $k_c:=2$ and respectively a critical diffusion exponent $m_c:=(N-2)/N$ where a change of behavior occurs. 
 
 For $k<k_c$, radially symmetric non-increasing stationary states, if they exist, are integrable and mass is preserved. This resctriction on $k$ corresponds exactly to the well-known classical fast diffusion regime $m>m_c$ in the case $\chi=0$ \cite{VazquezFDE}, where mass escapes to the far field but is still preserved. In our case, the behaviour of the tails is dominated by the non-linear diffusion effects even for $\chi>0$ as for the classical fast-diffusion equation when $m>m_c$. 
 
 If diffusion is too fast, i.e. $k>k_c$ and $m<m_c$, then no radially symmetric non-increasing stationary states of the rescaled equation \eqref{eq:KSresc} exist as stated in Theorem \ref{thm:noEsstates resc}. It is well known that mass escapes to infinity in the case of the classical fast diffusion equation ($\chi=0$) and integrable $L^\infty$-solutions go extinct in finite time (for a detailed explanation of this phenomenon, see \cite[Chapter 5.5]{VazquezFDE}). It would be interesting to explore this in our case.
\end{remark}

\begin{remark}  
  If $N\geq 2$ and $k\in[k^*,2)$ with
  \begin{equation}\label{kthmomentcond}
  k^*(N):=  -\dfrac N2+\sqrt{\frac{N^2}{4}+2N} \quad \in [1,2)\, ,
 \end{equation}
 then radially symmetric non-increasing solutions $\bar \rho \in \mX$ to equation \eqref{eq:sstates resc kpos} have unbounded $k$th moment.
  Indeed, assuming for a contradiction that $|x|^k\bar\rho\in L^1 \left(\RR^N\right)$. It then follows from \eqref{eq:sstates resc kpos} and \eqref{eq:kthmomentbound} that 
\begin{align*}
 |x|^k \bar\rho(x) 
 \geq  |x|^k \left(A \, \eta\,\left(\frac{|x|^k}{k} + I_k[\bar\rho]\right) + B|x|^2+ C[\bar\rho]\right)^{-N/k}
\end{align*}
a.e. on $\RR^N$, and the right-hand side is integrable only in the region $k^2+Nk-2N<0$. This condition yields \eqref{kthmomentcond}.
\end{remark}

\begin{proposition}[Necessary Condition for Global Minimisers]\label{prop:charac min kpos}
For $k\in(0,N)$, let $\rho$ be a global minimiser of $\mFr$ in $\mY_{2,k}$.
Then for any $\chi>0$, $\rho$ is continuous inside its support and satisfies
 \begin{align}
   \rho(x)^{-k/N} &= A W_k\ast \rho(x) +B|x|^2+ D_{k,\resc}[\rho] \quad \text{a.e. in}\, \,  \supp(\rho)\,,\label{eq:min resc kpos 1}\\
   \rho(x)^{-k/N} &\geq A W_k\ast \rho(x)+B|x|^2 + D_{k,\resc}[\rho] \quad \text{a.e. in}\, \, \RR^N\, .\label{eq:min resc kpos 2}
 \end{align}
Here, constants $A,B$ are given by \eqref{ab} and
\begin{align*}
D_{k, \resc}[\rho] :=\, -4B \mFr[\rho]+B \mV[\rho]-\left(\frac{N+k}{N-k}\right)\int_{\RR^N}\bar\rho^m(x)\, dx \, .
\end{align*}
Moreover, radially symmetric non-increasing global minimisers in $\mY_{2.k}$ are supported on the whole space, and so in that case \eqref{eq:min resc kpos 1} holds true in $\RR^N$.
\end{proposition}

\begin{proof} The proof of \eqref{eq:min resc kpos 1} and \eqref{eq:min resc kpos 2} follows analogously to Proposition \ref{prop:charac min}. Further, since $W_k$ is continuous and $\rho \in L^1\left(\RR^N\right)$, it follows from \eqref{eq:min resc kpos 1} that $\rho$ is continuous inside its support being a continuous function of the convolution between $W_k$ and $\rho$. Now, if $\rho$ is radially symmetric non-increasing, we argue as for Theorem \ref{thm:noEmins} to conclude that $\supp(\rho)=\RR^N$.
\end{proof}

\begin{remark}\label{rmk:kregimes min}
Just like \eqref{eq:sstates resc kpos}, condition \eqref{eq:min resc kpos 1} provides the behaviour of the tails for radially symmetric non-increasing global minimisers of $\mFr$ using the bounds in Lemma \ref{lem:kthmomentbound}. In particular, they have unbounded $k$th moment for any $\chi>0$ if $k\geq k^*$ with $k^*$ given by \eqref{kthmomentcond}, and they are not integrable for $k>k_c:=2$. Further, their second moment is bounded and $\rho^m\in L^1\left(\RR^N\right)$ if and only if  $k < 2N/(2+N)$. Note that
 $$
 \frac{2N}{2+N} < k^*(N) < k_c\, .
 $$
 Hence, radially symmetric non-increasing global minimisers with finite energy $\mFr[\rho]<\infty$ can only exist in the range $0<k<2N/(2+N)$. For $k\geq \tfrac{2N}{2+N}$, one may have to work with relative entropies instead.
\end{remark}

Apart from the Euler-Lagrange condition above, we have very little information about global minimisers of $\mFr$ in general, and it is not known in general if solutions to \eqref{eq:min resc kpos 1}-\eqref{eq:min resc kpos 2} exist. Thus, we use a different approach here than in the porous medium regime, showing existence of stationary states to \eqref{eq:KSresc} directly by a compactness argument.
%%%%%%%%%%%%%%%%%%%%%%%%%%%%%
%%%%%%%%%%%%%%%%%%%%%%%%%%%%%
%%%%%%%%%%%%%%%%%%%%%%%%%%%%%
%%%%%%%%%%%%%%%%%%%%%%%%%%%%%
Let us define the set
 \begin{equation*}
 \bar\mX := \left\{ \rho \in C(\RR^N) \cap \mX :  \, \int|x|^k\rho(x)\, dx<\infty\, ,\, \rho^*=\rho \, ,\, \lim_{r\to \infty} \rho(r) =0 \right\}\, ,
 \end{equation*}
where $\rho^*$ denotes the symmetric decreasing rearrangement of $\rho$.
\begin{theorem}[Existence of Stationary States]\label{thm:existence sstates}
 Let $\chi>0$ and $k \in (0,1] \cap (0,N)$. Then there exists a stationary state $\bar \rho \in \bar \mX$ for the rescaled system \eqref{eq:KSresc}.
\end{theorem}

Here, decay at infinity of the equilibrium distribution is a property we gain automatically thanks to the properties of the equation, but we choose to include it here a priori. Note that the restriction $k \leq 1$ arises from Lemma \ref{lem:kthmomentbound} as we need the upper and lower bounds in \eqref{eq:kthmomentbound} to scale with the same factor ($\eta=1$). By Corollary \ref{cor:sstates kpos}, this restriction on $k$ also means that we are in the range where stationary states have bounded $k$th moment since $(0,1] \cap (0,N) \subset (0,k^*)$.

\begin{proof}%[Proof of Theorem \ref{thm:existence sstates}]
 Corollary \ref{cor:sstates kpos} suggests that we are looking for a fixed point  of the operator $T: \mX \to \mX$,
 \begin{equation*}\label{opT}
  T \rho (x) :=  \left(A \frac{|x|^k}{k} \ast \rho + B|x|^2 + C\right)^{-N/k}\, .
 \end{equation*}
 For this operator to be well-defined, we need to be able to choose a constant $C=C[\rho]$ such that $\int_\RR T \rho(x)\, dx=1$. To show that this is indeed the case, let us define for any $\alpha >0$,
 \begin{align*}
 w(\alpha):= \int_{\RR^N} \left(\alpha + A \frac{|x|^k}{k} + B|x|^2\right)^{-N/k}\, dx\, , \qquad
 W(\alpha):= \int_{\RR^N} \left(\alpha + B|x|^2\right)^{-N/k}\, dx\, .
 \end{align*}
 Note that $w$ and $W$ are finite and well-defined since $k <2$. Furthermore, both $w$ and $W$ are continuous, strictly decreasing to zero as $\alpha$ increases, and blow-up at $\alpha=0$. Hence, we can take inverses $\underline{\delta}:= w^{-1}(1)>0$ and $\overline{\delta}:=W^{-1}(1)>0$. Fixing some $\rho \in \bar \mX$ and denoting by $M\left(\rho,C\right)$ the mass of $T\rho$, 
 we obtain from Lemma \ref{lem:kthmomentbound},
 $$
 M\left(\rho,\underline{\delta}-AI_k[\rho]\right)\geq 1\, , \qquad
 M\left(\rho,\overline{\delta}-AI_k[\rho]\right) \leq 1\, .
 $$
 Since $M\left(\rho,\cdot\right)$ is continuous and strictly decreasing on the interval $\left[\underline{\delta}-AI_k[\rho] , \overline{\delta}-AI_k[\rho]\right]$, we conclude that there exists $C[\rho]$ with $\underline{\delta}-A I_k[\rho] \leq C[\rho] \leq\overline{\delta}-A I_k[\rho]$ and $M\left(\rho,C[\rho]\right)=1$. From Lemma \ref{lem:kthmomentbound}, we obtain for all $x \in \RR^N$,
 $$
 \left(A I_k[\rho] + C[\rho]+ A \frac{|x|^k}{k} + B|x|^2\right)^{-N/k}
 \leq T\rho(x) \leq 
 \left(A I_k[\rho] + C[\rho]+ B|x|^2\right)^{-N/k},
 $$
 and integrating over $\RR^N$,
 $$
 w\left(A I_k[\rho] + C[\rho]\right)\leq 1 \leq W\left(A I_k[\rho] + C[\rho]\right)\, ,
 $$
 implying
 \begin{equation}\label{deltabound}
  0< \underline{\delta} \leq A I_k[\rho] + C[\rho] \leq\overline{\delta}< \infty.
 \end{equation}
As a consequence, we have a pointwise estimate for $T\rho$,
\begin{equation}\label{mbounds}
 m(x) \leq T\rho(x) \leq M(x),
\end{equation}
where we define
\begin{equation}\label{mbounddef}
m(x):= \left(\overline{\delta} + A \frac{|x|^k}{k} + B|x|^2\right)^{-N/k}, \qquad
M(x):= \left(\underline{\delta} +  B|x|^2\right)^{-N/k}\, .
\end{equation}

We are now ready to look for a fixed point of $T$. Applying $T$ to $\bar \mX$, we are able to make use of a variant of the Arz\'ela-Ascoli Theorem to obtain compactness. The key ingredients are the bounds in Lemma \ref{lem:kthmomentbound} and the uniform estimate \eqref{deltabound} since they allow us to derive the pointwise estimate \eqref{mbounds}, which gives decay at infinity and uniform boundedness of $T\rho$:
\begin{align}\label{decay+Tbound}
T \rho(x) \leq  \left(\underline{\delta} +B|x|^2\right)^{-N/k}
\leq \min\left\{B^{-N/k}\, |x|^{-2N/k},\, \underline{\delta}^{-N/k}\right\}\, .
\end{align}
Further, we claim $T\rho$ is $k$-H\"older continuous on compact balls $K_R:= \overline{B(0,R)} \subset \RR^N$, $R>0$,
\begin{align}
|T\rho(x_1)-T\rho(x_2)| \leq C_{R,N,k}\, |x_1-x_2|^k, \label{TH}
\end{align}
with $k$-H\"older semi-norm
\begin{align}
\left[ T\rho (\cdot)\right]_k =C_{R,N,k} = \frac{N}{k}\, \underline{\delta}^{-(1+N/k)}\, \left( \frac{A}{k}+3BR^{2-k}\right) >0\label{Hcoeff}\,.
\end{align}
To see this, let $G(x):= A\frac{|x|^k}{k} \ast \rho + B|x|^2 + C[\rho]$ and $u(G):=G^{-N/k}$ so that we can write 
\begin{align*}
 |T\rho(x_1)-T\rho(x_2)| 
 &= |G(x_1)^{-N/k} - G(x_2)^{-N/k}| \leq \mathrm{Lip}\left(u\right) |G(x_1)-G(x_2)|\\
 &\leq \mathrm{Lip}\left(u\right)  
 \left( A \,\left[ \frac{|\cdot|^k}{k} \ast \rho(\cdot)\right]_k + B\, \left[ |\cdot|^2\right]_k\right)
 \, |x_1-x_2|^k,
\end{align*}
where $\mathrm{Lip}(\cdot)$ denotes the Lipschitz constant on a suitable domain specified below. Indeed, $G(x)$ satisfies the inequality
$
0<\underline{\delta} \leq G(x) \leq A\tfrac{|x|^k}{k}+ B|x|^2+\overline{\delta}
$
for all $x \in \RR^N$ by \eqref{eq:kthmomentbound} and \eqref{deltabound}. Moreover, $G$ is $k$-H\"older continuous:
\begin{align*}\label{Hoelder1}
 \left| \frac{|x_1|^k}{k} \ast \rho(x_1) - \frac{|x_2|^k}{k} \ast \rho(x_2)\right|
 = &\dfrac1k \int_{\RR^N} \left| |x_1-y|^k-|x_2-y|^k\right| \rho(y) \, dy\notag\\
 \leq &\, \frac{|x_1-x_2|^k}{k} \, 2^{k-1} 
 \leq \, \frac{|x_1-x_2|^k}{k} 
\end{align*}
and hence $\left[ \frac{|\cdot|^k}{k} \ast \rho(\cdot)\right]_k\leq 1/k$ uniformly. Further, the $k$-H\"older semi-norm of $|x|^2$ is bounded by $3R^{2-k}$ on $K_R$: for $x,y \in K_R$, $x \neq y$ and $z :=x-y$, we have for $|z|\leq R$,
$$
\frac{\left| |x|^2-|y|^2\right|}{|x-y|^k}
\leq \frac{|z|^2+2|z|\min\{|x|,|y|\}}{|z|^k} 
\leq 3 R |z|^{1-k} 
\leq 3R^{2-k},
$$
and similarly for $|z|\geq R$,
$$
\frac{\left| |x|^2-|y|^2\right|}{|x-y|^k}
\leq \frac{2R^2}{R^k} = 2R^{2-k},
$$
and so $\left[ |\cdot|^2\right]_k \leq 3R^{2-k}$.
We are left to estimate the Lipschitz coefficient $\mathrm{Lip}\left(u\right)$ for $G \in [\underline{\delta}, \infty)$. Indeed, we can calculate it explicitly using the mean value theorem,
$$
|u(G_1)-u(G_2)| \leq \left(\max_{\xi \in [\underline{\delta}, \infty)} |u'(\xi)| \right) |G_1-G_2|,
$$
and so we have 
$$
\mathrm{Lip}\left( u \right) \leq \max_{\xi \in [\underline{\delta}, \infty)} |u'(\xi)| 
=\dfrac Nk \underline{\delta}^{-(1+N/k)}.
$$
This concludes the proof of H\"older continuity of $T\rho$ on $K_R$, \eqref{TH}-\eqref{Hcoeff}. Since \linebreak $\int_{\RR^N} |x|^k M(x) \, dx < \infty$ if $k \in (0,1]$, it follows from \eqref{mbounds} that $T\rho$ has bounded $k$th moment. Together with the estimate of the tails \eqref{decay+Tbound}, we have indeed $T\bar \mX\subset \bar \mX$, and so $T$ is well-defined.
We conclude that the operator $T: \bar \mX \to \bar \mX$ is compact by a variant of the Arz\'ela-Ascoli Theorem using uniform decay at infinity and uniform boundedness \eqref{decay+Tbound} together with equi-H\"older-continuity \eqref{TH}. Continuity of the map $T: \bar \mX \to \bar \mX$ can be analogously checked since the convolution with $W_k$ is a continuous map from $\bar \mX$ to $C(\RR^N)$ together with a similar argument as before for the H\"older continuity of $T \rho$. Additionally, we use that $C[\rho]$ is continuous in terms of $\rho$ as $M(\rho,C)$, the mass of $T\rho$, is a continuous function in terms of both $\rho$ and $C$ and strictly decreasing in terms of $C$, and hence $C[\rho]=M^{-1}(\rho,1)$ is continuous in terms of $\rho$. Here, $M^{-1}(\rho,\cdot)$ denotes the inverse of $M(\rho,\cdot)$.

Finally, by Schauder's fixed point theorem there exists $\bar \rho \in \bar \mX$ such that $T\bar \rho=\bar \rho$. In other words, $\bar \rho$ satisfies relation \eqref{eq:sstates resc kpos} on $\RR^N$. By continuity and radial monotonicity, we further have $\bar \rho \in L^\infty\left(\RR^N\right)$ from which we deduce the required regularity properties using $\supp(\bar\rho)=\RR^N$ and Lemma \ref{lem:regS}. 
We conclude that $\bar \rho$ is a stationary state of the rescaled equation according to Definition  \ref{def:sstates resc}.
\end{proof}

\begin{remark}
 Having established existence of radially symmetric stationary states to the rescaled equation \eqref{eq:KSresc}, it is a natural question to ask whether these stationary states correspond to minimisers of the rescaled free energy functional $\mFr$. For a stationary state $\bar \rho$ to have finite energy, we require in addition $\mV[\bar \rho]<\infty$, $\bar\rho^m \in L^1\left(\RR^N\right)$ and $|x|^k\bar\rho\in L^1\left(\RR^N\right)$, in which case $\bar \rho \in \mY_{2,k}$. As noted in Remark \ref{rmk:kregimes min}, 
 this is true if and only if $0<k < \tfrac{2N}{2+N}$. This restriction corresponds to $\tfrac{N}{2+N}<m<1$ and coincides with the regime of the fast diffusion equation ($\chi=0$) where the Barenblatt profile has second moment bounded and its $m$th power is integrable \cite{BDGV}.
\end{remark}
\begin{remark}
 In particular, the non-existence result in original variables Theorem \ref{thm:noEmins} means that there is no interaction strengths $\chi$ for which the energy functional $\mF_k$ admits radially symmetric non-increasing global minimisers. In this sense, there is no critical $\chi_c$ for $k>0$ as it is the case in the porous medium regime. Existence of global minimisers for the rescaled free energy functional $\mFr$ for all $\chi>0$ would provide a full proof of non-criticality in the fast diffusion range and is still an open problem for arbitrary dimensions $N$. We suspect that $\mFr$ is bounded below. In one dimension, one can establish equivalence between stationary states of the rescaled equation \eqref{eq:KSresc} and global minimisers of $\mFr$ by completely different methods, proving a type of reversed HLS inequality \cite{CCHCetraro}. The non-existence of a critical parameter $\chi$ is a very interesting phenomenon, which has already been observed in \cite{CL} for the one-dimensional limit case $k=1$, $m=0$.
\end{remark}

%%%%%%%%%%%%%%%%%%%%%%%%%%%%%%%%%%
%%%%%%%%%%%%%%%%%%%%%%%%%%%%%%%%%%
%%%%%%%%%%%%%%%%%%%%%%%%%%%%%%%%%%
\subsection{Numerical Simulations in One Dimension}
%%%%%%%%%%%%%%%%%%%%%%%%%%%%%%%%%%
%%%%%%%%%%%%%%%%%%%%%%%%%%%%%%%%%%
%%%%%%%%%%%%%%%%%%%%%%%%%%%%%%%%%%

To illustrate our analysis of the fast diffusion regime, we present numerical simulations in one dimension. We use a Jordan-Kinderlehrer-Otto (JKO) steepest descent scheme \cite{JKO,Otto} which was proposed in \cite{BCC08} for the logarithmic case $k=0$, and generalised to the porous-medium case $k\in(-1,0)$ in \cite{CG}. It corresponds to a standard implicit Euler method for the pseudoinverse of the cumulative distribution function, where the solution at each time step of the non-linear system of equations is obtained by an iterative Newton-Raphson procedure.
It can easily be extended to rescaled variables and works just in the same way in the fast diffusion regime $k\in(0,1)$. 

\begin{figure}[h!]
\begin{subfigure}{.49\textwidth}
  \centering
  \includegraphics[width=\linewidth]{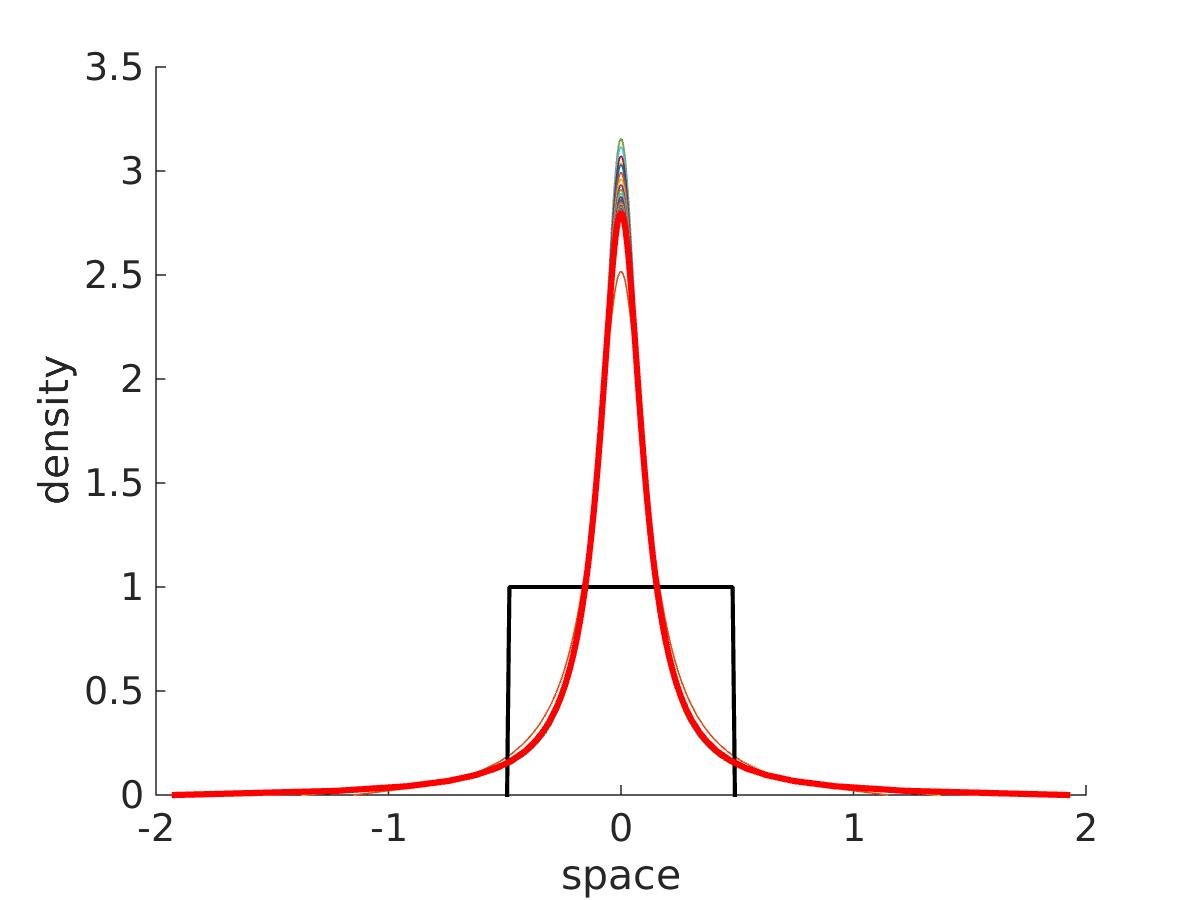}
  \caption{Density} 
  \label{fig:chi=12_k=02_resc=1_fig1a}
\end{subfigure}
\begin{subfigure}{.49\textwidth}
  \centering
  \includegraphics[width=\linewidth]{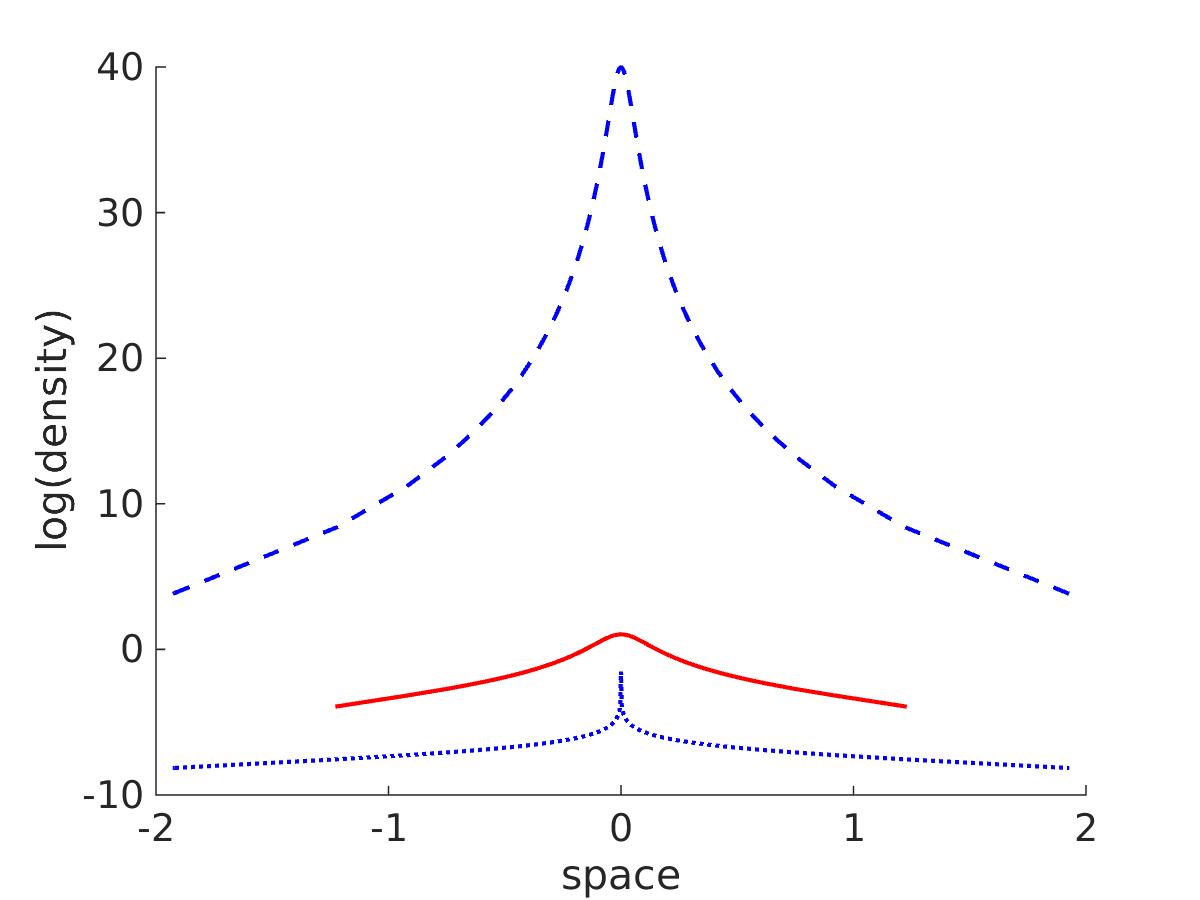}
  \caption{ $\log\left(m(x)\right) \leq \log\left(\bar\rho(x)\right) \leq \log\left( M(x)\right)$}
  \label{chi=12_k=02_resc=1_fig2b}
\end{subfigure}
\caption{Parameter choices: $\chi=1.2$, $k=0.2$. (A) Density distribution in rescaled variables: As initial data (black) we chose a characteristic supported on the centered ball of radius $1/2$, which can be seen to converge to the stationary state $\bar \rho$ (red); (B) Logplot of the density including bounds $m(x)$ (dotted blue) and $M(x)$ (dashed blue) as given in \eqref{mbounds}.}
\label{fig:chi=12_k=02_resc=1}
\end{figure}

Our simulations show that solutions in scaled variables for $k\in (0,1)$ converge always to a stationary state suggesting the existence of stationary states as discussed in the previous subsection.
Using the numerical scheme, we can do a quality check of the upper and lower bounds derived in \eqref{mbounds} for stationary states in $\bar \mX$:
\begin{equation*}
 m(x) \leq \bar\rho(x) \leq M(x)
\end{equation*}
with $m(x)$ and $M(x)$ given by \eqref{mbounddef}. Figures \ref{fig:chi=12_k=02_resc=1} and \ref{fig:chi=08_k=95_resc=1} show numerical results at two different points in the $(k,\chi)$-parameter space. For a more detailed description of the numerical scheme and a comprehensive list of numerical results, see \cite{CCHCetraro}.

%%%%%%%%%%%%%Used data from criticalchi=08 run 2
\begin{figure}[h!]
\begin{subfigure}{.49\textwidth}
  \centering
  \includegraphics[width=\linewidth]{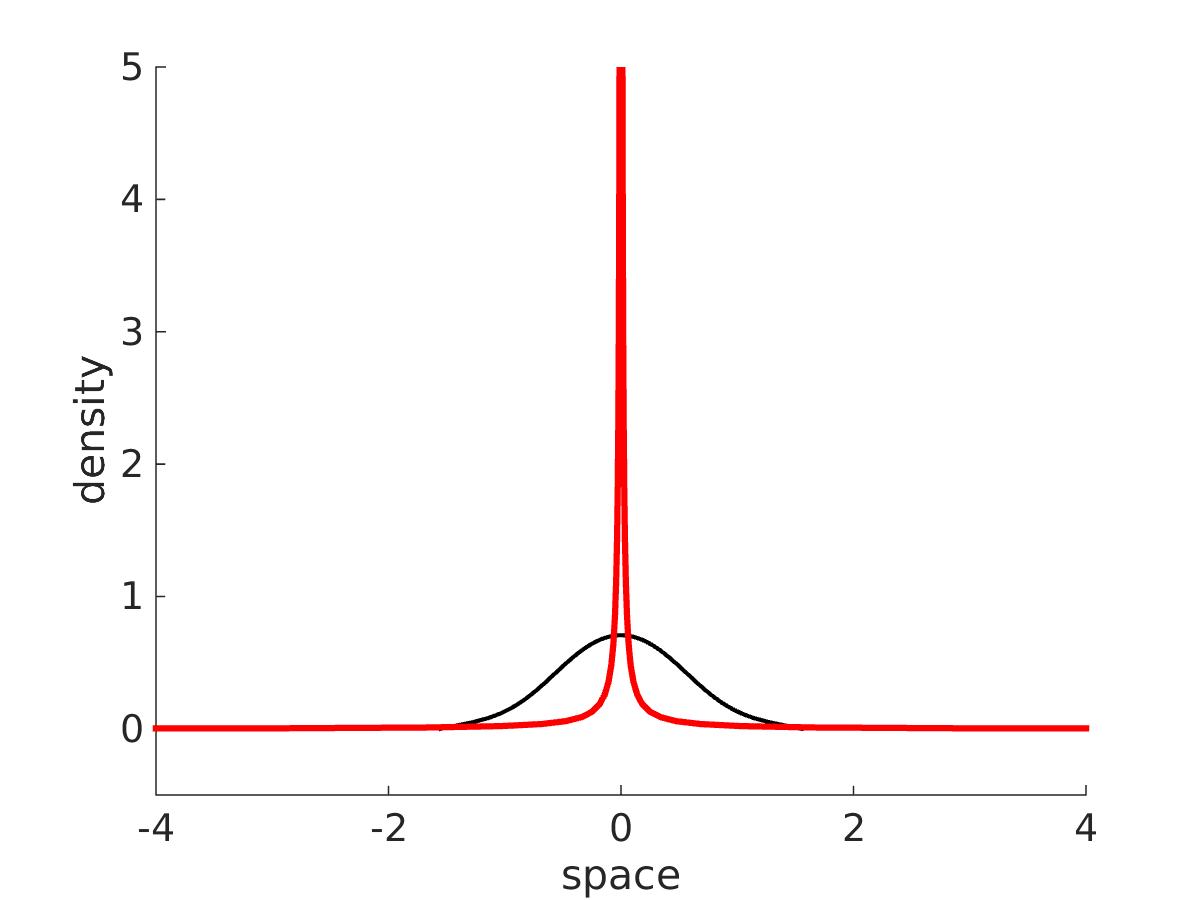}
  \caption{Density} 
  \label{fig:chi=08_k=95_resc=1_fig1}
\end{subfigure}
\begin{subfigure}{.49\textwidth}
  \centering
  \includegraphics[width=\linewidth]{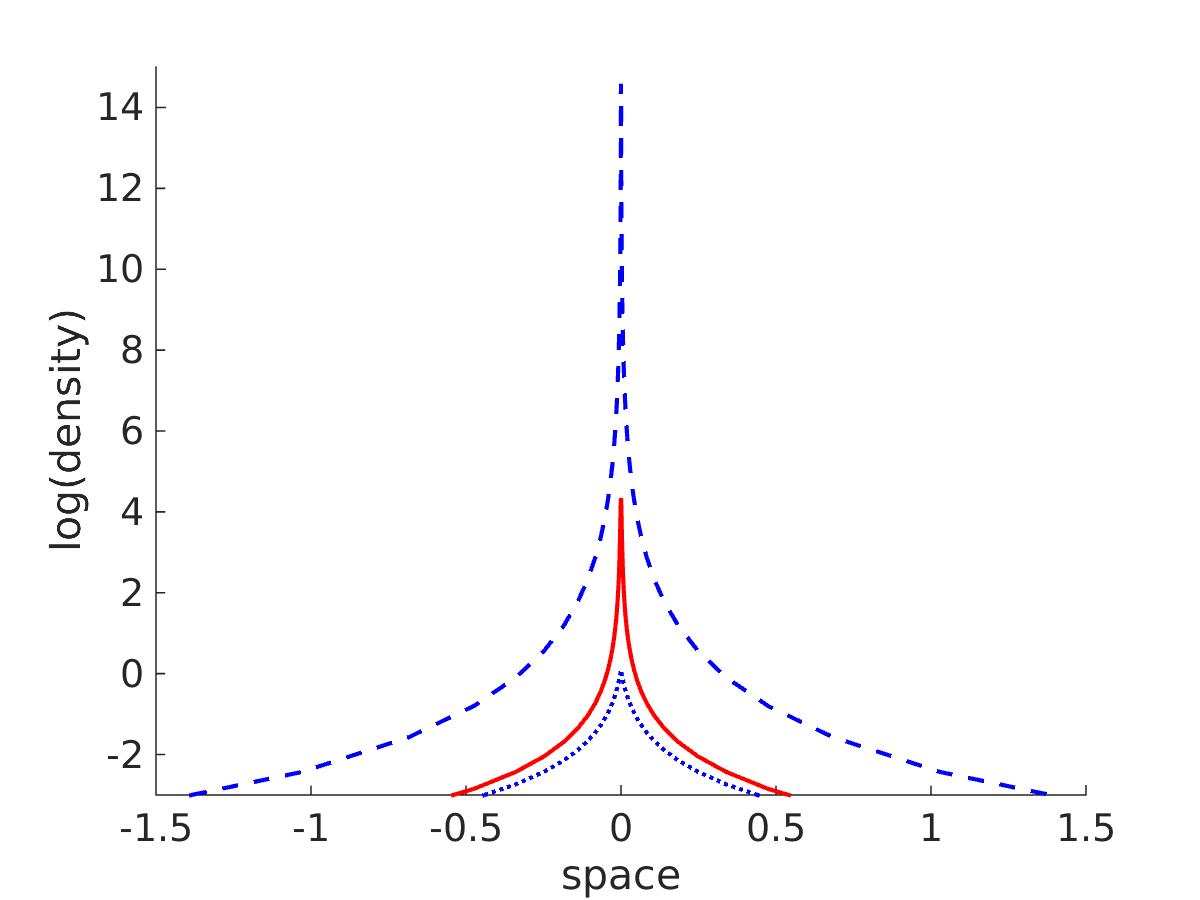}
  \caption{ $\log\left(m(x)\right) \leq \log\left(\bar\rho(x)\right) \leq \log\left( M(x)\right)$} 
  \label{chi=08_k=95_resc=1_fig2}
\end{subfigure}%
\caption{Parameter choices: $\chi=0.8$, $k=0.95$. (A) Density distribution in rescaled variables: As initial data (black) we chose a centered Gaussian distribution, which can be seen to converge to the stationary state $\bar \rho$ (red) - here, $\bar \rho$ is more peaked as $k$ is closer to $1$ and so we only display the lower part of the density plot ($\max_{x\in\RR} \bar \rho(x) = 75.7474$); (B) Logplot of the density including bounds $m(x)$ (dotted blue) and $M(x)$ (dashed blue) as given in \eqref{mbounds}.}
\label{fig:chi=08_k=95_resc=1}
\end{figure}

%%%%%%%%%%%%%%%%%%%%%%%%%%%%%%%%%%%%%%%%%%%%%%%%%%%%%%%%%%%%%%%%%%%%%%%%%%%%%%
%%%%%%%%%%%%%%%%%%%%%%%%%%%%%%%%%%%%%%%%%%%%%%%%%%%%%%%%%%%%%%%%%%%%%%%%%%%%%%
%%%%%%%%%%%%%%%%%%%%%%%%%%%%%%%%%%%%%%%%%%%%%%%%%%%%%%%%%%%%%%%%%%%%%%%%%%%%%%
%\newpage
\begin{appendix}
\section{Properties of $\psi_k$}
 
 We are here investigating in more detail the properties of the mean-field potential gradient for global minimisers in the porous medium regime. In more than one dimension, it can be expressed in terms of hypergeometric functions. Their properties are well understood and allow us to analyse the regularity properties of global minimisers. Since global minimisers of $\mF_k$ and $\mFr$ are radially symmetric by Proposition \ref{prop:charac min}, the aim is here to find the radial formulation of $\nabla S_k$ defined in \eqref{gradS}. In one dimension, explicit expressions are available, and so we are assuming from now on that $N\geq 2$. There are three different cases: (1) The Newtonian case $k=2-N$ with $N\geq 3$, (2) the range $1-N<k<0$, $k\neq 2-N$ where $\nabla(W_k \ast \rho )$ is well defined, and (3) the singular range $-N<k\leq 1-N$ where the force field is given by a Cauchy principle value.
 
 \begin{enumerate}
  \item  In the Newtonian case $k=2-N$, we have an explicit formula for the radial derivative of the force field using Newton's Shell Theorem,
 $$
 \partial_r\left(W_{2-N}\ast \rho\right)(r)
%  =\frac{\sigma_N M(r)}{\left|\partial B(0,r)\right|}
 = M(r) r^{1-N}\, ,
 $$
 where $M(r)=\sigma_N \int_0^r \rho(s)s^{N-1}\, ds$ is the mass of $\rho$ in a ball of radius $r$. Hence, we can write
  \begin{equation*}
  \partial_r(W_{2-N} \ast \rho)(r) = r^{1-N}\int_0^\infty\psi_{2-N}\left(\frac{\eta}{r}\right)  \rho(\eta)\eta^{N-1}  \, \rd \eta\, 
 \end{equation*}
 where $\psi_{2-N}$ is defined to have a jump singularity at $s=1$,
 \begin{equation}\label{Newtonianpsi}
  \psi_{2-N}(s):=
  \begin{cases}
   1\, &\text{if}\quad 0\leq s<1\, ,\\
   0\, &\text{if}\quad s>1\, .
  \end{cases}
 \end{equation}
 %%%%%%%%%%%%%%%%%%%%%%%%%%%%%%%%%%%%%%%%%%%%%%%%%%%%%%%%%%%%%%%
 \item In the range $1-N<k<0$ and $k \neq 2-N$, the mean-field potential gradient is given by
\begin{align*}
 \nabla S_k(x):=\nabla(W_k \ast \rho )(x) 
 &= \int_{\RR^N} \nabla W(x-y)\rho(y)\,\rd y\\
 &= \frac{1}{\sigma_N} \int_0^\infty \int_{\partial B(0,|y|)} \nabla W(x-y)\, \rd \sigma(y) \rho(|y|)\,\rd |y| \, .
 \end{align*}
 Denoting $|y|=\eta$, we can write for $x=re_1$,
\begin{align*}
 \frac{1}{\sigma_N}\int_{\partial B(0,|y|)} \nabla W(x-y)\, \rd \sigma(y)
 &= \frac{1}{\sigma_N}\int_{\partial B(0,|y|)} (x-y)|x-y|^{k-2}\, \rd \sigma(y)\\
 &=\left(\frac{1}{\sigma_N}\int_{\partial B(0,\eta)} e_1 \cdot (re_1-y)|r e_1 -y|^{k-2}\, \rd \sigma(y)\right)  \frac{x}{r}\\
 &=\eta^{N-1} \left(\frac{1}{\sigma_N}\int_{\partial B(0,1)}(r-\eta e_1.z)|r e_1 -\eta z|^{k-2}\, \rd \sigma(z)\right)  \frac{x}{r}\\
  &=\eta^{N-1} r^{k-1} \psi_k\left(\frac{\eta}{r}\right)  \frac{x}{r}\, ,
 \end{align*}
 where
 \begin{equation}\label{psidef}
  \psi_k\left(s\right)=
  \frac{1}{\sigma_N}\int_{\partial B(0,1)}(1-s e_1.z)|e_1 -s z|^{k-2}\, \rd \sigma(z)\, ,
  \qquad s \in [0,1)\cap (1,\infty)\, .
 \end{equation}
 By radial symmetry,
 \begin{align*}
  \nabla (W_k \ast \rho)(x) 
  &= r^{k-1} \left(\int_0^\infty\psi_k\left(\frac{\eta}{r}\right)  \rho(\eta)\eta^{N-1}  \, \rd \eta\right) \frac{x}{r}
  = \partial_r(W_k\ast \rho)(r) \frac{x}{r}\, 
 \end{align*}
 with 
 \begin{equation}\label{derivpsi0}
  \partial_r(W_k \ast \rho)(r) = r^{k-1}\int_0^\infty\psi_k\left(\frac{\eta}{r}\right)  \rho(\eta)\eta^{N-1}  \, \rd \eta\, .
 \end{equation}
%%%%%%%%%%%%%%%%%%%%%%%%%%%%%%%%%%%%%%%%
\item 
In the regime $-N<k\leq1-N$ however, the derivative of the convolution with the interaction kernel is a singular integral, and in this case the force field is define as
\begin{align*}
 \nabla S_k
:= &\int_\RR \frac{x-y}{|x-y|^{2-k}}\left(\rho(y)-\rho(x)\right)\, dy\\
= &\lim_{\delta \to 0} \int_{|x-y|>\delta} \frac{x-y}{|x-y|^{2-k}}\rho(y)\, dy
= \frac{x}{r} \partial_r S_k(r) 
\end{align*}
with the radial component given by
\begin{align*}
 \partial_r S_k(r) 
 =&r^{k-1}\int_0^\infty \psi_k\left(\frac{\eta}{r}\right)\left(\rho(\eta)-\rho(r)\right)\eta^{N-1}\, d\eta\\
 =&r^{k-1} \lim_{\delta \to 0}\int_{|r-\eta|>\delta}\psi_k\left(\frac{\eta}{r}\right)\rho(\eta)\eta^{N-1}\,d\eta\, ,
\end{align*}
and $\psi_k$ is given by \eqref{psidef} on $[0,1)\cap (1,\infty)$. 
 \end{enumerate}
 %%%%%%%%%%%%%%%%%%%%%%%%%%%%%%%%%%%%%%%%%%%%%%
 
For any $-N<k<0$ with $k\neq 2-N$, we can rewrite \eqref{psidef} as
\begin{equation}\label{psi1}
 \psi_k(s)=\frac{\sigma_{N-1}}{\sigma_N} \int_0^\pi \left(1-s\cos(\theta)\right) \sin^{N-2}(\theta) A(s,\theta)^{k-2}\,d \theta\, ,
 \qquad s \in [0,1)\cap (1,\infty)
\end{equation}
with
\begin{equation*}
 A(s,\theta)= \left(1+s^2-2s\cos(\theta)\right)^{1/2}\, .
\end{equation*}
%%%%%%%
%%%%%%%
It is useful to express $\psi_k$ in terms of Gauss Hypergeometric Functions. The hypergeometric function $F(a,b;c;z)$ is defined as the power series
\begin{equation}\label{powerseriesF}
 F(a,b;c;z)=\sum_{n=0}^\infty \frac{(a)_n(b)_n}{(c)_n} \, \frac{z^n}{n!}\,
\end{equation}
for $|z|<1$ and $a,b \in \C$, $c \in \C \backslash \{\Z^- \cup \{0\}\}$, see \cite{AAR}, where $(q)_n$ is the Pochhammer symbol defined for any $q>0$, $n\in\N$ by
$$
(q)_0=1, \qquad (q)_n=\frac{(n+q-1)!}{(q-1)!}\, .
$$
We will here make use of its well known integral representation \cite{AAR}
\begin{equation*}
 F(a,b;c;z)=\frac{\Gamma(c)}{\Gamma(b)\Gamma(c-b)}\int_0^1 t^{b-1}(1-t)^{c-b-1}(1-tz)^{-a}\,dt
\end{equation*}
for $c>b>0$, $a>0$ and $|z|<1$. Moreover, if $c-a-b>0$, then $F$ is well defined at $z=1$ and satisfies
$$
F(a,b;c;1)=\frac{\Gamma(c)\Gamma(c-a-b)}{\Gamma(c-a)\Gamma(c-b)}\, .
$$
Otherwise, we have the limiting case discussed in \cite{AAR}:
\begin{align}
 &\lim_{z\to 1^-} \frac{F(a,b;c;z)}{(1-z)^{c-a-b}} 
= \frac{\Gamma(c)\Gamma(a+b-c)}{\Gamma(a)\Gamma(b)}\, ,
\qquad \text{if}\quad c-a-b<0\, .
\label{lim1}
%  &\lim_{z\to 1^-} \frac{F(a,b;c;z)}{\log\left(\frac{1}{1-z}\right)} 
% = \frac{\Gamma(c)}{\Gamma(a)\Gamma(b)}\, ,
% \qquad \text{if}\quad c-a-b=0\, .\label{lim2}
\end{align}
Let us define
\begin{equation*}\label{defH}
 H(a,b;c;z):= \frac{\Gamma(b)\Gamma(c-b)}{\Gamma(c)} F(a,b;c;z)\, .
\end{equation*}
To express $\psi_k$ as a combination of hypergeometric functions, we write
\begin{align*}
 \psi_k(s)
 = &\frac{\sigma_{N-1}}{\sigma_N} \int_0^\pi \left(1-s\cos(\theta)\right) \left(1+s^2-2s\cos(\theta)\right)^{\frac{k-2}{2}}\sin^{N-2}(\theta)\,d \theta\\
 = &\frac{\sigma_{N-1}}{\sigma_N}\left(1+s\right)^{k-2} \int_0^\pi  \left(1-s\cos(\theta)\right)\left( 1-\frac{4s}{(1+s)^2} \cos^2\left(\frac{\theta}{2}\right)\right)^{\frac{k-2}{2}}\sin^{N-2}\left(\theta\right)\, d \theta\\
 = &\frac{\sigma_{N-1}}{\sigma_N}\left(1+s\right)^{k-2} \int_0^\pi \left( 1-\frac{4s}{(1+s)^2} \cos^2\left(\frac{\theta}{2}\right)\right)^{\frac{k-2}{2}}\sin^{N-2}\left(\theta\right)\, d \theta\\
 &-\frac{\sigma_{N-1}}{\sigma_N}\left(1+s\right)^{k-2} s\int_0^\pi \cos(\theta) \left( 1-\frac{4s}{(1+s)^2} \cos^2\left(\frac{\theta}{2}\right)\right)^{\frac{k-2}{2}}\sin^{N-2}\left(\theta\right)\, d \theta\\
 =: &f_1(s)-f_2(s)\, .
\end{align*}
Now, we use the change of variable $t=\cos^2\left(\theta/2\right)$ to get
\begin{align*}
 f_1(s)
 = &\frac{\sigma_{N-1}}{\sigma_N}\left(1+s\right)^{k-2} \int_0^\pi \left( 1-\frac{4s}{(1+s)^2} \cos^2\left(\frac{\theta}{2}\right)\right)^{\frac{k-2}{2}}\sin^{N-2}\left(\theta\right)\, d \theta\\
 = &\frac{\sigma_{N-1}}{\sigma_N}\left(1+s\right)^{k-2} 2^{N-2} \int_0^1 \left(1-\frac{4s}{(1+s)^2} t\right)^{\frac{k-2}{2}} t^{\frac{N-3}{2}} \left(1-t\right)^{\frac{N-3}{2}} \, dt \\
 =&\frac{\sigma_{N-1}}{\sigma_N}\left(1+s\right)^{k-2} 2^{N-2} H\left(a,b_1;c_1;z\right)
\end{align*}
with
$$
a:=1-\frac{k}{2},
\quad b_1:=\frac{N-1}{2},
\quad c_1:=N-1,
\quad z:=\frac{4s}{(1+s)^2}\,.
$$
Let us define $h_1(s):=f_1(s)$, and
\begin{equation*}
 h_2(s):= \frac{\sigma_{N-1}}{\sigma_N}\left(1+s\right)^{k-2} 2^{N-2} H\left(a,b_2;c_2;z\right)
\end{equation*}
with
$$
a:=1-\frac{k}{2},
\quad b_2:=\frac{N-1}{2},
\quad c_2:=N-1,
\quad z:=\frac{4s}{(1+s)^2}\,.
$$
Then
\begin{align*}
 f_2(s)
 =&\frac{\sigma_{N-1}}{\sigma_N}\left(1+s\right)^{k-2} s\int_0^\pi \cos(\theta) \left( 1-\frac{4s}{(1+s)^2} \cos^2\left(\frac{\theta}{2}\right)\right)^{\frac{k-2}{2}}\sin^{N-2}\left(\theta\right)\, d \theta\\
 =&-sh_1(s)+2sh_2(s)
\end{align*}
by the same change of variable. We conclude
\begin{equation}\label{psi2}
 \psi_k(s)=(1+s)h_1(s)-2sh_2(s)\,,  \qquad s \in [0,1) \cap (1,\infty)\, .
\end{equation}
%%%%%%%%%%%%%%%%%%%
\begin{figure}[h!]
\begin{subfigure}{.5\textwidth}
  \centering
  \includegraphics[width=\linewidth]{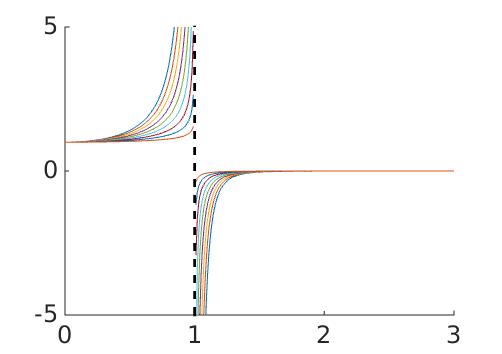}
  \caption{$-N<k<2-N$} 
  \label{fig:psik_fig1}
\end{subfigure}%
\begin{subfigure}{.5\textwidth}
  \centering
  \includegraphics[width=\linewidth]{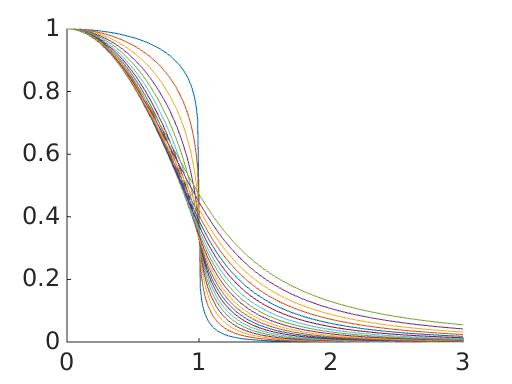}
  \caption{$2-N<k<0$} 
  \label{fig:psik_fig4}
\end{subfigure}
\caption{$\psi_k$ for different values of $k$ with $N=6$, increasing $k$ by $0.2$ for each plot.}
\label{fig:psik}
\end{figure}

Let us now study the behaviour of $\psi_k$ in more detail for $k\neq 2-N$. For any fixed $s \in [0,1) \cap (1,\infty)$,
$$
|\psi_k(s)|\leq \frac{1}{\sigma_N} \int_{\partial B(0,1)} |e_1-sx|^{k-1}\, d \sigma(x) <\infty
$$
and by the dominated convergence theorem, it is easy to see that $\psi_k$ is continuous on $s \in [0,1) \cap (1,\infty)$ for any $-N<k<2-N$ and $2-N<k<0$. A singularity occures at $s=1$ if $k<2-N$, however this singularity is integrable in the range $1-N<k<2-N$. 

In order to handle the expression of the mean-field potential gradient, it is important to understand the behaviour of $\psi_k$ at the limits of the integral $0$ and $\infty$ as well as at the singularity $s=1$.

%%%%%%%%%%%%%%%%%%%%%%%%%%%%%%%%%%%%%%%%%%%
\begin{lemma}[Behaviour at $0$]\label{Beh0}
For $\alpha >-1$, $-N<k<0$ and small $s>0$,
\begin{equation}\label{lowerlim1}
\psi_k(s)s^{\alpha}=s^{\alpha}+O\left(s^{\alpha+1}\right)\, .
\end{equation}
\end{lemma}
%%%%%%%%%%%%%%%%%%%%%%%%%%%%%%%%%%%%%%%%%%%
\begin{proof}
 Following the same argument as in \cite[Lemma 4.4]{Dong}, we obtain $\psi_k(0)=1$ for any $-N<k<0$, and so \eqref{lowerlim1} follows.
\end{proof}
%%%%%%%%%%%%%%%%%%%%%%%%%%%%%%%%%%%%%%%%%%%%%
Similarly, extending the argument in \cite[Lemma 4.4]{Dong} to  $-N<k<0$, we have
%%%%%%%%%%%%%%%%%%%%%%%%%%%%%%%%%%%%%%%%%%%
\begin{lemma}[Behaviour at $\infty$]\label{Behinfty}
For $-N<k<0$, 
\begin{equation}\label{lim0}
\lim_{s \to \infty} s^{2-k} \psi_k(s)=\frac{N+k-2}{N}\, .
\end{equation}
\end{lemma}
%%%%%%%%%%%%%%%%%%%%%%%%%%%%%%%%%%%%%%%%%%%
Further, it is obvious from \eqref{psi1} that $\psi_k(s)>0$ for $s \in (0,1)$. From \cite{Dong},
$$
\psi_k'(s)=\left(\frac{\sigma_{N-1}}{\sigma_N}\right) \frac{(k-2)(N+k-2)}{(N-1)} s \int_0^\pi\sin^N(\theta)A(s,\theta)^{k-4}\, d\theta\, ,
\qquad s \in [0,1) \cap (1,\infty)
$$
and hence $\psi_k$ is strictly decreasing for $k>2-N$ and strictly increasing for $k<2-N$. It then follows from \eqref{lim0} that in the super-Newtonian regime $k>2-N$, $\psi_k$ converges to zero as $s\to\infty$, is finite and continuous at $s=1$, and strictly positive on $[0,\infty)$ (Figure \ref{fig:psik_fig4}). In the sub-Newtonian regime $-N<k<2-N$ on the other hand, the monotonicity of $\psi_k$ and the fact that $\psi_k$ converges to 0 as $s \to \infty$ imply that 
$$
\lim_{s\to 1^-}\psi_k(s)=+\infty\, ,
\qquad
\lim_{s\to 1^+}\psi_k(s)=-\infty\, ,
$$
and so we conclude that $\psi_k<0$  on $(1,\infty)$ if $-N<k<2-N$ (Figure \ref{fig:psik_fig1}). We summarise these observations in the following lemma:

\begin{lemma}[Overall Behaviour]\label{lem:psibehaviour}
Let $\psi_k$ be as defined in \eqref{psidef}.
 \begin{enumerate}[(i)]
  \item If $2-N<k<0$, then $\psi_k$ is continuous, positive and strictly decreasing on $[0,\infty)$.
  \item If $-N<k<2-N$, then $\psi_k$ is continuous, positive and strictly increasing on $[0,1)$, and it is continuous, negative and strictly increasing on $(1,\infty)$. Further, it has a singularity at $s=1$ which is integrable for $1-N<k<2-N$.
 \end{enumerate}
\end{lemma}

Using the hypergeometric function representation of $\psi_k$, we can characterise its behaviour near the singularity.

%%%%%%%%%%%%%%%%%%%%%%%%%%%%%%%%%%%%%%%%%%%%
\begin{lemma}[Behaviour at $1$]\label{Beh1}
For $\alpha\in \RR$ and $\eps>0$ small, we have
\begin{enumerate}[(1)]
 \item in the super-Newtonian regime $2-N<k<0$ and for $s=1\pm\eps$:
  \begin{equation*}%\label{upperlim2}
  \psi_k(s)s^{\alpha}=\psi_k(1)+O\left(\eps\right)\, ,
\end{equation*}
\item in the sub-Newtonian regime $-N<k<2-N$ and
\begin{enumerate}[(i)]
 \item for $s=1-\eps$:
 \begin{align}\label{upperlim1-}
 \psi_k\left(s\right)s^\alpha
 =K_1 \eps^{N+k-2} + K_2\eps^{N+k-1}+O\left( \eps^{N+k}\right)\, ,
\end{align}
 \item for $s=1+\eps$:
 \begin{align}\label{upperlim1+}
 \psi_k\left(s\right)s^\alpha
 = -K_1 \eps^{N+k-2} + K_3\eps^{N+k-1}+O\left( \eps^{N+k}\right)\, ,
\end{align}
\end{enumerate}
where
\begin{gather}
 K_1=\left(\frac{\sigma_{N-1}}{\sigma_N}\right) \frac{\gamma}{2}>0\,,
 \qquad
 K_2[\alpha]=-\left(\frac{\sigma_{N-1}}{\sigma_N}\right) \left(\frac{B_1+\gamma(1-N+2\alpha)}{4}\right)\, ,\label{K1K2}\\
 K_3[\alpha]=-\left(\frac{\sigma_{N-1}}{\sigma_N}\right) \left(\frac{B_1+\gamma(2k+N-5+2\alpha)}{4}\right)\notag%\label{K3}
\end{gather}
and
\begin{equation}
 \gamma=\frac{\Gamma(c_2-b_2)\Gamma(a+b_2-c_2)}{\Gamma(a)}>0\,.\label{gamma}
\end{equation}
\end{enumerate}
\end{lemma}
%%%%%%%%%%%%%%%%%%%%%%%%%%%%%%%%%%%%%%%%%%
\begin{proof}
(1) follows directly from the fact that $\psi_k$ is continuous at $s=1$ \cite[Lemma 4.4]{Dong}. In order to prove (2), we make use of expression \eqref{psi2} for $\psi_k$ in terms of hypergeometric functions and known expansions around the point of singularity. Denoting $\delta:=\eps/|2-\eps|>0$, we have for any $\beta>0$,
\begin{equation}\label{deltabeta}
 \delta^\beta
 =\left(\frac{\eps}{2}\right)^\beta + \beta \left(\frac{\eps}{2}\right)^{\beta+1} +O\left(\eps^{\beta+2}\right)\, .
\end{equation}
From \eqref{psi2} we can write 
$$\psi_k(s) =(1+s)h_1(s)-2sh_2(s) 
 =(1+s)\left(h_1(s)-h_2(s)\right) +(1-s)h_2(s)\, , $$ 
and hence, denoting $z=1-\delta^2$, we obtain for $s=1-\eps$:
\begin{align}
  2^{2-N}\frac{\sigma_{N}}{\sigma_{N-1}} \psi_k(1-\eps)
 =\,
 &\left(2-\eps\right)^{k-1} 
 \left(H\left(a,b_1;c_1;z\right)-H\left(a,b_2;c_2;z\right)\right)\notag\\
 &+\eps \left(2-\eps\right)^{k-2} H\left(a,b_2;c_2;z\right)\, \notag\\
 =\, 
 &\left(2-\eps\right)^{k-1} \delta^{N+k-3}
 \left(\frac{H\left(a,b_1;c_1;z\right)}{(1-z)^{c_1-a-b_1}}-\frac{H\left(a,b_2;c_2;z\right)}{(1-z)^{c_2-a-b_2}}\right)\notag\\
 &+ \eps \left(2-\eps\right)^{k-2} \delta^{N+k-3}\left(\frac{H\left(a,b_2;c_2;z\right)}{(1-z)^{c_2-a-b_2}}\right)\,. \label{psi3a}
\end{align}
Similarly, above the singularity point at $s=1+\eps$, we obtain:
\begin{align}
 2^{2-N}\frac{\sigma_{N}}{\sigma_{N-1}}  \psi_k(1+\eps)
 =\, 
 &\left(2+\eps\right)^{k-1} \delta^{N+k-3}
 \left(\frac{H\left(a,b_1;c_1;z\right)}{(1-z)^{c_1-a-b_1}}-\frac{H\left(a,b_2;c_2;z\right)}{(1-z)^{c_2-a-b_2}}\right)\notag\\
 &-\eps \left(2+\eps\right)^{k-2} \delta^{N+k-3}\left(\frac{H\left(a,b_2;c_2;z\right)}{(1-z)^{c_2-a-b_2}}\right)\,. \label{psi3b}
\end{align}
Using the power series expression \eqref{powerseriesF} for hypergeometric functions, we can write
\begin{gather*}
 \left(\frac{H\left(a,b_1;c_1;z\right)}{(1-z)^{c_1-a-b_1}}-\frac{H\left(a,b_2;c_2;z\right)}{(1-z)^{c_2-a-b_2}}\right)
 = \sum_{n=0}^\infty A_n \frac{z^n}{n!} = \sum_{m=0}^\infty \frac{(-1)^m B_m}{m!}\, \delta^{2m}\, ,\\
 B_m:=\sum_{n=m}^\infty \frac{A_n}{(n-m)!}\, ,\\
 A_n:=\left(\frac{(c_1-a)_n(c_1-b_1)_n}{(c_1)_n} - b_1 \frac{(c_2-a)_n(c_2-b_2)_n}{(c_2)_n}\right)\frac{\Gamma(b_1)\Gamma(c_1-b_1)}{\Gamma(a)}\, .
\end{gather*}
In the singularity regime $-N<k<2-N$, we have
$$
c_1-a-b_1=c_2-a-b_2=\frac{N+k-3}{2}<0\, ,
$$
and so we can make use of \eqref{lim1} to show that the leading order term vanishes:
\begin{align*}
 B_0&=\lim_{\delta\to0} \sum_{m=0}^\infty \frac{(-1)^m B_m}{m!}\, \delta^{2m}
 =\lim_{z\to 1^-} \left(\frac{H\left(a,b_1;c_1;z\right)}{(1-z)^{c_1-a-b_1}}-\frac{H\left(a,b_2;c_2;z\right)}{(1-z)^{c_2-a-b_2}}\right)\\
 &=\frac{\Gamma(c_1-b_1)\Gamma(a+b_1-c_1)}{\Gamma(a)} - \frac{\Gamma(c_2-b_2)\Gamma(a+b_2-c_2)}{\Gamma(a)}
 =0\, .
\end{align*}
Hence
\begin{align*}
& \frac{H\left(a,b_1;c_1;z\right)}{(1-z)^{c_1-a-b_1}}-\frac{H\left(a,b_2;c_2;z\right)}{(1-z)^{c_2-a-b_2}}
= -B_1\delta^2+O(\delta^4)\, ,\\
&\frac{H\left(a,b_2;c_2;z\right)}{(1-z)^{c_2-a-b_2}}
=\frac{\Gamma(c_2-b_2)\Gamma(a+b_2-c_2)}{\Gamma(a)} +O\left(\delta^2\right):=\gamma+O\left(\delta^2\right)\, .
\end{align*}
Substituting these estimates and making use of \eqref{deltabeta}, \eqref{psi3a} becomes
\begin{align*}
  2^{2-N}\frac{\sigma_{N}}{\sigma_{N-1}} \psi_k(1-\eps)
 =\,
 &\eps \left(2-\eps\right)^{k-2} \left(\gamma \delta^{N+k-3}+O\left(\delta^{N+k-1}\right)\right)\,, \notag\\
 &+\left(2-\eps\right)^{k-1}
 \left( -B_1 \delta^{N+k-1}+O\left(\delta^{N+k+1}\right)\right)\notag\\
 %%%%%%%%%%%%%%%%%%%%%%%%%%%%%%%%%%%%%%%%%
 =\,
 &\eps \left[2^{k-2}-\eps(k-2)2^{k-3}+O\left(\eps^2\right)\right] \notag \\
 &\times\left[\gamma \left(\frac{\eps}{2}\right)^{N+k-3} + \gamma\left(N+k-3\right)\left(\frac{\eps}{2}\right)^{N+k-2} + O\left(\eps^{N+k-1}\right)\right]\notag\\
 &+\left[2^{k-1}+O\left(\eps\right)\right]
 \left[ -B_1 \left(\frac{\eps}{2}\right)^{N+k-1}+O\left(\eps^{N+k}\right)\right]\notag\\
 %%%%%%%%%%%%%%%%%%%%%%%%%%%%%%%%%%%%%%%%%
 =\,
 &\gamma 2^{-N+1}\eps^{N+k-2}
 + \gamma2^{-N}\left[\left(N+k-3\right)+(2-k)\right]\eps^{N+k-1}\notag\\
 &-B_1 2^{-N}\eps^{N+k-1}+O\left(\eps^{N+k}\right)\notag\\ 
 %%%%%%%%%%%%%%%%%%%%%%%%%%%%%%%%%%%%%%%%%
 =\,
 &\gamma 2^{-N+1}\eps^{N+k-2}
 + 2^{-N}\left[\gamma\left(N-1\right)-B_1\right]\eps^{N+k-1}
 +O\left(\eps^{N+k}\right)\, .
\end{align*}
Similarly, \eqref{psi3b} has expansion
\begin{align*}
  2^{2-N}\frac{\sigma_{N}}{\sigma_{N-1}} \psi_k(1+\eps)
 =\,
 &-\eps \left(2+\eps\right)^{k-2} \left(\gamma \delta^{N+k-3}+O\left(\delta^{N+k-1}\right)\right)\,, \notag\\
 &+\left(2+\eps\right)^{k-1}
 \left( -B_1 \delta^{N+k-1}+O\left(\delta^{N+k+1}\right)\right)\notag\\
 %%%%%%%%%%%%%%%%%%%%%%%%%%%%%%%%%%%%%%%%%
 =\,
 &-\eps \left[2^{k-2}+\eps(k-2)2^{k-3}+O\left(\eps^2\right)\right] \notag \\
 &\times\left[\gamma \left(\frac{\eps}{2}\right)^{N+k-3} + \gamma\left(N+k-3\right)\left(\frac{\eps}{2}\right)^{N+k-2} + O\left(\eps^{N+k-1}\right)\right]\notag\\
 &+\left[2^{k-1}+O\left(\eps\right)\right]
 \left[ -B_1 \left(\frac{\eps}{2}\right)^{N+k-1}+O\left(\eps^{N+k}\right)\right]\notag\\
 %%%%%%%%%%%%%%%%%%%%%%%%%%%%%%%%%%%%%%%%%
 =\,
 &-\gamma 2^{-N+1}\eps^{N+k-2}
 + \gamma2^{-N}\left[-\left(N+k-3\right)+(2-k)\right]\eps^{N+k-1}\notag\\
 &-B_1 2^{-N}\eps^{N+k-1}+O\left(\eps^{N+k}\right)\notag\\ 
 %%%%%%%%%%%%%%%%%%%%%%%%%%%%%%%%%%%%%%%%%
 =\,
 &-\gamma 2^{-N+1}\eps^{N+k-2}
 + 2^{-N}\left[\gamma\left(5-N-2k\right)-B_1\right]\eps^{N+k-1}
 +O\left(\eps^{N+k}\right)\, .
\end{align*}
We conclude
\begin{align*}
 \psi_k(1-\eps)=
 \left(\frac{\sigma_{N-1}}{\sigma_{N}}\right) \left(\frac{\gamma}{2}\right)\eps^{N+k-2}
 + \left(\frac{\sigma_{N-1}}{\sigma_{N}}\right) \left(\frac{\gamma\left(N-1\right)-B_1}{4}\right)\eps^{N+k-1}
 +O\left(\eps^{N+k}\right)\, ,
\end{align*}
\begin{align*}
 \psi_k(1+\eps)=
 -\left(\frac{\sigma_{N-1}}{\sigma_{N}}\right) \left(\frac{\gamma}{2}\right)\eps^{N+k-2}
 + \left(\frac{\sigma_{N-1}}{\sigma_{N}}\right) \left(\frac{\gamma\left(5-N-2k\right)-B_1}{4}\right)\eps^{N+k-1}
 +O\left(\eps^{N+k}\right)\, ,
\end{align*}
and so (2)(i)-(ii) directly follow.
\end{proof}

\end{appendix}

%%%%%%%%%%%%%%%%%%%%%%%%%%%%%%

%%%%%%%%%%%%%%%%%%%%%%%%%%%%%%%%%%%%%%%%%%%%%%%%%%%%%%%%%%%%%%%%%%%%
%
%
%    Acknowledgments
%
%
%%%%%%%%%%%%%%%%%%%%%%%%%%%%%%%%%%%%%%%%%%%%%%%%%%%%%%%%%%%%%%%%%%%%%%%%%%

\section*{Acknowledgments}
\small{JAC was partially supported by the Royal Society via a Wolfson Research Merit Award. FH acknowledges support from the EPSRC grant number EP/H023348/1 for
the Cambridge Centre for Analysis.}

%%%%%%%%%%%%%%%%%%%%%%%%%%%%%%

% \bibliographystyle{plain}
\bibliographystyle{abbrv}
\bibliography{biblio}

\begin{thebibliography}{10}

\bibitem{AmGiSa05}
L.~Ambrosio, N.~Gigli, and G.~Savar\'e.
\newblock {\em Gradient flows in metric spaces and in the space of probability
  measures}.
\newblock Lectures in Mathematics. Birkh\"auser, 2005.

\bibitem{AAR}
G.~E. Andrews, R.~Askey, and R.~Roy.
\newblock {\em Special functions}, volume~71 of {\em Encyclopedia of
  Mathematics and its Applications}.
\newblock Cambridge University Press, Cambridge, 1999.

\bibitem{BL}
S.~Bian and J.-G. Liu.
\newblock Dynamic and steady states for multi-dimensional {K}eller-{S}egel
  model with diffusion exponent {$m>0$}.
\newblock {\em Comm. Math. Phys.}, 323(3):1017--1070, 2013.

\bibitem{BCC08}
A.~Blanchet, V.~Calvez, and J.~A. Carrillo.
\newblock Convergence of the mass-transport steepest descent scheme for the
  subcritical {P}atlak-{K}eller-{S}egel model.
\newblock {\em SIAM J. Numer. Anal.}, 46(2):691--721, 2008.

\bibitem{BCC12}
A.~Blanchet, E.~A. Carlen, and J.~A. Carrillo.
\newblock Functional inequalities, thick tails and asymptotics for the critical
  mass {P}atlak-{K}eller-{S}egel model.
\newblock {\em J. Funct. Anal.}, 262(5):2142--2230, 2012.

\bibitem{BCL}
A.~Blanchet, J.~A. Carrillo, and P.~Lauren{\c{c}}ot.
\newblock Critical mass for a {P}atlak-{K}eller-{S}egel model with degenerate
  diffusion in higher dimensions.
\newblock {\em Calc. Var. Partial Differential Equations}, 35(2):133--168,
  2009.

\bibitem{BlaCaMa07}
A.~Blanchet, J.~A. Carrillo, and N.~Masmoudi.
\newblock Infinite time aggregation for the critical {P}atlak-{K}eller-{S}egel
  model in {$\Bbb R^2$}.
\newblock {\em Comm. Pure Appl. Math.}, 61(10):1449--1481, 2008.

\bibitem{BlaDoPe06}
A.~Blanchet, J.~Dolbeault, and B.~Perthame.
\newblock Two dimensional {Keller-Segel} model in {$\RR^2$}: optimal critical
  mass and qualitative properties of the solution.
\newblock {\em Electron. J. Differential Equations}, 2006(44):1--33
  (electronic), 2006.

\bibitem{BDGV}
M.~Bonforte, J.~Dolbeault, G.~Grillo, and J.~L. V{\'a}zquez.
\newblock Sharp rates of decay of solutions to the nonlinear fast diffusion
  equation via functional inequalities.
\newblock {\em Proc. Natl. Acad. Sci. USA}, 107(38):16459--16464, 2010.

\bibitem{CaCa06}
V.~Calvez and J.~Carrillo.
\newblock Volume effects in the {Keller-Segel} model: energy estimates
  preventing blow-up.
\newblock {\em J. Math. Pures Appl.}, 86:155--175, 2006.

\bibitem{CaCa11}
V.~Calvez and J.~A. Carrillo.
\newblock Refined asymptotics for the subcritical {K}eller-{S}egel system and
  related functional inequalities.
\newblock {\em Proc. Amer. Math. Soc.}, 140(10):3515--3530, 2012.

\bibitem{CCHCetraro}
V.~Calvez, J.~A. Carrillo, and F.~Hoffmann.
\newblock The geometry of diffusing and self-attracting particles in a
  one-dimensional fair competition regime.
\newblock {\em preprint arXiv:1603.07767}.

\bibitem{CG}
V.~Calvez and T.~O. Gallou{\"e}t.
\newblock Particle approximation of the one dimensional {K}eller-{S}egel
  equation, stability and rigidity of the blow-up.
\newblock {\em Discrete Contin. Dyn. Syst.}, 36(3):1175--1208, 2016.

\bibitem{CaLo92}
E.~Carlen and M.~Loss.
\newblock Competing symmetries, the logarithmic {HLS} inequality and {O}nofri's
  inequality on $\mathbb{S}\sp n$.
\newblock {\em Geom. Funct. Anal.}, 2:90--104, 1992.

\bibitem{CF}
E.~A. Carlen and A.~Figalli.
\newblock Stability for a {GNS} inequality and the log-{HLS} inequality, with
  application to the critical mass {K}eller-{S}egel equation.
\newblock {\em Duke Math. J.}, 162(3):579--625, 2013.

\bibitem{CaMcCVi03}
J.~Carrillo, R.~McCann, and C.~Villani.
\newblock Kinetic equilibration rates for granular media and related equations:
  entropy dissipation and mass transportation estimates.
\newblock {\em Rev. Mat. Iberoamericana}, 19:1--48, 2003.

\bibitem{CaMcCVi06}
J.~Carrillo, R.~McCann, and C.~Villani.
\newblock Contractions in the $2$-{W}asserstein length space and thermalization
  of granular media.
\newblock {\em Arch. Ration. Mech. Anal.}, 179:217--263, 2006.

\bibitem{CCV}
J.~A. Carrillo, D.~Castorina, and B.~Volzone.
\newblock Ground states for diffusion dominated free energies with logarithmic
  interaction.
\newblock {\em SIAM J. Math. Anal.}, 47(1):1--25, 2015.

\bibitem{CHVY}
J.~A. Carrillo, S.~Hittmeir, B.~Volzone, and Y.~Yao.
\newblock Nonlinear aggregation-diffusion equations: Radial symmetry and long
  time asymptotics.
\newblock {\em preprint arXiv:1603.07767}.

\bibitem{CaTo00}
J.~A. Carrillo and G.~Toscani.
\newblock Asymptotic {$L^1$}-decay of solutions of the porous medium equation
  to self-similarity.
\newblock {\em Indiana Univ. Math. J.}, 49(1):113--142, 2000.

\bibitem{CLM}
P.-H. Chavanis, P.~Lauren{\c{c}}ot, and M.~Lemou.
\newblock Chapman-{E}nskog derivation of the generalized {S}moluchowski
  equation.
\newblock {\em Phys. A}, 341(1-4):145--164, 2004.

\bibitem{CM}
P.~H. Chavanis and R.~Mannella.
\newblock Self-gravitating {B}rownian particles in two dimensions: the case of
  {$N=2$} particles.
\newblock {\em Eur. Phys. J. B}, 78(2):139--165, 2010.

\bibitem{CLW}
L.~Chen, J.-G. Liu, and J.~Wang.
\newblock Multidimensional degenerate {K}eller-{S}egel system with critical
  diffusion exponent {$2n/(n+2)$}.
\newblock {\em SIAM J. Math. Anal.}, 44(2):1077--1102, 2012.

\bibitem{CW}
L.~Chen and J.~Wang.
\newblock Exact criterion for global existence and blow up to a degenerate
  {K}eller-{S}egel system.
\newblock {\em Doc. Math.}, 19:103--120, 2014.

\bibitem{ChiPe81}
S.~Childress and J.~Percus.
\newblock Nonlinear aspects of chemotaxis.
\newblock {\em Math. Biosciences}, 56:217--237, 1981.

\bibitem{CL}
T.~Cie{\'s}lak and P.~Lauren{\c{c}}ot.
\newblock Global existence vs. blowup in a one-dimensional
  {S}moluchowski-{P}oisson system.
\newblock In {\em Parabolic problems}, volume~80 of {\em Progr. Nonlinear
  Differential Equations Appl.}, pages 95--109. Birkh\"auser/Springer Basel AG,
  Basel, 2011.

\bibitem{DoPe04}
J.~Dolbeault and B.~Perthame.
\newblock Optimal critical mass in the two dimensional {Keller-Segel} model in
  $\mathbb{R}\sp 2$.
\newblock {\em C. R. Math. Acad. Sci. Paris}, 339:611--616, 2004.

\bibitem{Dong}
H.~Dong.
\newblock The aggregation equation with power-law kernels: ill-posedness, mass
  concentration and similarity solutions.
\newblock {\em Comm. Math. Phys.}, 304(3):649--664, 2011.

\bibitem{JaLu92}
W.~J\"{a}ger and S.~Luckhaus.
\newblock On explosions of solutions to a system of partial differential
  equations modelling chemotaxis.
\newblock {\em Trans. Amer. Math. Soc.}, 329:819--824, 1992.

\bibitem{JKO}
R.~Jordan, D.~Kinderlehrer, and F.~Otto.
\newblock The variational formulation of the {F}okker-{P}lanck equation.
\newblock {\em SIAM J. Math. Anal.}, 29(1):1--17, 1998.

\bibitem{KeSe70}
E.~Keller and L.~Segel.
\newblock Initiation of slime mold aggregation viewed as an instability.
\newblock {\em J. Theor. Biol.}, 26:399--415, 1970.

\bibitem{KeSe71a}
E.~Keller and L.~Segel.
\newblock Model for chemotaxis.
\newblock {\em J. Theor. Biol.}, 30:225--234, 1971.

\bibitem{kesavan}
S.~Kesavan.
\newblock {\em Topics in functional analysis and applications}.
\newblock John Wiley \& Sons, Inc., New York, 1989.

\bibitem{LieLo01}
E.~Lieb and M.~Loss.
\newblock {\em Analysis}, volume~14 of {\em Graduate Studies in Math.}
\newblock Amer. Math. Soc., Providence, RI, second edition, 2001.

\bibitem{LW}
J.-G. Liu and J.~Wang.
\newblock A note on {$L^\infty$}-bound and uniqueness to a degenerate
  {K}eller-{S}egel model.
\newblock {\em Acta Appl. Math.}, 142:173--188, 2016.

\bibitem{Nagai95}
T.~Nagai.
\newblock Blow-up of radially symmetric solutions to a chemotaxis system.
\newblock {\em Adv. Math. Sci. Appl.}, 5:581--601, 1995.

\bibitem{Nanjundiah73}
V.~Nanjundiah.
\newblock Chemotaxis, signal relaying and aggregation morphology.
\newblock {\em J. Theor. Biol.}, 42:63--105, 1973.

\bibitem{Otto}
F.~Otto.
\newblock The geometry of dissipative evolution equations: the porous medium
  equation.
\newblock {\em Comm. Partial Differential Equations}, 26(1-2):101--174, 2001.

\bibitem{Perthame06}
B.~Perthame.
\newblock {\em Transport equations in biology}.
\newblock Frontiers in mathematics. Birkh\"auser, 2006.

\bibitem{Rudin87}
W.~Rudin.
\newblock {\em Real and complex analysis}.
\newblock McGraw-Hill Book Co., New York, third edition, 1987.

\bibitem{Sickel}
W.~Sickel.
\newblock Composition operators acting on {S}obolev spaces of fractional
  order---a survey on sufficient and necessary conditions.
\newblock In {\em Function spaces, differential operators and nonlinear
  analysis ({P}aseky nad {J}izerou, 1995)}, pages 159--182. Prometheus, Prague,
  1996.

\bibitem{Silvestre}
L.~Silvestre.
\newblock Regularity of the obstacle problem for a fractional power of the
  {L}aplace operator.
\newblock {\em Comm. Pure Appl. Math.}, 60(1):67--112, 2007.

\bibitem{Stein}
E.~M. Stein.
\newblock {\em Singular integrals and differentiability properties of
  functions}.
\newblock Princeton Mathematical Series, No. 30. Princeton University Press,
  Princeton, N.J., 1970.

\bibitem{Strohmer2008}
G.~Str{\"o}hmer.
\newblock Stationary states and moving planes.
\newblock In {\em Parabolic and {N}avier-{S}tokes equations. {P}art 2},
  volume~81 of {\em Banach Center Publ.}, pages 501--513. Polish Acad. Sci.
  Inst. Math., Warsaw, 2008.

\bibitem{Sugi2}
Y.~Sugiyama.
\newblock Application of the best constant of the {S}obolev inequality to
  degenerate {K}eller-{S}egel models.
\newblock {\em Adv. Differential Equations}, 12(2):121--144, 2007.

\bibitem{Sugi1}
Y.~Sugiyama.
\newblock The global existence and asymptotic behavior of solutions to
  degenerate quasi-linear parabolic systems of chemotaxis.
\newblock {\em Differential Integral Equations}, 20(2):133--180, 2007.

\bibitem{VazquezFDE}
J.~L. V{\'a}zquez.
\newblock {\em Smoothing and decay estimates for nonlinear diffusion
  equations}, volume~33 of {\em Oxford Lecture Series in Mathematics and its
  Applications}.
\newblock Oxford University Press, Oxford, 2006.
\newblock Equations of porous medium type.

\bibitem{VazquezPME}
J.~L. V{\'a}zquez.
\newblock {\em The porous medium equation}.
\newblock Oxford Mathematical Monographs. The Clarendon Press, Oxford
  University Press, Oxford, 2007.
\newblock Mathematical theory.

\bibitem{Villani03}
C.~Villani.
\newblock {\em Topics in optimal transportation}, volume~58 of {\em Graduate
  Studies in Math.}
\newblock Amer. Math. Soc, Providence, 2003.

\bibitem{Yao}
Y.~Yao.
\newblock Asymptotic behavior for critical {P}atlak-{K}eller-{S}egel model and
  a repulsive-attractive aggregation equation.
\newblock {\em Ann. Inst. H. Poincar\'e Anal. Non Lin\'eaire}, 31(1):81--101,
  2014.

\end{thebibliography}

\end{document}